\documentclass[a4paper,10pt,reqno]{amsart}
\usepackage{a4wide,color,array}
\usepackage{cite}
\usepackage{hyperref}
\usepackage{amsmath,amssymb,amsthm,color}
\hypersetup{linktocpage,colorlinks,linkcolor={red},citecolor={blue}}
\usepackage[english]{babel}
\usepackage[T1]{fontenc}
\usepackage[utf8]{inputenc}
\usepackage{dsfont}
\usepackage{enumitem}
\usepackage[a4paper,left=2.6cm,right=2.6cm,top=3cm,bottom=3.5cm]{geometry}%
\usepackage{marginnote}%

\usepackage{float}
\usepackage{subcaption}
\usepackage{accents}
\usepackage{tcolorbox}
\usepackage{graphicx}

\usepackage{xspace}
\usepackage{bm}				
\usepackage{bbm,upgreek}		
\usepackage[e]{esvect}		
\usepackage{esint}			
\usepackage{etoolbox}			
\usepackage{calc}				
\usepackage{eso-pic}			
\usepackage{datetime}			

\usepackage{lmodern}
\numberwithin{equation}{section}

\usepackage{enumitem}
\setlist{nosep}
\setlist{noitemsep}

\newtheorem{theorem}{Theorem}%

\newtheorem{proposition}{Proposition}[section]
\newtheorem{lemma}[proposition]{Lemma}%
\newtheorem{property}[proposition]{Property}
\newtheorem{conjecture}{Conjecture}%
\newtheorem{remark}{Remark}[section]%

\newtheorem{assumption}[remark]{Assumption}

\numberwithin{equation}{section}%

\newcommand{\dR}{\mathbb{R}}%

\newcommand{\E}{\mathbb{E}}

\newcommand{\ve}{\varepsilon}

\newcommand{\dP}{\mathbb{P}}%
\newcommand{\Var}{\mathrm{Var}}%
\newcommand{\supp}{\mathrm{supp}}

\newcommand{\mc}{\mathcal}

\newcommand{\diag}{\mathrm{diag}}
\newcommand{\dd}{\mathrm{d}}

\newcommand{\tr}{\mathrm{Tr}}

\newcommand{\bM}{\mathsf{M}}

\title{Large deviations for the extreme eigenvalues\\ of a deformed GOE}

\author{Jeanne Boursier}
\author{Alice Guionnet}

\begin{document}

\begin{abstract}
We establish a large deviation principle for the smallest eigenvalue of a random matrix model composed of the sum of a GOE matrix and a diagonal matrix with an outlier. Our result generalizes and unifies previously studied cases.
\end{abstract}
\maketitle

\setcounter{tocdepth}{1}
\tableofcontents

\section{Introduction}

\subsection{Setting of the problem}

{In this paper, we study the asymptotic behavior of the smallest eigenvalue of 
a sequence of random matrices $\{X_{N}\}_{N\in \mathbb N}$ where $X_N$ is given by the sum of two $N \times N$ matrices: a matrix $G_{N}$ taken from the Gaussian Orthogonal Ensemble (GOE) of variance  $t$, and a deterministic matrix $B_{N}$ whose empirical eigenvalue distribution converges weakly towards a probability measure $\nu$. We also allow the deterministic matrix to have an outlier $\Lambda$ below the support of $\nu$.}

{In this setting, the typical behavior of the spectrum of $X_{N}=G_{N}+B_{N}$ is well known. The empirical measure of the eigenvalues of $X_{N}$ converges towards the free convolution $\sigma_{t}\boxplus\nu$ where $\sigma_{t}$ is the semicircular distribution with variance $t$ \cite{VDN92,BV93}. 
This probability measure was described in detail in \cite{biane}.} Moreover, it was proved in   \cite{CaDoFeFe} that the smallest eigenvalue $\lambda_{1}(X_{N})$ of $X_{N}$ converges towards a limit $\ell^{\Lambda}_{\nu,t}$ which is equal to the left boundary $\ell_{\nu,t}$ of the support of 
 $\sigma_{t}\boxplus\nu$ if and only if  $\Lambda$ is above some threshold,  a phenomenon called the Ben-Arous-Baik-P\'ech\'e (BBP) transition (see Theorem \ref{convth}).

Our aim is to study the large deviations for the law of  $\lambda_{1}(X_{N})$,  the smallest eigenvalue of $X_{N}$. Our study is motivated by a problem arising in the theory of spin glasses. It is well known in the physics literature that the equilibrium properties of the $p$-spin glass model are related to the topology of the energy landscape of a random function called the Thouless-Anderson-Palmer (TAP) free energy \cite{Thouless1977-THOSOS-3}. In the case of the pure Ising $p$-spin glass, the Hessian of the TAP free energy contains a projector term, which may produce a negative eigenvalue at low enough temperature \cite{plefka,AJBray,parisi2004supersymmetry}. By the Kac-Rice formula, computing the (annealed) number of local minima  therefore  requires studying the large deviations of  the smallest eigenvalue of a  random matrix of the form considered in this paper.

Several other works have studied large deviation principles (LDPs) for the extreme eigenvalues of GOE matrices. The article   \cite{maida2007large}, which examined a rank-one deformation of the GOE,  was extended in \cite{benaych2012large} to finite rank deformations, and in \cite{biroli2020large}, which proved an LDP for the joint distribution of extreme eigenvalues and their associated eigenvectors. 
The article \cite{HuGu1}  investigated the extreme eigenvalues 
of   Wigner matrices, namely 
self-adjoint random matrices with  entries independent above the diagonal. They showed that if the entries are 
sharp sub-Gaussian in the sense that their Laplace transform is bounded above by the Laplace transform of the Gaussian variable with the same variance, the large deviations are the same as for Gaussian entries. 
 This came as a surprise, as large deviations for the smallest eigenvalue of Wigner matrices are generally very sensitive to the distribution of the entries, so that even their speed depends on their tail.  For distributions with tails heavier than Gaussian, an LDP was established in \cite{augeri2016large}, showing that the speed and the
 rate function depend on the tail decay rate.  In contrast, the  large deviations when the entries follow  general sub-Gaussian distributions have the same speed as for  Gaussian variables, but their rate function is non-universal and in general is different from the Gaussian one below a certain threshold, a transition indicating eigenvector localization phenomena \cite{augeri2021large,cook2023full}.  Large deviations were also investigated for more structured matrices. They were first studied for the sum of two unitarily invariant models  \cite{guionnet2020largeb} and more recently 
for matrices  with a variance profile, highlighting the impact of the variance profile on the rate function \cite{Hu22,ducatez2024large}.
 In \cite{mckenna2021large}, an LDP was obtained for $G_N + B_N$ where $B_N$ is a diagonal matrix without outliers, namely when $\Lambda=\ell_{\nu}$.

 The strategy used in \cite{mckenna2021large}  is based on a tilt by spherical integrals introduced in \cite{HuGu1} and subsequently developed in diverse contexts in \cite{HuGu2,cook2023full,ducatez2024large,guionnet2020largeb,augeri2021large}. However, this strategy can only be used to study large deviations below the limiting outlier $\ell_{\nu,t}^{\Lambda}$. On the other hand, \cite{maida2007large} studied the case where $\nu=\delta_{0}$, namely when $B_{N}$ has only one non-trivial eigenvalue $\Lambda$. Large deviations {were then obtained} thanks to an exact estimation of the density due to the asymptotics of spherical integrals derived in \cite{GuMa}.  In this paper, we prove a large deviation principle for $\lambda_1(X_{N})$, unifying the results {of} \cite{maida2007large} (where $\nu=\delta_{0}$) and \cite{mckenna2021large} (where no outlier is permitted, namely $\Lambda=\ell_{\nu}$). Our approach is new and consists of  estimating 
 the density of the law of the smallest eigenvalue  by exhibiting a fixed point functional equation for the large deviation rate function.

\subsection{Main result}

The GOE of size $N\times N$ with variance parameter $t>0$ is the random symmetric matrix
$G_N=(G_{ij})_{1\le i,j\le N}$ defined as follows: $G_{ji}=G_{ij}$ for all $i,j$, and the family
$(G_{ij})_{1\le i\le j\le N}$ consists of independent centered Gaussian random variables such that
for every $i\le j$,
\begin{equation}\label{var:GOEt}
    \Var(G_{ij})=\frac{t}{N}\bigl(1+1_{i=j}\bigr).
\end{equation}
Equivalently, $G_N$ has density proportional to
$\exp\!\bigl(-\frac{N}{4t}\tr(G^2)\bigr)$ with respect to Lebesgue measure on the space $\mathcal S_N(\dR)$ of $N\times N$ symmetric matrices with real entries. 
It is well known, see e.g.,\ \cite{wigner}, that the empirical measure of the eigenvalues  of $G_{N}$ converges almost surely towards the semi-circle law $\sigma_{t}$ given by
$$\sigma_{t}(\dd x)=\frac{1_{[-2\sqrt{t},2\sqrt{t}]}}{2\pi t}\sqrt{4t-x^{2}} \dd x\,.$$
We make the following assumptions on the sequence of matrices $(B_N)_{N\in\mathbb N}$:

\begin{assumption}\label{assumption:B}
Let $B_N\in\mc S_N(\dR)$ and write $b_1^N\le \cdots\le b_N^N$ for its eigenvalues. Assume:
\begin{enumerate}
\item There exists $\nu\in\mc P(\dR)$ with $\inf\supp(\nu)>-\infty$ such that, letting
$\mu_{B_N}:=\frac1N\sum_{i=1}^N\delta_{b_i^N}$,
\begin{equation}\label{eq:a_imu}
\lim_{N\to\infty} d(\mu_{B_N},\nu)=0 ,
\end{equation}
where $d$ is the bounded-Lipschitz distance defined in~\eqref{def:distance} below.
\item Let $\ell_\nu:=\inf\supp(\nu)$. There exists $\Lambda\in\dR$ such that
\begin{equation}\label{eq:Lambdalim}
b_1^N\underset{N\to \infty}{\longrightarrow} \Lambda \le  \ell_\nu .
\end{equation}
\item The second smallest eigenvalue satisfies
\begin{equation}\label{conv2}
b_2^N\underset{N\to \infty}{\longrightarrow} \ell_\nu .
\end{equation}
\item $\sup_N |b_N^N|<\infty$.
\end{enumerate}
We denote by $\eta_{\nu}:[0,1]\to\mathbb R$ the (right-continuous) quantile function of $\nu$,
i.e.\ the generalized inverse of $F_\nu(x)=\nu((-\infty,x])$, so that
$\nu=(\eta_\nu)_\#\mathrm{Unif}[0,1]$.
\end{assumption}

Under this assumption, the empirical measure and the  smallest eigenvalue of $X_{N}=B_{N}+G_{N}$ converge almost surely. We now describe their limits more precisely. 
First, it was shown by Voiculescu \cite{VDN92,BV93}, see also \cite[Corollary 5.4.11]{AGZ}, that the empirical measure of the eigenvalues of $X_{N}$ converges almost surely towards the free convolution $\nu\boxplus\sigma_{t}$. Following  
\cite{biane}  we can define uniquely  this probability measure by giving a formula for its Stieltjes transform.
If we denote by $G_{\mu}$ the Stieltjes transform of a probability measure $\mu$ on $\mathbb R$ given, for $z$ outside the support of $\mu$, by 
\begin{equation*}
    G_\mu(z)= \int \frac{\dd\mu(x)}{z-x} \, , 
\end{equation*}
then, the Stieltjes transform $G_{\nu\boxplus\sigma_{t}}$ of $\sigma_{t}\boxplus\nu$ is defined  for every $z\in \mathbb C^{+}$ by
\begin{equation}\label{freeco}G_{\nu\boxplus\sigma_{t}}(z)= G_{\nu}(\omega_{\nu,t}(z))\,.\end{equation}
Here, $\omega_{\nu,t}$ is the subordination function defined as the functional inverse of
\begin{displaymath}
H_{\nu,t}:
\left|
  \begin{array}{rcl}
   \overline{\Omega_{\nu,t}} & \longrightarrow   &\mathbb{C}^+\cup\dR \\
    z & \longmapsto & z+tG_\nu(z) \\
  \end{array},
\right.
\end{displaymath}
with $
   \Omega_{\nu,t}:=\{u+i v\in \mathbb{C}^+:v>v_t(u)\}$
where for every $u\in \dR$, we set
\begin{equation}\label{eq:defvu}
    v_t(u):=\inf\Bigr\{v\geq 0:\int \frac{\dd\nu(\lambda)}{|\lambda-u|^2+v^2}\leq \frac{1}{t}\Bigr\}.
\end{equation}
We discuss free convolution and some of its properties in Appendix \ref{sec:freeconv}.

{We denote by $\ell_{\nu,t}$ (resp. $\ell_{\nu}$) the left boundary point of the support of $\nu\boxplus\sigma_{t}$ (resp. of $\nu$).
Note that  $H_{\nu,t}$ is increasing on $({-\infty}, \omega_{\nu,t}(\ell_{\nu,t}))$ and decreasing on $(\omega_{\nu,t}(\ell_{\nu,t}),\ell_{\nu})$, so that  the equation $H_{\nu,t}(x)=\lambda$ has at most two solutions smaller than $\ell_\nu$, with the smallest being $\omega_{\nu,t}(\lambda)$.}

{When $-\Lambda$ is large enough, it is well-known that the smallest eigenvalue $\lambda_1(X_{N})$ of $X_{N}=G_N+B_N$ escapes the bulk, becoming smaller than $\ell_{\nu,t}$.}
{Throughout the paper, we denote 
\begin{equation*}
    \ell_{\nu,t}^\Lambda:=\begin{cases}
        H_{\nu,t}(\Lambda) & \text{if $\Lambda\leq \omega_{\nu,t}(\ell_{\nu,t})$}\\
        \ell_{\nu,t}& \text{if $\Lambda\geq \omega_{\nu,t}(\ell_{\nu,t})$}.
    \end{cases}
\end{equation*}}
We can  state the following theorem:
{\begin{theorem}\label{convth}
The smallest eigenvalue of $X_{N}=G_{N}+B_{N}$ converges  almost surely towards $\ell_{\nu,t}^\Lambda$ when $N$ goes to infinity. Moreover the map $\Lambda\in(-\infty,\ell_{\nu})\mapsto \ell_{\nu,t}^{\Lambda}$ is continuous, and for every $M>0$ the map $\nu\mapsto \ell_{\nu,t}^{\Lambda}$ is continuous on $\mathcal P([-M,M])$ endowed with the bounded-Lipschitz distance~\eqref{def:distance} (equivalently, the weak topology restricted to measures with a common compact support; see Lemma~\ref{smoothness}).
\end{theorem}

The convergence of the smallest eigenvalue was studied  in \cite[Theorem 5.1]{CaDoFeFe} in the case where  $G_{N}$ is a GUE matrix rather than a GOE matrix. They show  that the spectrum of $G_{N}+B_{N}$ concentrates 
in a neighborhood of the support of $\sigma_{t}\boxplus\mu_{B_{N}}$, from which outliers can be characterized as the image by $H_{\nu,t}$ of the outliers of $B_{N}$ by \eqref{freeco}. This extends to the case where $G_{N}$ follows the GOE, see e.g \cite[Theorem 2.5]{BaErSc}. 
The continuity statements are proven in  Lemma \ref{smoothness}.

{Let $\Lambda\leq \ell_{\nu}$ and 
 $\lambda\leq \ell_{\nu,t}$. Throughout the paper, when   $\lambda\in (H_{\nu,t}(\ell_\nu),\ell_{\nu,t}]$, we denote $\omega_{\nu,t}^*(\lambda)$ the unique solution in $[\omega_{\nu,t}(\ell_{\nu,t}),\ell_{\nu})$ of the equation $H_{\nu,t}(\gamma)=\lambda$. Note that when $\lambda\ge  H_{\nu,t}(\Lambda)$ and $\Lambda\in[ \omega_{\nu,t}(\ell_{\nu,t}),\ell_{\nu})$, $\lambda  \in (H_{\nu,t}(\ell_\nu),\ell_{\nu,t}]$ since $H_{\nu,t}$ is decreasing on $[ \omega_{\nu,t}(\ell_{\nu,t}),\ell_{\nu})$ (see e.g Lemma  \ref{smoothness}).
 Define
\begin{equation}\label{eq:defg}
    \gamma_\Lambda(\lambda):=\begin{cases}
        \Lambda & \text{if $\Lambda\leq \omega_{\nu,t}(\ell_{\nu,t})$, or $\Lambda\geq \omega_{\nu,t}(\ell_{\nu,t})$ and $\lambda\leq H_{\nu,t}(\Lambda)$},\\
        \omega^*_{\nu,t}(\lambda) & \text{if $\lambda\ge   H_{\nu,t}(\Lambda)$ and $\Lambda\in[ \omega_{\nu,t}(\ell_{\nu,t}),\ell_{\nu})$}.
    \end{cases}
\end{equation}
The two cases agree on the overlap $\lambda=H_{\nu,t}(\Lambda)$ (where $\omega_{\nu,t}^{*}(\lambda)=\Lambda$) and at $\Lambda=\omega_{\nu,t}(\ell_{\nu,t})$, so $\gamma_{\Lambda}$ is well defined.
Notice that $\gamma_\Lambda(\lambda)=\min(\Lambda,\omega_{\nu,t}^*(\lambda))$ (see Figure \ref{figure:3cases}).

\begin{figure}
\centering
\newcommand{\figW}{0.46\textwidth}

\begin{subfigure}[b]{\figW}
	\centering
	\begin{tcolorbox}[colframe=black, colback=white, boxrule=0.45pt]
		\includegraphics[width=\linewidth]{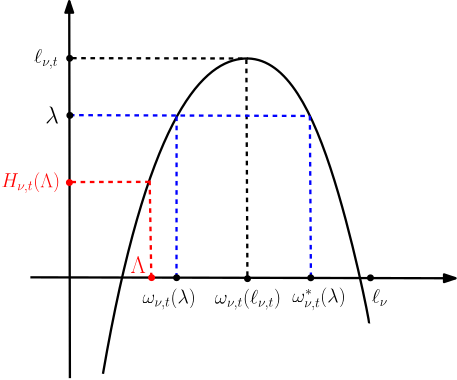}
	\end{tcolorbox}
	\subcaption{$\Lambda \leq \omega_{\nu,t}(\lambda)$}
\end{subfigure}
\hfill
\begin{subfigure}[b]{\figW}
	\centering
	\begin{tcolorbox}[colframe=black, colback=white, boxrule=0.45pt]
		\includegraphics[width=\linewidth]{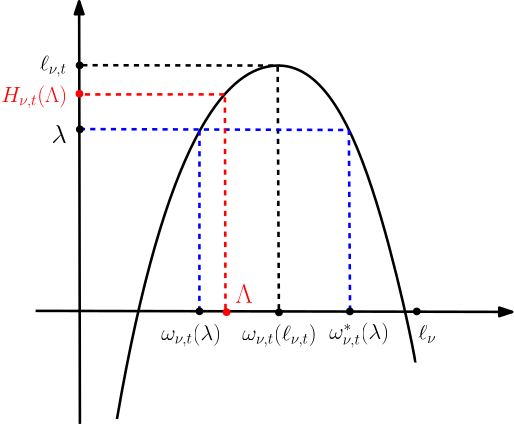}
	\end{tcolorbox}
	\subcaption{$\omega_{\nu,t}(\lambda)\leq\Lambda \leq\omega_{\nu,t}(\ell_{\nu,t}) $}
\end{subfigure}

\vspace{0.35cm}

\begin{subfigure}[b]{\figW}
	\centering
	\begin{tcolorbox}[colframe=black, colback=white, boxrule=0.45pt]
		\includegraphics[width=\linewidth]{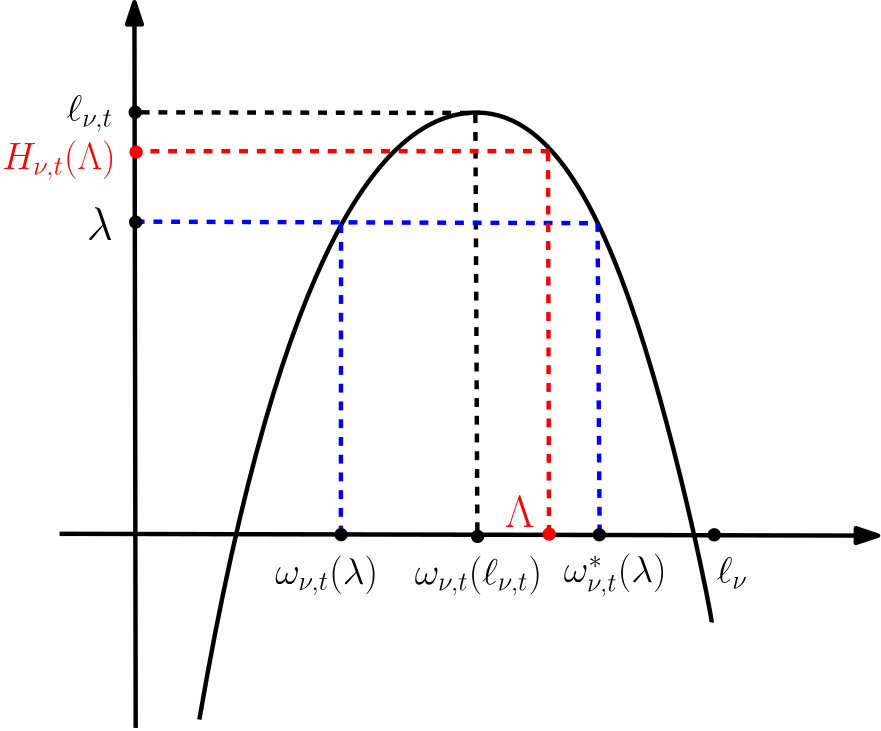}
	\end{tcolorbox}
	\subcaption{$ \omega_{\nu,t}(\ell_{\nu,t})\leq \Lambda\leq \omega_{\nu,t}^*(\lambda)$}
\end{subfigure}
\hfill
\begin{subfigure}[b]{\figW}
	\centering
	\begin{tcolorbox}[colframe=black, colback=white, boxrule=0.45pt]
		\includegraphics[width=\linewidth]{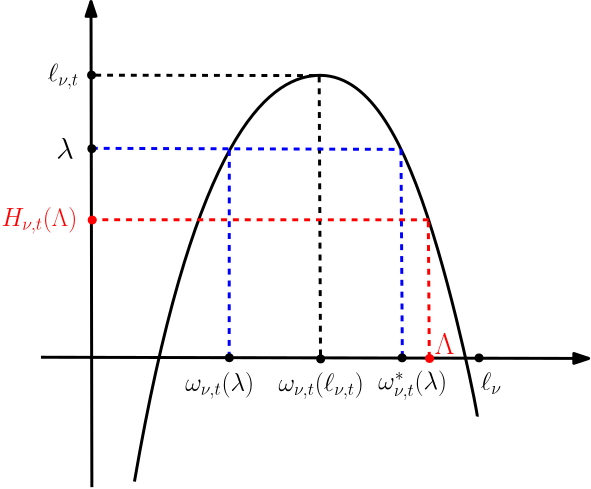}
	\end{tcolorbox}
	\subcaption{$\omega_{\nu,t}^*(\lambda)\leq \Lambda\leq \ell_\nu$}
\end{subfigure}

\caption{The graph of the function $H_{\nu,t}$}
\label{figure:3cases}
\end{figure}

 For every probability measure $\nu$ on $\dR$, we let $S_\nu$ be the logarithmic potential defined by
    \begin{equation}\label{def:Snu1}
        S_\nu:x\in \dR\setminus \supp(\nu)\mapsto -\int \log  |x-\lambda|\dd\nu(\lambda).
    \end{equation}
    Whenever $S_{\nu}$ (or $S_{\nu,t}$) is evaluated at a boundary point of $\supp(\nu)$ (for instance at $\ell_{\nu}$ or at a limiting value such as $\omega_{\nu,t}^{*}(x)$ approaching the edge), the expression is understood as the corresponding one-sided limit.
We set $ I_{\nu,t}^{\Lambda}$ to be infinite on $(\ell_{\nu,t},+\infty)$ and otherwise be given by 
\begin{equation}\label{eq:expl formula}
I_{\nu,t}^{\Lambda}(\lambda)
=
\frac{1}{2}
\left[
\left(
S_{\nu}(\omega_{\nu,t}(\lambda))
-\frac{(\lambda-\omega_{\nu,t}(\lambda))^{2}}{2t}
\right)
-
\left(
S_{\nu}(\gamma_{\Lambda}(\lambda))
-\frac{(\lambda-\gamma_{\Lambda}(\lambda))^{2}}{2t}
\right)
\right].
\end{equation}
Then, the main result of our paper is the following large deviation principle.

\begin{theorem}[LDP for the smallest eigenvalue,  GOE case.]\label{theorem:LDP}
Let $t>0$. Let $(B_N)_{N\in\mathbb N}$ be a sequence of matrices of size $N\times N$ satisfying Assumptions \ref{assumption:B}. Let $\nu$ and $\Lambda$ be as in \eqref{eq:a_imu} and \eqref{eq:Lambdalim}. Let $G_{N}$ be a sequence of GOE matrices of size $N\times N$ and variance $t$.

Then, the smallest eigenvalue of $G_N+B_N$ satisfies a large deviation principle at speed $N$ with good rate function $I_{\nu,t}^{\Lambda}$. Moreover $I_{\nu,t}^{\Lambda}$ is strictly convex on $(-\infty, \ell_{\nu,t})$ and lower semicontinuous up to $\ell_{\nu,t}$,
 and {attains its minimum value at $\ell_{\nu,t}^\Lambda$.}
\end{theorem}
Further properties of the rate function are discussed in Lemma \ref{lemma:convexity}.

\begin{remark}
Assumption \ref{assumption:B} (4) should not be necessary. 
\end{remark}

\begin{remark}
Our result can be applied when $B_N$ has no outlier by taking $\Lambda=\ell_{\nu}$, thus recovering the result of \cite{mckenna2021large}, and in the case where $\nu=\delta_{0}$, thus recovering the result of \cite{maida2007large}. This is detailed in Section \ref{prior}.\ In the no-outlier case $\Lambda=\ell_{\nu}$, the quantities $G_{\nu}(\Lambda)$, $H_{\nu,t}(\Lambda)$ and all subsequent formulas are understood in the sense of the one-sided (left) limit $\Lambda\uparrow\ell_{\nu}$.%
\end{remark}

Our result extends readily to the case where $G_{N}$ is a GUE matrix, namely when  $G_{N}$ is Hermitian, with independent entries  above the diagonal which are  complex centered  Gaussian variables  with independent real and imaginary parts with variance $t/2N$, and real on the diagonal with variance $t/N$.

\begin{theorem}[LDP for the smallest eigenvalue, GUE case]\label{theorem:LDPGUE}
Let $t>0$. Let $(B_N)_{N\in\mathbb N}$ be a sequence of matrices of size $N\times N$ satisfying Assumptions \ref{assumption:B}. Let $\nu$ and $\Lambda$ be as in  \eqref{eq:a_imu} and \eqref{eq:Lambdalim}. Let $G_{N}$ be a sequence of GUE matrices of size $N\times N$ and variance $t$.
Then, the smallest eigenvalue of $G_N+B_N$ satisfies a large deviation principle at speed $2N$ with good rate function $I_{\nu,t}^{\Lambda}$.\end{theorem}
We will discuss in Section \ref{sec:GUE} this generalization. 
 Another natural generalization concerns the case where $B_{N}$ has several outliers,  we then make the following conjecture
\begin{conjecture} Assume that $B_{N}$ satisfies \eqref{eq:a_imu} and  \eqref{eq:Lambdalim}, but instead of \eqref{conv2}, the $k$ smallest eigenvalues converge towards
 $\Lambda_{1}\le  \Lambda_{2}\le \cdots\le \Lambda_{k}\le \ell_{\nu}$ whereas the $(k+1)$th converges towards $\ell_{\nu}$. Let $G_{N}$ be a GOE matrix. We then conjecture that the $k$ smallest eigenvalues $\lambda_{1}\le\lambda_{2}\le \cdots \le \lambda_{k}$ of $G_{N}+B_{N}$ satisfy  a large deviation principle with  speed $N$ and good rate function  $\sum_{i=1}^{k} I^{\Lambda_{i}}_{\nu,t}(\lambda_{i})$.
 \end{conjecture}
 In the case where the other eigenvalues of $B_{N}$ vanish, this conjecture agrees with \cite[Proposition 2.7]{HuGu2}. Moreover, it agrees with the convergence of the outliers of $B_{N}+G_{N}$ towards $(\ell^{\Lambda_{i}}_{\nu,t})_{1\le i\le k}$. Note that by the contraction principle, the resulting rate function for $\lambda_{1}$ would then be given by
 $$I^{\Lambda_{1},\ldots,\Lambda_{k}}_{\nu,t}(x)=\inf\left\{ \sum_{i=1}^{k} I^{\Lambda_{i}}_{\nu,t}(\lambda_{i}), \lambda_{1}=x\le\lambda_{2}\le\cdots \le\lambda_{k}\right\}$$
 which is equal to $I^{\Lambda_{1}}_{\nu,t}(x)$ if $x\le \ell^{\Lambda_{2}}_{\nu,t}$ but to $\sum_{j\le i} I^{\Lambda_{j}}_{\nu,t}(x)$ if $x\in [\ell^{\Lambda_{i}}_{\nu,t},\ell^{\Lambda_{i+1}}_{\nu,t}]$ by monotonicity of the rate function on this segment. We believe that the same approach as the one developed in this article should apply,  but it  is not straightforward.

\subsection{Notations} 
For the reader's convenience, we collect in this section all notations used in this paper, some of which have been defined already.
\begin{itemize}
\item Let $S^{N-1}\subset \mathbb{R}^N$ denote the unit sphere, and let $\mathcal{S}_N(\mathbb{R})$ denote the space of real symmetric $N\times N$ matrices.
\item We denote by $\mc{P}(\mathbb R)$ the set of probability measures on $\dR$. We denote by $\eta_{\nu}:[0,1]\to\mathbb R$ the (right-continuous) quantile function of $\nu$,
i.e.\ the generalized inverse of $F_\nu(x)=\nu((-\infty,x])$.
    \item For any $\nu\in \mc{P}(\mathbb R)$, we denote by $\mathrm{supp}(\nu)$ the support of $\nu$ and $\ell_{\nu}$  the infimum of $\mathrm{supp}(\nu)$. 
    For  $t>0$ we denote $\ell_{\nu,t}:=\ell_{\nu\boxplus \sigma_t}$. 
    \item For any $\nu\in \mc{P}(\mathbb R)$,  we let $S_\nu$ be the logarithmic potential defined by
    \begin{equation*}
        S_\nu:x\in \dR\setminus \supp(\nu)\mapsto -\int\log  |x-\lambda|\dd\nu(\lambda).
    \end{equation*}
We let $\mathbb{C}^+=\{z=x+iy:y>0 \}$. Define for $\nu\in\mc{P}(\mathbb R)$ with support contained in $[a,b]$, where $a=\inf\supp(\nu)$ and $b=\sup\supp(\nu)$,
\begin{equation}\label{def:Gnu}
    G_\nu:\mathbb{C}^+ \cup \dR\setminus [a,b]:z\mapsto \int \frac{\dd\nu(x)}{z-x}.
\end{equation}
We then let $G_\nu(a)=\lim_{x\uparrow a}G_\nu(x)$ and $G_\nu(b)=\lim_{x\downarrow b}G_\nu(x)$, understood as one-sided limits in $\dR\cup\{-\infty,+\infty\}$ (they may be infinite). One can show that $G_\nu|_{\dR}$ is a bijection from $\dR\setminus [a,b]$ to $(G_\nu(a),G_\nu(b))\setminus \{0\}$. We will write
\begin{equation}\label{def:Knu}
    K_\nu :(G_\nu(a),G_\nu(b))\setminus \{0\}\to \dR\setminus [a,b]
\end{equation}
its inverse.
\item For every $\nu\in \mc{P}(\dR)$, we let\begin{equation}\label{def:rhonut}
H_{\nu,t}(\Lambda)=\Lambda+tG_\nu(\Lambda)
\end{equation}
for $\Lambda\in\dR\setminus\supp(\nu)$, and $H_{\nu,t}(\ell_{\nu}):=\ell_{\nu}+tG_{\nu}(\ell_{\nu})\in\dR\cup\{-\infty\}$ by the one-sided limit convention introduced for~\eqref{def:Gnu}.
and 
\begin{equation}\label{def:ellnut}
    \ell_{\nu,t}^\Lambda=\begin{cases}
        H_{\nu,t}(\Lambda) & \text{if $\Lambda\leq \omega_{\nu,t}(\ell_{\nu,t})$}\\
        \ell_{\nu,t}& \text{if $\Lambda\geq \omega_{\nu,t}(\ell_{\nu,t})$}.
    \end{cases}
\end{equation}
\item {For every $\nu\in \mc{P}(\dR)$ and $\lambda\in [H_{\nu,t}(\ell_\nu),\ell_{\nu,t}]$, we denote $\omega_{\nu,t}^*(\lambda)$ the unique solution $\gamma$ of
\begin{equation}\label{def:omega*}
    H_{\nu,t}(\gamma)=\lambda,\quad \gamma\in [\omega_{\nu,t}(\ell_{\nu,t}),\ell_{\nu}].
\end{equation}}
Here $H_{\nu,t}(\ell_{\nu})\in\dR\cup\{-\infty\}$ is understood via the one-sided limit convention of~\eqref{def:Gnu}; when $H_{\nu,t}(\ell_{\nu})=-\infty$, $\omega_{\nu,t}^{*}$ is defined on $(-\infty,\ell_{\nu,t}]$. The endpoint values $\omega_{\nu,t}^{*}(\ell_{\nu,t})=\omega_{\nu,t}(\ell_{\nu,t})$ and $\omega_{\nu,t}^{*}(H_{\nu,t}(\ell_{\nu}))=\ell_{\nu}$ are defined by continuity.

\item We denote $(e_1,\ldots, e_N)$ the canonical basis of $\dR^N$.

\item We let $d$ be the bounded Lipschitz distance on probability measures on $\dR$ defined for all $ \mu,\nu\in \mc{P}(\dR)$ by
\begin{equation}\label{def:distance}
    d(\mu,\nu)=\sup\left| \int f \dd(\mu -\nu)\right|,
\end{equation}
where the supremum is taken over measurable functions $f:\dR\to\dR$ satisfying $\|f\|_{\infty}+\|f\|_{L}\le 1$, with
$$\Vert f\Vert_{L}:= \sup_{x\neq y}\left|\frac{f(x)-f(y)}{x-y}\right|,\qquad \|f\|_{\infty}=\sup_{x\in \mathbb R}|f(x)|\,.$$
\item For every $p\geq 1$, we denote $\Delta^{p}$ the $p$-simplex
\begin{equation}\label{def:Deltap}
    \Delta^{p}:=\{(y_1,\ldots,y_{p})\in [0,1]^{p}:y_1+\cdots+y_{p}=1\}.
\end{equation}
\item $\lambda^{N}=(\lambda_{1},\ldots,\lambda_{N})$ denotes an $N$-dimensional vector in $\mathbb R^{N}$ and $\dd\lambda^{N}=\dd\lambda_{1}\cdots \dd\lambda_{N}$.
\item $O(A)$ (resp. $o(A)$) designates a function so that there exists a universal constant $C$ such that $|O(A)|\le C|A|$ (resp. $o(A)/|A|$ goes to zero as $A$ goes to zero or infinity (according to the context)). $o_{A}(1)$ is a quantity going to zero when $A$ goes to infinity or zero, depending on the context.
\item If $u$ is a vector in $\mathbb R^{k}$ or $\mathbb C^{k}$, $\|u\|=\sqrt{\sum |u_{i}|^{2}}$,  and $\|u\|_{\infty}=\max|u_{i}|$. If $E$ is a vector subspace of $\mathbb R^{k}$ or $\mathbb C^{k}$, we denote by $\|v\|_{E}$ the $\ell_{2}$ norm of the orthogonal projection of $v$ on $E$. 
\end{itemize}

\subsection{{Sketch of the proof}}\label{sub:proof ideas}
We  assume in these {heuristics} that the matrix $B_N$ has eigenvalues $\Lambda := \eta_0 < \eta_1 < \cdots < \eta_p$ with multiplicities $1$, $N_1, \ldots, N_p$, such that {for every} $i = 1, \ldots, p$, $N_i / N$ converges to $\alpha_i > 0$. 
We outline the arguments to derive the weak large deviation principle, namely {to} estimate the probability that the smallest eigenvalue $\lambda_{1}(X_{N})$ of $X_{N}=G_{N}+B_{N}$ is close to a given ${x \leq \ell_{\nu, t}}$. 
Let $\varepsilon > 0$.
The probability that  $\lambda_{1}(X_{N})$  lies in $(x - \varepsilon, x + \varepsilon)$ is given by
\begin{equation*}
 \mathbb P( \lambda_1(X_N) \in (x - \varepsilon, x + \varepsilon))=\frac{1}{Z_{N}^{t}}  \int_{\lambda_1(X_N) \in (x - \varepsilon, x + \varepsilon)} e^{-\frac{N}{4t} \operatorname{Tr}((X_N - B_N)^2)} \, \dd X_N,
\end{equation*}
where $\dd X_N$ is the Lebesgue measure on symmetric matrices of size $N \times N$ and $Z_{N}^{t}$ is the normalizing constant. Let us denote by $\lambda_1 \leq \cdots \leq \lambda_N$ the eigenvalues of $X_N$ and ${(v_1,\ldots, v_{N})}$ the associated orthonormal eigenvectors.  We show in Lemma \ref{lemma:projection}
that
one can split the term in the density into
\begin{equation*}
    \operatorname{Tr}((X_N - B_N)^2) = \lambda_1^2 - 2 \lambda_1 \langle v_1, B_N v_1 \rangle + 2 \langle v_1, B_N^2 v_1 \rangle - \langle v_1, B_N v_1 \rangle^2 + \operatorname{Tr}((p_{v_1}^\perp (X_N - B_N) p_{v_1}^\perp)^2),
\end{equation*}
where $p_{v_1}^\perp$ stands for the projection onto the ortho-complement of $v_1$.  For every $\lambda \in \{\Lambda=\eta_0,\eta_{1} \ldots, \eta_p\}$, let us denote $E_\lambda$ the associated eigenspace of $B_{N}$ and,   for a  vector $v\in \mathbb R^{N}$,  we write $\Vert v\Vert_{E_\lambda}$ the Euclidean  norm of the projection of $v$  on $E_{\lambda}$.
We then observe that $ \langle v_1, B_N v_1 \rangle$ and $ \langle v_1, B_N^{2} v_1 \rangle$  only {depend} on $Y_{N}(v_{1}):=(\Vert v_1 \Vert_{E_{\eta_0}}^2,\ldots,\Vert v_1 \Vert_{E_{\eta_p}}^2)$ and $\lambda_{1}$.  
In fact 
$$
 \frac{1}{4t} (- 2 \lambda_1 \langle v_1, B_N v_1 \rangle + 2 \langle v_{1}, B_N^{2} v_1 \rangle - \langle v_1, B_N v_1 \rangle^2)=:  L_{\nu,t}^{\Lambda}(\lambda_{1},Y_{N}(v_{1}))$$
 is a polynomial  $L_{\nu,t}^{\Lambda}$ in $\lambda_{1}$ and $Y_{N}(v_{1})$, see Lemma \ref{lemma:projection}.
Performing the change of variables $X_N \mapsto (\lambda_1, \ldots, \lambda_N, v_1, O)$, we find a constant $C_{N}^{t}$  such that
\begin{align}
 P_{N}(x,\varepsilon):= &  \int_{\lambda_1 \in (x - \varepsilon, x + \varepsilon)} e^{-\frac{N}{4t} \operatorname{Tr}((X_N - B_N)^2)} \, \dd X_N = 
 C_{N}^{t}\int_{\lambda_1 \in (x - \varepsilon, x + \varepsilon)} \dd \lambda_1\, \dd v_1\,  e^{-NL_{\nu,t}^{\Lambda}(\lambda_{1},Y_{N}(v_{1}))-N\frac{\lambda_{1}^{2}}{4t}}\label{heu1}\\
    &\qquad\qquad  \times
    \int    \prod_{j > 1} |\lambda_1 - \lambda_j|e^{-\frac{N}{4t} \operatorname{Tr}(( OD(\lambda)O^{T}-p_{v_1}^\perp  B_N p_{v_1}^\perp)^2)} \prod_{2 \leq i < j\leq N} |\lambda_i - \lambda_j| 1_{\lambda_1 \leq \ldots \leq \lambda_N} \, \dd\lambda_2 \cdots \, \dd\lambda_N \, \dd O,\nonumber
\end{align}
where $D(\lambda)=\mbox{diag}(\lambda_{2},\ldots,\lambda_{N})$, $O$ follows the Haar measure on $v_{1}^{\perp}$.
Notice that 
$OD(\lambda)O^{T}$ follows the law of the $(N-1)\times (N-1)$ GOE matrix $\sqrt{(N-1)/N}G_{N-1}$ (living in $v_{1}^{\perp}$) under the distribution 
\begin{equation}\label{def:prob}
\mathbb P_{N-1}(\dd\lambda,\dd O):=\frac{1}{Z_{N-1}^{t (N-1)/N}}  e^{-\frac{N}{4t} \operatorname{Tr}(( OD(\lambda)O^{T})^2)} \prod_{2 \leq i < j\leq N} |\lambda_i - \lambda_j| 1_{\lambda_2 \leq \cdots \leq \lambda_N} \, \dd\lambda_2 \cdots \, \dd\lambda_{N} \, \dd O\,.\end{equation}
Putting everything together,  if $\mathbb P_{p_{v_1}^\perp  B_N p_{v_1}^\perp}$ designates the law of $X_{N-1}:=p_{v_1}^\perp  B_N p_{v_1}^\perp+\sqrt{(N-1)/N}G_{N-1}$,  we find that $ P_{N}(x,\varepsilon)$ equals
\begin{multline}
  C_{N}^{t}Z_{N-1}^{t (N-1)/N}
  \int_{\lambda_1 \in (x - \varepsilon, x + \varepsilon)} \dd \lambda_1\, \dd v_1\,  e^{-NL_{\nu,t}^{\Lambda}(\lambda_{1},Y_{N}(v_{1}))-N\frac{\lambda_{1}^{2}}{4t}}
    \int_{\lambda_{2}\ge \lambda_{1}}    \prod_{j > 1} |\lambda_1 - \lambda_j| \dd \mathbb P_{p_{v_1}^\perp  B_N p_{v_1}^\perp}(\lambda_2,\ldots,\lambda_{N})\,.
\end{multline}
We observe that $X_{N-1}$ is the $(N-1)$-dimensional compression of $X_N=G_N+B_N$ onto $v_1^{\perp}$, with GOE variance parameter $t(N-1)/N$; in particular, its empirical eigenvalue measure converges towards $\nu\boxplus\sigma_{t}$ almost surely. Restricting ourselves to this set, we can approximate  $ \prod_{j > 1} |\lambda_1 - \lambda_j|$ by $\exp(N\int \log |x-y|\dd \nu\boxplus\sigma_{t})$ in the above integral.
Fix $Y=(y_{0},\ldots, y_{p}) \in \Delta^{p+1}$, where $\Delta^{p+1}$ is the $(p+1)$-simplex \eqref{def:Deltap}. 
Restricting the domain of integration to vectors $v_1 \in \mathbb{R}^N$ such that $Y_{N}(v_{1})$ is close to $Y$
and using the large deviation principle for $Y_{N}(v_{1})$ when $v_{1}$ follows the uniform law on the sphere  stated in Lemma \ref{lemma:Dirichlet}, we roughly obtain the lower bound
\begin{multline}
    \liminf_{N \to \infty} \frac{1}{N} \log\frac{1}{Z_{N}^{t}}\int_{\lambda_1(X_N) \in (x - \varepsilon, x + \varepsilon)} e^{-\frac{N}{4t} \operatorname{Tr}((X_N - B_N)^2)} \, \dd X_N \gtrsim -J_{\nu, t}^{\Lambda}(x, Y) {+} \int \log |\lambda - x| \, \dd(\nu \boxplus \sigma_t)(\lambda)  \\
{-} \frac{x^2}{4t}-C_t    + \liminf_{N \to \infty} \frac{1}{N} \log \mathbb P_{B_{N-1}'}(\lambda_1\geq x),\label{lk}
\end{multline}
where $B_{N-1}'$ is the $(N-1)\times (N-1)$ matrix living in $v_{1}^{\perp}$ where it equals  $p_{v_1}^\perp B_N p_{v_1}^\perp$, 
$C_t$ is a finite constant and 
$$
    J_{\nu, t}^{\Lambda} : (\lambda, Y) \in (-\infty, \ell_{\nu, t}] \times \Delta^{p+1} \mapsto L_{\nu,t}^{\Lambda}(\lambda,Y)
    - \frac{1}{2} \sum_{k = 1}^p \alpha_k \log (y_k / \alpha_k).$$
By Courant-Fischer's theorem (see Lemma \ref{lemma:spectrum}), when ${Y_{N}(v_{1})\simeq Y}$, the matrix $B_{N-1}'$ {has a unique outlier approximately equal to $\Lambda_1 := \Phi(\Lambda, Y)$, which is defined as the  smallest solution  in $[\Lambda, \eta_{1}]$ of the equation}
\begin{equation*}
    \frac{y_0}{\Lambda - \Lambda_1} + \sum_{k = 1}^p \frac{y_k}{\eta_k - \Lambda_1} = 0.
\end{equation*}
{
Note that it always exists when $y_{0}$ and $y_{1}$ are positive since  the function of $\Lambda_{1}$ on the left hand side  then goes from $-\infty$ to $+\infty$. We will see it is then unique. }
{Moreover,} the spectrum of $B_{N-1}'$ converges to ${\nu=\sum_{k=1}^p \alpha_k\delta_{\eta_k}}$.
As a consequence, we see that the last term in \eqref{lk} is also the probability of deviation of the smallest eigenvalue $\lambda_{1}$ above $x$
of an $(N-1)\times (N-1)$ matrix given by the sum of a GOE matrix and a matrix $B_{N-1}'$ with empirical measure of the eigenvalues converging to $\nu$ and with outlier $\Phi(\Lambda,Y)$. We set
\begin{equation}\label{heu2}
    F_{\nu,t}^-(\Lambda, x) := -\lim_{\varepsilon \to 0} \liminf_{N \to \infty} \frac{1}{N} \log \mathbb{P}(\lambda_1(X_{N}) \in (x - \varepsilon, x + \varepsilon)).
\end{equation}
As a consequence of the previous considerations,  we obtain the functional inequality
$$
    -F_{\nu,t}^-(\Lambda, x) \gtrsim \sup_{Y \in \Delta^{p+1}} \Bigl\{ -J_{\nu, t}^{\Lambda}(x, Y) - C_t + \int \log |\lambda - x| \, \dd (\nu \boxplus \sigma_t)(\lambda) - \frac{x^2}{4t} - \inf_{y \geq x} F_{\nu,t}^-(\Phi(\Lambda, Y), y) \Bigr\}\,,
$$
for some constant $C_t$ depending only on $t$. {In other words, we find 
\begin{equation}\label{heu3}F_{\nu,t}^-(\Lambda, x) \lesssim \inf_{Y \in \Delta^{p+1}} \Bigl\{ J_{\nu, t}^{\Lambda}(x, Y) + C_t - \int \log |\lambda - x| \, \dd(\nu \boxplus \sigma_t)(\lambda) + \frac{x^2}{4t} + \inf_{y \geq x} F_{\nu,t}^-(\Phi(\Lambda, Y), y) \Bigr\}\,.
\end{equation}}
{One can define  $F^{+}_{\nu,t}$  similarly but with $\limsup$ instead of $\liminf$ and show that
$$
    F_{\nu,t}^+(\Lambda, x) \gtrsim \inf_{Y \in \Delta^{p+1}} \Bigl\{ J_{\nu, t}^{\Lambda}(x, Y) + C_t - \int \log |\lambda - x| \, \dd(\nu \boxplus \sigma_t)(\lambda) + \frac{x^2}{4t} + \inf_{y \geq x} F_{\nu,t}^+(\Phi(\Lambda, Y), y) \Bigr\}\,.
$$}

We then solve the  above functional inequalities. We first show that $F^{+}_{\nu,t}$ and $F^{-}_{\nu,t}$  vanish at ${\ell^{\Lambda}_{\nu,t}}$, are decreasing on $(-\infty,{\ell^{\Lambda}_{\nu,t}}]$ and increasing on $({\ell^{\Lambda}_{\nu,t}},\ell_{\nu,t}]$. Therefore  $ \inf_{y \geq x} F_{\nu,t}^\pm(\Phi(\Lambda, Y), y)$ vanishes except when $x\ge {\ell^{\Phi(\Lambda,Y)}_{\nu,t}}$, namely when $\omega_{\nu,t}(x)\ge \Phi(\Lambda,Y)$. Then $ \inf_{y \geq x} F_{\nu,t}^\pm(\Phi(\Lambda, Y), y)=  F_{\nu,t}^\pm(\Phi(\Lambda, Y), x)$.
 We then 
prove that {$F_{\nu,t}^+, F_{\nu,t}^-$} are  in fact equal to $I_{\nu,t}^{\Lambda}$
which {satisfies} the functional equation
$$
    I_{\nu,t}^{\Lambda}(x)=\inf_{Y\in\Delta^{p+1}}\Bigr\{J_{\nu,t}^{\Lambda}(x,Y)+C_t-\int \log|y-x|\dd(\nu\boxplus\sigma_t)(y)+\frac{x^2}{4t}+ I_{\nu,t}^{\Phi(\Lambda,Y)}(x)1_{\omega_{\nu,t}(x)\ge \Phi(\Lambda,Y)
    }\Bigr\}\,.
 $$

This establishes the weak LDP which can easily be upgraded to a full LDP. 
{We finally extend} the result to general matrices that satisfy Assumptions \ref{assumption:B}. This is achieved by bounding $B_N$ from below and above by two matrices that have a finite number of eigenvalues, and by using monotonicity arguments.

\section{Large deviations in the pure point case}
Let $B_{N}$ be an  $N\times N$ diagonal matrix with eigenvalues $\Lambda$ {with} multiplicity $1$ and $\eta_1,\ldots,\eta_p$ with multiplicity $N_{1} ,\cdots,N_{p}$ such that {for every} $i\in \{1,\ldots,p\}$, $N_{i}/N$ converges towards $\alpha_{i}>0$. Let $G_N$ be a GOE matrix of size $N\times N$ and variance $t>0$. We assume that $\Lambda=\eta_{0}\le  \eta_{1}<\eta_{2} < \cdots <\eta_{p}$ and denote
\begin{equation}\label{def:nu}
\nu:=\sum_{i=1}^{p}\alpha_{i}\delta_{\eta_{i}}\,.\end{equation}
Throughout this section we denote by $\lambda_{1}\le\lambda_{2}\le\cdots\le\lambda_{N}$ the eigenvalues of $X_{N}=B_{N}+G_{N}$. We denote in short $\lambda^{N}=(\lambda_{1},\ldots,\lambda_{N})$. 

Let $\Delta^{p+1}$ be the $(p+1)$-simplex as defined in \eqref{def:Deltap}. We recall that  $\nu\boxplus\sigma_{t}$ is the {limiting} distribution of the spectrum of $G_{N}+B_N$, namely the almost sure limit of $\frac{1}{N}\sum_{i=1}^{N}\delta_{\lambda_{i}}$, and $\ell_{\nu,t}$ denotes the leftmost point in the support of $\nu\boxplus \sigma_{t}$. 
The goal of this section is to prove the following large deviation principle, which specializes Theorem \ref{theorem:LDP} to the case where $B_N$ has a fixed number of distinct eigenvalues.
Let $I_{\nu,t}^{\Lambda}$ be as defined in \eqref{eq:expl formula}.

\begin{proposition}[LDP in the pure point case]\label{LDP:dis}
Let $\nu$ be as in \eqref{def:nu} and $\Lambda\leq \ell_{\nu}=\inf_{1\leq k\leq p} \eta_k$.
\begin{enumerate} 
\item The map $I_{\nu,t}^{\Lambda}$ is a good rate function which is  strictly convex and vanishes at ${\ell^{\Lambda}_{\nu,t}}$. 
\item Let $\lambda_{1}$ be the smallest eigenvalue of $X_{N}=G_{N}+B_N$. Then the law of $\lambda_{1}$ satisfies a large deviation principle with speed $N$ and rate function $I_{\nu,t}^{\Lambda}$.

\end{enumerate} 
\end{proposition}
It is easy to see that $I_{\nu,t}^{\Lambda}$  is a good rate function. 
Indeed, it is continuous on $(-\infty,\ell_{\nu,t})$ since $S_{\nu}$ is continuous on $(-\infty,\ell_{\nu})$, $\omega_{\nu,t}$ is continuous and increasing on $(-\infty,\ell_{\nu,t})$, with values in
$(-\infty,\omega_{\nu,t}(\ell_{\nu,t})]$, and $\omega_{\nu,t}^{*}$ is continuous with values in $[\omega_{\nu,t}(\ell_{\nu,t}),\ell_{\nu}]$. {This shows that $I_{\nu,t}^{\Lambda}$ is continuous except possibly at $\ell_{\nu,t}$ where it is lower semicontinuous. Moreover, $I_{\nu,t}^{\Lambda}(\ell^{\Lambda}_{\nu,t})=0$. Indeed, if $\Lambda\leq \omega_{\nu,t}(\ell_{\nu,t})$, then $\ell_{\nu,t}^\Lambda=H_{\nu,t}(\Lambda)$ and $\omega_{\nu,t}(\ell^{\Lambda}_{\nu,t})=\Lambda=\gamma_\Lambda(\ell^{\Lambda}_{\nu,t})$. Besides, if $\Lambda\geq \omega_{\nu,t}(\ell_{\nu,t})$, $\ell_{\nu,t}^\Lambda=\ell_{\nu,t}$ and $\omega_{\nu,t}(\ell_{\nu,t} )=\omega_{\nu,t}^*(\ell_{\nu,t})=\gamma_\Lambda(\ell^{\Lambda}_{\nu,t})$.}
Let us show that its level sets are compact. As
 $x\to -\infty$, we notice that $\omega_{\nu,t}(x)\sim x$  whereas $\gamma_{\Lambda}(x)$ stays bounded (when it involves $\omega_{\nu,t}^{*}$, the latter is confined to $[\omega_{\nu,t}(\ell_{\nu,t}),\ell_{\nu}]$ by~\eqref{def:omega*} and hence bounded). It follows that
\begin{equation}\label{atinfinity}
   {I_{\nu,t}^{\Lambda}}(x)=\frac{1}{2}S_\nu(\omega_{\nu,t}(x))+\frac{x^2}{4t}+o(x^2)=\frac{x^2}{4t}+o(x^2)\underset{x \to -\infty}{\longrightarrow}+\infty,
\end{equation}
where $o(x^{2})/x^{2}$ goes to zero as $x$ goes to infinity.
{Hence, the level sets of $I_{\nu,t}^{\Lambda}$ are compact. The convexity of $I_{\nu,t}^{\Lambda}$ is established in Lemma \ref{lemma:convexity}, along with the fact that the minimum is achieved at $\ell^{\Lambda}_{\nu,t}$.} It follows in particular that $I_{\nu,t}^{\Lambda}$ is non-negative. Hence, $I_{\nu,t}^{\Lambda}$ is a convex good rate function.

The fact that $\lambda_{1}$ is exponentially tight is clear since $B_{N}$ is uniformly bounded and the smallest eigenvalue of the GOE is well known to be exponentially tight {\cite[Section 2.6.2]{AGZ}}. We more precisely find that (see e.g the proof of {\cite[Lemma 1.8]{HuGu1}}), since the entries are Gaussian,
   the operator norm of $X_{N}$, namely $\|X_{N}\|_{\mathrm{op}}=\sup_{\|v\|=1}|\langle v,X_{N} v\rangle|$,
 is such that
there exists a finite constant $C$ such that for $N$ large enough,
\begin{equation}\label{expt}\mathbb P(\|X_{N}\|_{\mathrm{op}}\ge K)\le e^{-N(K-C)}.\end{equation}

The main point of the proof of Proposition \ref{LDP:dis}, {on which we will focus for the remainder of this section}, is therefore to prove a weak large deviation principle, namely that for every $x< \ell_{\nu,t}$,
\begin{equation}\label{wldp}
-I_{\nu,t}^{\Lambda}(x)\le \lim_{\delta\rightarrow 0}\liminf_{N\rightarrow\infty}\frac{1}{N}\log \mathbb P(|\lambda_{1}-x|\le \delta)
\le
\lim_{\delta\rightarrow 0}\limsup_{N\rightarrow\infty}\frac{1}{N}\log \mathbb P(|\lambda_{1}-x|\le \delta)\le -I_{\nu,t}^{\Lambda}(x).\end{equation}
In fact,  the large deviation principle \eqref{wldp} is easy to prove for $x>\ell_{\nu,t}$. Indeed, by concentration of measure, since Gaussian variables satisfy a log-Sobolev inequality with a constant independent of their mean, we know that there exist positive universal constants $C_{1},C_{2}$ so that for  every $\zeta>0$, see e.g
 \cite[Corollary 1.4(b)]{GZ00}, 
 \begin{equation}\label{con0}\dP\Bigr(d\Bigr(\frac{1}{N}\sum_{i=1}^N \delta_{\lambda_{i}},\mathbb E\Bigr[\frac{1}{N}\sum_{i=1}^N \delta_{\lambda_{i}}\Bigr]\Bigr)\ge  \zeta\Bigr)\le\frac{C_{1}}{\zeta^{3/2}} e^{-C_{2}N^{2}\zeta^{2}}\,,\end{equation}
 where  $d$ is the distance defined in \eqref{def:distance}.
 Furthermore, since $\mathbb E[\frac{1}{N}\sum \delta_{\lambda_{i}}]$ converges weakly towards $\nu_{\alpha}\boxplus\sigma_{t}$, 
 $d(\mathbb E[\frac{1}{N}\sum \delta_{\lambda_{i}}], \nu_{\alpha}\boxplus\sigma_{t})$ goes to zero as $N$ goes to infinity.
 Therefore,   for every $\zeta>0$,  there exists $C_{\zeta}>0$ such that for $N$ large enough, 
\begin{equation}\label{con1}\dP\Bigr( d\Bigr(\frac{1}{N}\sum_{i=1}^N \delta_{\lambda_{i}},\nu_{\alpha}\boxplus\sigma_{t}\Bigr)\ge  \zeta\Bigr)\le e^{-C_{\zeta/2}N^{2}}\,,\end{equation}
with $\nu_{\alpha}=\sum\alpha_{i}\delta_{\eta_{i}}$. 
 However, if  for  some $\ve>0$, $x>\ell_{\nu, t}+2\ve$, there exists $\zeta>0$ small enough so that $\{|\lambda_{1}-x|\le \ve\}\subset \{d\Bigr(\frac{1}{N}\sum \delta_{\lambda_{i}},\nu_{\alpha}\boxplus\sigma_{t}\Bigr)\ge  \zeta\}$ from which we conclude from \eqref{con1}  that
$$\lim_{\delta\rightarrow 0}\limsup_{N\rightarrow\infty}\frac{1}{N}\log \mathbb P(|\lambda_{1}-x|\le \delta)\le -\infty\,.$$
We do not need to prove \eqref{wldp} 
at $x=\ell_{\nu,t}$ either. Indeed, if $\ell_{\nu,t}=\ell^{\Lambda}_{\nu,t}$, then Theorem \ref{convth} shows that for every $\delta>0$,
$$\limsup_{N\rightarrow\infty}\frac{1}{N}\log \mathbb P(|\lambda_{1}-\ell_{\nu,t}^{\Lambda}|\le \delta)=\liminf_{N\rightarrow\infty}\frac{1}{N}\log \mathbb P(|\lambda_{1}-\ell_{\nu,t}^{\Lambda}|\le \delta)=0,$$
which corresponds to the announced statement, since by Proposition \ref{LDP:dis}, $I_{\nu,t}^{\Lambda}(\ell_{\nu,t}^{\Lambda})=0$.
{If ${\ell_{\nu,t}}>\ell^{\Lambda}_{\nu,t}$, the monotonicity of the large deviation rate function proven in Lemma \ref{monot} implies
that for every $\delta>0$ small enough, every $y<\ell_{\nu,t}$ so that $|y-\ell_{\nu,t}|<|y-\ell_{\nu,t}^{\Lambda}|/2$,}
$$\limsup_{N\rightarrow\infty}\frac{1}{N}\log \mathbb P(|\lambda_{1}-\ell_{\nu,t}|\le \delta)\le
\limsup_{N\rightarrow\infty}\frac{1}{N}\log \mathbb P(|\lambda_{1}-y|\le \delta)$$ 
so that  equation \eqref{wldp} at $y<\ell_{\nu,t}$ gives, after taking the limit $\delta\rightarrow 0$, 
$$\limsup_{\delta\rightarrow 0}\limsup_{N\rightarrow\infty}\frac{1}{N}\log \mathbb P(|\lambda_{1}-\ell_{\nu,t}|\le \delta)\le
-I_{\nu,t}^{\Lambda}(y)$$
and the left-continuity of $I_{\nu,t}^{\Lambda}$ at $\ell_{\nu,t}$ implies by letting $y$ go to $\ell_{\nu,t}$ that
\begin{equation}\label{bo}\limsup_{N\rightarrow\infty}\frac{1}{N}\log \mathbb P(|\lambda_{1}-\ell_{\nu,t}|\le \delta)\le -I_{\nu,t}^{\Lambda}(\ell_{\nu,t})\,.\end{equation}
Similarly for the lower bound, for every $\delta>0$ and $\gamma\in (0,\delta/4)$, we have
$$\liminf_{N\rightarrow\infty}\frac{1}{N}\log \mathbb P(|\lambda_{1}-\ell_{\nu,t}|\le \delta)
\ge \liminf_{N\rightarrow\infty}\frac{1}{N}\log \mathbb P(|\lambda_{1}-\ell_{\nu,t}-\delta/2|\le \gamma),$$ which implies by equation  \eqref{wldp}, after letting $N$ go to infinity  and $\gamma$ go to zero, that
$$\liminf_{N\rightarrow\infty}\frac{1}{N}\log \mathbb P(|\lambda_{1}-\ell_{\nu,t}|\le \delta)\ge
-I_{\nu,t}^{\Lambda}(\ell_{\nu,t}-\delta/2).$$
Equation \eqref{wldp} at  ${x=\ell_{\nu,t}>\ell^{\Lambda}_{\nu,t}}$,  then follows from the left-continuity of $I_{\nu,t}^{\Lambda}$ at $\ell_{\nu,t}$ (and \eqref{bo}).

Hence, the rest of this section  is dedicated to the proof of \eqref{wldp}.
As announced in the {heuristics}, the proof of \eqref{wldp} starts by factorizing the terms depending on $\lambda_{1}$.

\subsection{Projection onto the first eigenvector}

 Let $v_1$ be an eigenvector of $X_N$ of norm $1$ associated to $\lambda_1$. One can first project onto $v_1^\perp$ using the following lemma:

\begin{lemma}\label{lemma:projection}
Let  $X,B\in \mc{S}_N(\dR)$. Let $\lambda_1$ be an eigenvalue of $X$ and $v_1\in S^{N-1}$ be an associated eigenvector. One may write 
 \begin{equation}\label{eq:splitting}
     \tr((X-B)^2)=\lambda_1^2 -2\lambda_1 \langle v_1,Bv_1\rangle+2\langle v_1, B^2 v_1\rangle-\langle v_1, B v_1\rangle^2+\tr((p_{v_1}^\perp(X-B)p_{v_1}^\perp)^2),
 \end{equation}
 where $p_{v_1}^\perp$ is the {projection} onto the orthocomplement of $v_1$. { Moreover, let  $(\Lambda, \eta_{i}, 1\le i\le p)$ be the eigenvalues of $B$. Denote by 
  $w_0$ a unit eigenvector of $B$ associated to the eigenvalue $\Lambda$ and 
  $P_{E_{\eta_{k}}}$ the orthogonal projection onto the eigenspace $E_{\eta_{k}}$ corresponding to the eigenvalue $\eta_{k}$ for $k\in \{1,\ldots,p\}$. Denote by 
  $Y^N=(Y_0^N,\ldots,Y_p^N)$ the vector in $\Delta^{p+1}$ given by 
\begin{equation}\label{eq:defY}
 Y_0^N=\langle v_1,w_0\rangle^{2}, \quad   Y_k^N=\Vert P_{E_{\eta_k}}v_1\Vert^2,\quad 1\leq k\leq p.  
\end{equation}
Then,
\begin{equation}\label{equi}-2\lambda_1 \langle v_1,Bv_1\rangle+2\langle v_1, B^2 v_1\rangle-\langle v_1, B v_1\rangle^2= 4t L^{\Lambda}_{\nu,t}(\lambda_{1},Y^{N})\end{equation}
with $L_{\nu,t}^{\Lambda}:(-\infty, \ell_{\nu, t}]\times \Delta^{p+1}\to\dR$, given for all $x\in  (-\infty, \ell_{\nu, t}] $ and $Y\in \Delta^{p+1}$ by
\begin{equation}\label{defL}  L_{\nu, t}^{\Lambda}(x, Y) := -\frac{x}{2t} \Bigl(\Lambda y_0 + \sum_{k = 1}^p \eta_k y_k \Bigr) + \frac{1}{2t} \Bigl(\sum_{k = 1}^p \eta_k^2 y_k + \Lambda^2 y_0 \Bigr) 
    - \frac{1}{4t} \Bigl(\Lambda y_0 + \sum_{k = 1}^p \eta_k y_k \Bigr)^2.\end{equation}}
\end{lemma}

\begin{proof}
Let $v_1,\ldots,v_N$ be a family of eigenvectors of $X$. By writing the matrix $X-B$ in the basis $v_1,\ldots,v_N$, one can  check that
\begin{equation}\label{eq:tr1}
    \tr((X-B)^2)=\langle v_{1}, (X-B)v_{1}\rangle^{2}+ 2\sum_{j>1}\langle v_{1}, (X-B)v_{j}\rangle^{2}+\tr\Bigl((p_{v_{1}}^{\perp}(X-B)p_{v_{1}}^{\perp})^{2}\Bigr).
\end{equation}
One may then write
\begin{equation*}
   (X-B)v_1=\sum_{j=1}^N \langle v_j,(X-B)v_1 \rangle v_j.
\end{equation*}
Therefore
\begin{equation*}
    \langle v_1, (X-B)^2 v_1\rangle=\sum_{j=1}^N \langle v_j, (X-B)v_1\rangle^2
\end{equation*}
which allows one to rewrite \eqref{eq:tr1} as
\begin{equation*}
    \tr((X-B)^2)=2\langle v_{1}, (X-B)^{2} v_{1}\rangle- \langle v_{1}, (X-B)v_{1}\rangle^{2}+\tr\Bigl((p_{v_{1}}^{\perp}(X-B)p_{v_{1}}^{\perp})^{2}\Bigr).
\end{equation*}
One can finally check using $X v_1=\lambda_1 v_1$ that
\begin{equation}\label{eq:pol}
    2\langle v_{1}, (X-B)^{2} v_{1}\rangle- \langle v_{1}, (X-B)v_{1}\rangle^{2}=\lambda_1^2 -2\lambda_1 \langle v_1,Bv_1\rangle+2\langle B^2 v_1, v_1\rangle-\langle v_1, B v_1\rangle^2.
\end{equation}
{Equality} \eqref{eq:splitting} follows. To prove \eqref{equi}, observe that 

\begin{align}
    \langle v_1, B_{N} v_1\rangle&= \Lambda \langle v_1, w_0 \rangle^2+\sum_{k=1}^p \eta_k \Vert P_{E_{\eta_k}}v_1 \Vert^2=\Lambda Y_{0}^{N}+\sum_{k=1}^{p}\eta_{k}Y_{k}^{N},\label{totoo} \\
    \langle v_1, B_{N}^2 v_1\rangle&=\Lambda^2 \langle v_1, w_0\rangle^2 +\sum_{k=1}^p \eta_k^2 \Vert P_{E_{\eta_k}} v_1\Vert^2=\Lambda^{2}Y_{0}^{N}+\sum_{k=1}^{p}\eta_{k}^{2}Y_{k}^{N}.\label{toto}
\end{align}
{Equality} \eqref{equi} follows by plugging \eqref{totoo} and  \eqref{toto} in \eqref{eq:pol}.

\end{proof}

\subsection{LDP for Dirichlet variables}

One can observe that, for $v_1$ uniformly distributed on $S^{N-1}$, the random variable $Y^N$  defined in \eqref{eq:defY} follows a Dirichlet law with parameter $(\frac{1}{2},\frac{N_{1}}{2},\ldots,\frac{N_{p}}{2})$ which are such that $N_{i}/N$ goes to $\alpha_{i}>0$ {for every} $i=1,\ldots,p$. Using the explicit density of the Dirichlet distribution, it is  easy to prove by Laplace's method   \cite[Theorem 3.3]{HuGu2}  that $Y^N$ satisfies the following LDP:

\begin{lemma}\label{lemma:Dirichlet}
Let $v_1$ be uniformly distributed on $S^{N-1}$. Define \begin{equation}
 Y_0^N=\langle v_1,w_0\rangle^{2}, \quad   Y_k^N=\Vert P_{E_{\eta_k}}v_1\Vert^2,\quad 1\leq k\leq p.  
\end{equation}
Then, the sequence of random variables $(Y^N)$ satisfies a large deviation principle at speed $N$ and with good  rate function given by
\begin{equation}\label{eq:defI}
    I_{\nu}:Y=(y_0,\ldots,y_p)\in \Delta^{p+1}\mapsto -\frac{1}{2}\sum_{k=1}^p \alpha_{k}\log (y_k/\alpha_{k}).
\end{equation}
\end{lemma}

\subsection{Outliers of $p_{v_1}^\perp B_N p_{v_1}^\perp$}

As outlined in Subsection \ref{sub:proof ideas}, the proof involves computing the probability that the smallest eigenvalue of $G_{N-1}+B_{N-1}'$ is larger than $x$, where $G_{N-1}$ is a GOE matrix of size $(N-1)\times (N-1)$ on the orthocomplement of $v_1$ and $B_{N-1}'$ is the $(N-1)\times (N-1) $ matrix given by $p_{v_1}^\perp B_N p_{v_1}^\perp$ in $v_{1}^{\perp}$. Thus, we first study the outliers of the matrix $B_{N-1}'$ when $v_{1}$ is {such} that $Y_{N}$ is equal to some  $Y$, and then invoke known results on the BBP transition. Here,  $x\leq \ell_{\nu,t}$ and $Y\in \Delta^{p+1}$ are given. 
{We define $\Phi(\Lambda,\cdot):\Delta^{p+1}\to \dR$ by setting for all $Y\in \Delta^{p+1}$, {
\begin{equation}\label{def:PhiL}
    \Phi(\Lambda,Y)=\begin{cases} \Lambda & \text{if $y_0=0$},\\
    \Lambda_1 & \text{if $y_0 \in (0,1], y_{1}\in (0,1]$},\\
        \eta_1  & \text{if $y_1=0$ and $y_{0}\neq 0$},
    \end{cases}
\end{equation}
where $\Lambda_{1}\in (\Lambda,\eta_{1})$ is the unique  solution in $(\Lambda,\eta_{1})$ of the equation
\begin{equation}\label{eq:sm}
    \frac{y_0}{\Lambda-\Lambda_1}+\sum_{k=1}^p \frac{y_k}{\eta_k-\Lambda_1}=0.
\end{equation}
Observe that $\Lambda_{1}$ is always uniquely defined when $y_{0}$ and $y_{1}$ are positive. Indeed, because $y_{0}=1-\sum_{k=1}^{p}y_{k}  $, \eqref{eq:sm} is equivalent to\begin{equation}\label{eq:sm2}
f(\Lambda_{1})=1\quad \mbox{ with }\quad f(x)=\sum_{k=1}^{p}y_{k}\frac{\eta_{k}-\Lambda}{\eta_{k}-x}\,.\end{equation}
$f$ is continuous on $[\Lambda,\eta_{1})$ with $f(x)\to+\infty$ as $x\uparrow\eta_{1}$ when $y_{1}>0$, and $f$ extends continuously to $\eta_{1}$ when $y_{1}=0$; its image is $f([\Lambda,\eta_{1}])= [\sum_{k=1}^{p} y_{k}, (+\infty)1_{y_{1}>0}+(\sum_{k=2}^{p}y_{k}\frac{\eta_{k}-\Lambda}{\eta_{k}-\eta_{1}})1_{y_{1}=0}]$. Hence, since $\sum_{k=1}^{p}y_{k}=1-y_{0}<1$ if $y_{0}>0$, we see that 
\eqref{eq:sm2} has at least a solution when $y_{0}$ and $y_{1}$ are positive by the intermediate value theorem. Because 
$f$ is strictly increasing on $[\Lambda,\eta_{1}]$, this solution is unique.   Note also that $y_{0}=1$ implies $y_{1}=0$ and therefore $  \Phi(\Lambda,Y)=\eta_{1}$.  Moreover, when $y_{1}=0$, there might be no solution if $\sum_{k=2}^{p}y_{k}\frac{\eta_{k}-\Lambda}{\eta_{k}-\eta_{1}}<1$, for instance if $y_{0}$ is big enough.
}

\begin{lemma}\label{lemma:spectrum}
{Let $B_{N}$ be an  $N\times N$ diagonal matrix with eigenvalues $\eta_0:=\Lambda<\eta_1<\cdots<\eta_p$ with multiplicity $N_0=1, N_{1},\cdots,N_{p}$ such that {for every} $i\in \{1,\ldots,p\}$, $N_{i}\geq 2$. Let $v_1\in \dR^N$ be a unit vector and  $B_{N-1}'$ be the $(N-1)\times (N-1)$ matrix given by the restriction of $B_{N}$ to the orthocomplement of $v_{1}$.}
Moreover:

\begin{enumerate}
\item{  Let $\kappa>0$. 
Let $\Phi(\Lambda,\cdot)$  be given by   \eqref{def:PhiL}. Then $\Phi(\Lambda,\cdot)$ is continuous on $[0,1]\times [\kappa,1]\times [0,1]^{p-1}\cap \Delta^{p+1}$. }

\item Assume $v_1\notin \cup_{k=0}^p E_{\eta_k}$. For each $k=0,\ldots,p$, set
\begin{equation}\label{eq:v proj}
y_k(v_{1}):=\|P_{E_{\eta_k}}v_1\|^2,\qquad Y(v_{1}):=(y_0(v_{1}),\ldots,y_p(v_{1})).
\end{equation}
Assume $y_0(v_1)>0$. Then the matrix $B_{N-1}'$
has eigenvalues $\Lambda_1<\cdots<\Lambda_p$ of multiplicity $1$, and $\eta_1,\ldots,\eta_p$ of multiplicities
$N_1-1,\ldots,N_p-1$ (where we adopt the convention that multiplicities are added when eigenvalues coincide). Moreover, for every $k=1,\ldots,p$, $\Lambda_k\in(\eta_{k-1},\eta_k]$, and
$\Lambda_1=\Phi(\Lambda,Y(v_1))$. 
\item If  $v_1\in E_{\eta_{k}}$ for some $k=0,\ldots,p$, then the matrix $B_{N-1}'$ has  eigenvalues $\eta_i$ with multiplicity $N_{i}$ {{for every} $i\in\{0,\ldots,p\}\setminus \{k\}$}, and  $\eta_{k}$ with multiplicity $N_k-1$. Moreover, 
the smallest eigenvalue $\Lambda_1$ of $B_{N-1}'$
equals $\Phi(\Lambda,Y(v_{1}))$.
\end{enumerate}
\end{lemma}

\begin{proof} {
Let us prove (1).  Let us first show the continuity of $\Phi(\Lambda,\cdot)$ on  $Y\in\bigl([\ve,1]\times [\kappa,1]\times[0,1]^{p-1}\bigr)\cap\Delta^{p+1}$ for some $\ve>0$.  We see from \eqref{eq:sm} that $\Lambda_{1}$ stays away from $\Lambda$ and $\eta_{1}$. In fact, because $y_{0}\ge \ve$ and $\sum y_{i}\le 1$
$$\frac{\ve}{\Lambda_{1}-\Lambda}\le \frac{y_{0}}{\Lambda_{1}-\Lambda}=\sum_{k=1}^{p}\frac{y_{k}}{\eta_{k}-\Lambda_{1}}\le \frac{1}{\eta_{1}-\Lambda_{1}}$$
implies that $\Lambda_{1}-\Lambda\ge \ve(\eta_{1}-\Lambda_{1})$ and therefore that $\Lambda_{1}-\Lambda\ge \ve(1+\ve)^{-1}(\eta_{1}-\Lambda) $. The same argument shows that since $y_{1}\ge\kappa$ and $y_{0}\le 1$,
$$\frac{1}{\Lambda_{1}-\Lambda}\ge  \frac{y_{0}}{\Lambda_{1}-\Lambda}=\sum_{k=1}^{p}\frac{y_{k}}{\eta_{k}-\Lambda_{1}}\ge \frac{y_{1}}{\eta_{1}-\Lambda_{1}}\ge \frac{\kappa}{\eta_{1}-\Lambda_{1}}$$
shows that $\eta_{1}-\Lambda_{1}\ge  \kappa(1+\kappa)^{-1}(\eta_{1}-\Lambda) $.  }
Hence, $\Lambda_{1}$ is at distance of order $\ve$ of $\Lambda$ and $\kappa$ of  $\eta_{1}$. {Thus, if $Y,Y'\in [\ve,1]\times [\kappa,1]\times[0,1]^{p-1}\cap\Delta^{p+1}$, we easily deduce from \eqref{eq:sm} that $\Phi(\Lambda,Y)$ is the zero of a smooth function and hence 
 that there exists a finite constant $C(\ve,\kappa)$ so that 
$$|\Phi(\Lambda,Y)-\Phi(\Lambda,Y')|\le C(\ve,\kappa) \sum_{i=1}^{p}|y_{i}-y_{i}'|\,,$$
which implies the continuity of $\Phi(\Lambda,.)$ on $ [\ve,1]\times [\kappa,1]\times[0,1]^{p-1}\cap\Delta^{p+1}$.
When $y_{0}$ goes to zero, we have since the $\eta_{i}$'s  are bounded above by $\eta_{p}$ and $\sum y_{k}=1-y_{0}$,
$$\frac{y_{0}}{\Lambda_{1}-\Lambda}=\sum_{k=1}^{p}\frac{y_{k}}{\eta_{k}-\Lambda_{1}}\ge \frac{1-y_{0}}{\eta_{p}-\Lambda_{1}}$$
which insures that $\Lambda_{1}-\Lambda\le y_{0}(1-y_{0})^{-1}(\eta_{p}-\Lambda)$ goes to zero as $y_{0}$ goes to zero.
It is clear that $\Phi(\Lambda,Y)$ was extended by continuity to $\Lambda$ when $y_{0}$ goes to zero, which in turn completes the proof of  the continuity of $\Phi(\Lambda,.)$ on $[0,1]\times [\kappa,1]\times [0,1]^{p-1}\cap \Delta^{p+1}$}. 

We next turn to the proof of (2).
For each $k$, the eigenspace of $B_{N-1}'$ associated with $\eta_k$ is
\begin{equation*}
    E_{\eta_k}'=E_{\eta_k}\cap v_1^\perp.
\end{equation*}
Hence $\dim(E_{\eta_k}')$ is either $N_k$ (if $P_{E_{\eta_k}}(v_1)=0$) or $N_k-1$ (if $P_{E_{\eta_k}}(v_1)\neq 0$).
In particular, since $P_{E_{\eta_0}}(v_1)\neq 0$, we have $E_{\eta_0}'=\{0\}$ and $\Lambda(=\eta_0)$ is not an eigenvalue of
$B_{N-1}'$.

Let $\lambda_1'\leq \cdots\leq \lambda_{N-1}'$ be the eigenvalues of $B_{N-1}'$, and $\lambda_1\leq\cdots\leq \lambda_N$ those
of $B_N$. By the Courant--Fischer theorem,
\begin{equation*}
    \lambda_1\leq \lambda_1'\leq \cdots \leq \lambda'_{N-1}\leq \lambda_N.
\end{equation*}
Thus, for each $k=1,\ldots,p$, there is at most one eigenvalue of $B_{N-1}'$ in $(\eta_{k-1},\eta_k)$; when it exists we denote it
by $\Lambda_k$. If $\eta_k$ has multiplicity $N_k-1$ in $B_{N-1}'$, then such a $\Lambda_k$ lies in $(\eta_{k-1},\eta_k)$; otherwise,
with a slight abuse of notation, we set $\Lambda_k:=\eta_k$.

By the Courant-Fischer theorem, $\Lambda_1$ can be expressed  as
\begin{equation*}
    \Lambda_1=\inf_{v\in \dR^N:\Vert v\Vert_{2}=1, \langle v,v_1\rangle=0}\langle v, B_N v\rangle.
\end{equation*}
Let $u\in \dR^N$ be a minimizer in the above problem. By variational calculus, there exist $C_1, C\in\dR$ such that
\begin{equation*}
    B_N u=C u+C_{1} v_1.
\end{equation*}
Since $\langle u,v_1\rangle=0$, we have
\begin{equation*}
    \langle u,B_N u\rangle=C=\Lambda_1.
\end{equation*}
Since $\Lambda_1\notin \{\lambda_1,\ldots,\lambda_N\}$, {we deduce that   for every $i=1,\ldots,N$, if  we denote by $w_{i}$ the eigenvector of $B_{N}$ with eigenvalue $\lambda_{i}$ and for a vector $r\in S^{N-1}$
$r_{i}=\langle r, w_{i}\rangle$, }
\begin{equation*}
    u_i=C_{1}\frac{(v_1)_i}{\lambda_i-\Lambda_1}.
\end{equation*}
Thus, $\Lambda_1$ satisfies
\begin{equation*}
  \langle u,v_{1}\rangle= C_{1}\sum_{i=1}^N \frac{(v_1)_i^2}{\lambda_i-\Lambda_1}=0,
\end{equation*}
which can be written,  since $C_{1}\neq 0$ as $u$ has norm one, 
\begin{equation*}
    \frac{y_0(v_{1})}{\Lambda_1-\Lambda}+\sum_{k=1}^p \frac{y_k(v_{1})}{\Lambda_1-\eta_k}=0.
\end{equation*}
$\Lambda_{1}$ is thus the unique solution of \eqref{eq:sm} as claimed. 
Item (3) is clear. \end{proof}

We next study the typical behavior of the smallest eigenvalue of $p_{v_1}^\perp (G_N+B_N)p_{v_1}^\perp$ in $v_{1}^{\perp}$,  when $v_1$ is as in Lemma \ref{lemma:spectrum}.
It is a straightforward application of  Lemma \ref{lemma:spectrum} and Theorem \ref{convth} as this matrix is the sum of a GOE matrix and the matrix $B_{N-1}'$ with spectrum asymptotically distributed according to $\nu$ and outlier $\Phi(\Lambda,Y(v_{1}))$ as in  Lemma \ref{lemma:spectrum}.
\begin{lemma}[Smallest eigenvalue of the remainder]\label{lemma:eigenvalue proj2} Let $\Lambda=\eta_{0}\le \eta_{1}<\eta_{2}<\cdots<\eta_{p}$.
Let $B_{N}$ be  an $N\times N$ diagonal matrix with eigenvalue $\Lambda$ with multiplicity $1$ and $\eta_1,\ldots,\eta_p$ of multiplicities $N_{i}$ such that $N_{i}/N$ goes to $\alpha_{i}>0$. Let $\nu:=\sum_{k=1}^p\alpha_{k} \delta_{\eta_k}$. Let $v_{1}$ and $Y(v_1)$  be as in Lemma \ref{lemma:spectrum}. Then the smallest eigenvalue of $p_{v_1}^\perp (G_N+B_N)p_{v_1}^\perp$ in $v_{1}^{\perp}$ converges towards $H_{\nu,t}(\Phi(\Lambda,Y(v_1)))$ if $\Phi(\Lambda,Y(v_1))\le \omega_{\nu,t}(\ell_{\nu,t})$ and to $\ell_{\nu,t}$ otherwise. 
\end{lemma}

\subsection{A functional inequality for the large deviation upper bound}
{We denote for $\alpha\in\Delta^{p}$ and $\delta>0$, $\mc{B}_{\delta}^{N}(\alpha)$ the $\ell_{\infty}$-ball
$$\mc{B}_{\delta}^{N}(\alpha)=\Bigr\{ \bM=(M_{1},\ldots,M_{p})\in (\mathbb{N})^p:\sup_{1\le i\le p}\Bigr|\frac{M_{i}}{N}-\alpha_{i}\Bigr|\le\delta,\quad \sum_{i=1}^{p}M_{i}=N-1\Bigr\}.$$
 We shall assume without loss of generality that $\delta$ is  strictly smaller than $\min\alpha_{i}$ so that the integers $M_{i}$ are strictly positive.
Let $B_N$ be a diagonal matrix with eigenvalues $\Lambda$ with multiplicity $1$,  and $\eta_1,\ldots,\eta_p$ of respective  multiplicities $\bM=(M_1,\ldots,M_p)$. Let $\bM\in \mc{B}_{\delta}^{N}(\alpha)$. 
{For every} $\Lambda\leq \ell_{\nu}$,  let us denote
$\dP_{N, \Lambda,\bM}$ the law of $X_{N}=G_N+B_N$ and
\begin{equation*}
    P_{N, \Lambda,\bM}(x,\ve):=\dP_{N,\Lambda,\bM}(|\lambda_1(X_{N})-x|\leq \ve),
\end{equation*}}
{Recall that  for $\alpha\in\Delta^{p}$, we denote $\nu=\sum\alpha_{i}\delta_{\eta_{i}}$. For all $\Lambda\leq \ell_\nu$ and $x\leq \ell_{\nu,t}$, let us define
\begin{equation}\label{eq:defF+}
    F_{\nu,t}^+(\Lambda,x):=-\limsup_{\Lambda'\to \Lambda}\lim_{\ve\downarrow 0}\lim_{\delta\downarrow 0}\limsup_{N\to\infty}\sup_{\bM\in \mc{B}_{\delta}^{N}(\alpha)}\frac{1}{N}\log P_{N,\Lambda',\bM}(x,\ve).
\end{equation}}

\begin{remark}\label{remark:lsc}
Let us remark that, by construction, $\Lambda \mapsto F_{\nu,t}^+(\Lambda,x)$ is lower semicontinuous. Indeed, if $f:\dR\to\dR$ is a measurable function, the function
    \begin{equation*}
        g:x\in \dR\mapsto \limsup_{y\rightarrow x} f(y)=\lim_{\ve\to 0}\sup_{B(x,\ve)}f
    \end{equation*}
is upper semicontinuous.    
   
\end{remark}
Recall the map $\Phi$ defined in  \eqref{def:PhiL}. Let us define
\begin{equation}\label{def:C_t}
    C_t:=\frac{1}{2}-\frac{1}{2}\ln t.
\end{equation}
Recall the map $L_{\nu,t}^{\Lambda}$ defined in \eqref{defL} and the map {$I_{\nu}:Y\in \Delta^{p+1}\mapsto - \frac{1}{2} \sum_{k = 1}^p \alpha_k \log (y_k / \alpha_k)$} from Lemma \ref{lemma:Dirichlet}. We set    
{\begin{equation}\label{def:JnutL}   
J_{\nu,t}^{\Lambda}:(x,Y)\in (-\infty, \ell_{\nu, t}]\times \Delta^{p+1}\mapsto L_{\nu, t}^{\Lambda}(x, Y)+I_{\nu}(Y),
\end{equation}}    
and, for $(\Lambda, x,Y)\in (-\infty,\ell_\nu]\times(-\infty, \ell_{\nu, t}]\times \Delta^{p+1}$, we define
{\begin{equation}\label{defK} K_{\nu,t}( \Lambda, x,Y)
= J_{\nu,t}^{\Lambda}(x,Y)-C_t-\int \log|\lambda-x|\dd(\nu\boxplus\sigma_t)(\lambda)+\frac{x^2}{4t}
\,.\end{equation}}
We denote  
$f(x^{-})=\lim_{\ve\downarrow 0}f(x-\ve)$. The goal of this section is to prove the following:

\begin{lemma}[Functional upper bound]\label{lemma:upper bound}
{Let $\Lambda\leq \ell_{\nu}$. Let $\Phi(\Lambda,\cdot)$ be as in \eqref{def:PhiL} and $\ell_{\nu,t}^\Lambda$ be as in \eqref{def:ellnut}. For every $x<\ell_{\nu,t}$, we have
\begin{equation}\label{eq:ub}
    F_{\nu,t}^+(\Lambda,x)\geq \inf_{Y\in\Delta^{p+1}}\Bigr\{K_{\nu,t}(\Lambda,x,Y)+ F_{\nu,t}^+(\Phi(\Lambda,Y),x^{-})1_{\omega_{\nu,t}(x)> \Phi(\Lambda,Y)
    }\Bigr\}
 \end{equation}
 Moreover, $x\mapsto  F_{\nu,t}^+(\Lambda,x)$ decreases 
    on $(-\infty, \ell^{\Lambda}_{\nu,t})$ and increases on $(\ell^{\Lambda}_{\nu,t},\ell_{\nu,t})$, while it vanishes at $\ell^{\Lambda}_{\nu,t}$. It is  bounded from below by $c(x-\ell^{\Lambda}_{\nu,t})^{2}$ for some $c>0$.}
\end{lemma}

\medskip

\begin{proof} 
The properties of $F^{+}_{\nu,t}$ listed at the end of the {lemma} are proven in Lemma \ref{Fprop} and Lemma \ref{monot}, and we  focus below on the proof of the functional inequality \eqref{eq:ub}.   We fix $\alpha\in \Delta^{p}\cap (0,1]^{p}$ and denote by $\nu_{\alpha}=\sum \alpha_{i}\delta_{\eta_{i}}$.  For $\delta>0$ and  $\bM \in \mc{B}_{\delta}^{N}(\alpha)$ we denote  by $\nu_{\frac{\bM}{N}}=\frac{1}{N}\delta_{\Lambda}+\sum \frac{M_{i}}{N}\delta_{\eta_{i}}$ the empirical measure of  the eigenvalues of a diagonal matrix  $B_{N}(\bM)$   with eigenvalues  $\eta_{i}$ with multiplicity $M_{i}$,  $\bM=(M_{1},\ldots,M_{p})$, $\Lambda$ with multiplicity $1$. Let us first prove that for every $\delta>0$, there exists $C_{\delta}>0$ such that, for $N$ large enough, for every  $\bM\in \mc{B}_{\delta}^{N}(\alpha)$, 
\begin{equation}\label{con12}\dP_{N,\Lambda,\bM}\Bigr( d\Bigr(\frac{1}{N}\sum_{i=1}^N \delta_{\lambda_{i}},\nu_{\alpha}\boxplus\sigma_{t}\Bigr)\ge 8\|\eta\|_\infty\sqrt{\delta p} \Bigr)\le e^{-\frac{1}{2}C_{\delta}N^{2}}\,.\end{equation}
We observe that \eqref{con12} is uniform in $\bM\in \mc{B}_{\delta}^{N}(\alpha)$.  In fact, recall from \eqref{con0} that there are two universal positive constants $C_{1},C_{2}$ so that for every $\zeta>0$, for every $\bM\in\mathbb N^{p}$ so that $\sum M_{i}=N-1$,
 \begin{equation}\label{con02}\dP_{N,\Lambda,\bM}\Bigr(d\Bigr(\frac{1}{N}\sum_{i=1}^N \delta_{\lambda_{i}},\mathbb E_{\dP_{N,\Lambda,\bM}}\Bigl[\frac{1}{N}\sum_{i=1}^N \delta_{\lambda_{i}}\Bigr]\Bigr)\ge  \zeta\Bigr)\le\frac{C_{1}}{\zeta^{3/2}} e^{-C_{2}N^{2}\zeta^{2}}\,,\end{equation}
On the other hand, for every $\bM\in \mc{B}_{\delta}^{N}(\alpha)$,  for $N$ large enough, 
$d(\mathbb E_{\dP_{N,\Lambda,\bM}}[\frac{1}{N}\sum \delta_{\lambda_{i}}], \nu_{\frac{\bM}{N}}\boxplus\sigma_{t})$ goes to zero. We claim that this convergence is uniform.
Indeed, by the Cauchy-Schwarz and Hoffman-Wielandt inequalities \cite[Lemma 2.1.19]{AGZ},  for every Lipschitz function $f$, if $X_{N},Y_{N}$ are two self-adjoint matrices,
$$\left|\frac{1}{N}\tr(f(X_{N}+G_{N}))-\frac{1}{N}\tr(f(Y_{N}+G_{N}))\right|\le\|f\|_{L}\sqrt{ \frac{1}{N}\tr((X_{N}-Y_{N})^{2})},
$$
where we recall that $\|f\|_L$ stands for the Lipschitz seminorm of $f$. As a consequence,
\begin{equation}\label{con200}
d\left( \frac{1}{N}\sum_{i=1}^{N}\delta_{\lambda_{i}(G_{N}+B_{N}(\bM))}, \frac{1}{N}\sum_{i=1}^{N}\delta_{\lambda_{i}(G_{N}+B_{N}(\bM'))}\right)\le \sqrt{\frac{1}{N}\tr((B_N(\bM)-B_N(\bM'))^2)}.
\end{equation}
It follows that
\begin{equation}\label{con20}
d\left( \frac{1}{N}\sum_{i=1}^{N}\delta_{\lambda_{i}(G_{N}+B_{N}(\bM))}, \frac{1}{N}\sum_{i=1}^{N}\delta_{\lambda_{i}(G_{N}+B_{N}(\bM'))}\right)\le \Vert \eta\Vert_{\infty} \sqrt{\frac{2}{N}\sum_{i=1}^{p} |M_{i}-M_{i}'|}\le 2\|\eta\|_\infty\sqrt{\delta p}.\end{equation}
This control also holds in the large $N$ limit,  yielding
$d(\nu_{\frac{\bM}{N-1}}\boxplus\sigma_{t},\nu_{\frac{\bM'}{N-1}}\boxplus\sigma_{t})\le  2\|\eta\|_\infty\sqrt{\delta p}$. Moreover, \eqref{con20} also holds after taking the expectation. As a consequence, we deduce that uniformly on $\bM\in  \mc{B}_{\delta}^{N}(\alpha)$, for $N$ large enough,
\begin{equation}\label{bobo}
d\Bigl(\mathbb E_{\dP_{N,\Lambda,\bM}}\Bigl[\frac{1}{N}\sum_{i=1}^{N}\delta_{\lambda_{i}}\Bigr], \nu_{\alpha}\boxplus \sigma_t\Bigr)\le 4\|\eta\|_\infty\sqrt{\delta p},\end{equation}
from which we get  \eqref{con12} uniformly from \eqref{con02} by taking $\zeta=4\|\eta\|_\infty\sqrt{\delta p}$.

Hereafter $\alpha $ is fixed and we denote in short $\nu$ for $\nu_{\alpha}$.
 Moreover, since the sequence $(B_{N})_{N\in\mathbb N}$ is bounded uniformly for the operator norm, there exists $L_{0}$ finite {($L_0=(2+\|B_{N}\|_{\mathrm{op}})^{2}$) such that} for every $L>L_{0}$, by \cite[(2.6.21)]{AGZ}
\begin{equation}\label{con2}\sup_{\bM\in\mc{B}_{\delta}^{N}(\alpha) }\dP_{N,\Lambda,\bM}\Bigr(\frac{1}{N}\sum \lambda_{i}^{2}\ge L\Bigr)\le e^{-(L-L_{0})N^{2}}\,.\end{equation}
Therefore, if for $L>L_{0}$ and $\delta>0$ fixed, we let 
$$K_{L,\delta}=\left\{ \lambda^{N}\in\mathbb R^{N} : \frac{1}{N}\sum \lambda_{i}^{2}\le L, d\Bigr(\frac{1}{N}\sum \delta_{\lambda_{i}},\nu_{\alpha}\boxplus\sigma_{t}\Bigr)\le \delta \right\},$$
then we deduce from \eqref{con2} and \eqref{con12} that   there exists some positive real number $c(\delta,L)$ so that  for $N$ large enough $\dP_{N,\Lambda,\bM}(K_{L,\delta}^{c})\le e^{-c(L,\delta) N^{2}}$ and therefore
\begin{equation}\label{eq:loc}
 P_{N,\Lambda,\bM}(x,\ve)=\dP_{N,\Lambda,\bM}\Bigr(|\lambda_{1}(X_{N})-x|\le\ve, \lambda^{N}(X_{N})\in K_{L,\delta}\Bigr) +O(e^{-c(L,\delta) N^{2}})\end{equation}
  where $O(e^{-c(L,\delta) N^{2}})$ is bounded above by $e^{-c(L,\delta) N^{2}}$. We denote in the following
 $$\mathcal C_{L,\delta}(x,\ve):=\{  \lambda^{N}\in K_{L,\delta}: \lambda_{1}<\lambda_{2}<\cdots<\lambda_{N}, |\lambda_{1}-x|\le \ve\}\,.$$
  By \eqref{eq:loc}, one needs to upper bound 
 $$P^{L,\delta}_{N, \Lambda,\bM}(x,\ve):=\dP_{N,\Lambda,\bM}\Bigr(\lambda^{N}(X_{N})\in \mathcal C_{L,\delta}(x,\ve)\Bigr)\,.$$
 Let us perform the change of variables $X_N\mapsto (\lambda_1,\ldots,\lambda_N, O_{N})$ where $\lambda_1<\ldots<\lambda_N$ are the eigenvalues of $X_N$ and $O_{N}$ the  corresponding orthogonal matrix so that $X_{N}=O_{N}D_{N}(\lambda)O_{N}^{T}$ with $D_{N}(\lambda)$ the diagonal matrix with entries $(\lambda_{i})_{1\le i\le N}$. We find 
\begin{equation*}
     P^{L,\delta}_{N, \Lambda,\bM}(x,\ve)=\frac{1}{Z_{N}^{t}}\int_{\lambda^{N}\in \mathcal C_{L,\delta}(x,\ve)} \prod_{i< j}|\lambda_i-\lambda_j|1_{\lambda_1<\ldots <\lambda_N}e^{-\frac{N}{4t}\tr( (O_{N}D_{N}(\lambda)O_{N}^{T}-B_{N})^2 ) }\dd O_{N}\dd \lambda^{N},
\end{equation*}
where $\dd O_{N}$ is the Haar measure on $O_N(\dR)$ and $Z_N^t$ is the normalizing constant of the deformed GOE measure, i.e.\ the value
of the same integral when the indicator of $\{\lambda^N \in C_{L,\delta}(x,\varepsilon)\}$ is replaced by $1$, $\dd\lambda^{N}$ denotes $ \dd \lambda_1\ldots\dd \lambda_N$ in short. 
Applying Lemma \ref{lemma:projection} (see \eqref{equi} and \eqref{defL}) gives, by the definition \eqref{defL} of $L_{\nu,t}^{\Lambda}$, 
\begin{eqnarray*}
   P_{N,\Lambda,\bM}^{L,\delta}(x,\ve)&=&\int_{S^{N-1}}\dd v_1\int   \dd \mathbb P_{v_1}(O_{N-1})  \mathbb M (\mathcal C_{L,\delta}(x,\ve), v_{1},O_{N-1})\\
       \mathbb M (A, v, O) &=&
   \frac{1}{Z^{t}_{N}}   \int_{\lambda^{N}\in A
    } \prod_{i< j}|\lambda_i-\lambda_j|e^{-N L_{\nu,t}^{\Lambda}(\lambda_1,Y(v))-\frac{N}{4t}\lambda_1^2  -\frac{N}{4t}\tr((OD_{N-1}(\lambda)O^{T}-B_{N-1}')^2) }
    \dd \lambda^{N}
\end{eqnarray*}
where $\mathbb P_{v_1}$ is the Haar measure on the orthocomplement  of $v_{1}$, $B_{N-1}'$ the projection of $B_{N}$ onto the orthocomplement of $v_{1}$  and $D_{N-1}(\lambda)$ is  the diagonal matrix with entries $(\lambda_{2},\ldots,\lambda_{N})$. 

Let $\chi, \kappa,  \zeta>0$.  We denote $V_{0}=\{ v\in S^{N-1}: y_{1}(v)=\Vert P_{E_{\eta_1}}v\Vert^2<2\kappa\}$ where we recall that  $P_{E_{\eta_k}}v$ is the orthogonal projection of the vector $v$ onto the eigenspace $E_{\eta_{k}}$ of $B_{N}$ for the eigenvalue $\eta_{k}$ and $\Vert v\|_{E_{\eta_{k}}}=\Vert P_{E_{\eta_k}}v\Vert$ its Euclidean norm.
 Since $I_\nu$ is lower semicontinuous, one can cover $V_{0}^{c}$ by a finite union of measurable sets  $\cup_{1\le k\le M} V_{Y_{k},\zeta_{k},\chi}$ where for $Y=(y_{0},\ldots,y_{p})\in\Delta^{p+1}$, and $\zeta,\chi>0$, 
\begin{equation}\label{eq:defEY}
    V_{Y,\zeta, \chi}:=\Bigr\{v\in S^{N-1}:\max_{\ 0\leq k\leq p} |\Vert P_{E_{\eta_k}}v\Vert^2-y_k|\leq \zeta,  I_{\nu}((\Vert P_{E_{\eta_k}}v\Vert^2)_{0\le k\le p})\ge I_{\nu}((y_{k})_{0\le k\le p})-\chi
    \Bigr\}
\end{equation}
and $(Y_{k})_{1\le k\le M}$ is a family of points of $\Delta^{p+1}$. We can assume from the start that $\zeta_{j}\le\zeta\le\kappa$ for every $j$  for some $\zeta>0$ to be chosen later and so that $Y_{j}(1)\ge \kappa>0$.  We denote in short $V_{k}$ for $ V_{Y_{k},\zeta_{k},\chi}$.
We therefore find

\begin{equation}\label{exp}
   P_{N,\Lambda,\bM}^{L,\delta}(x,\ve) \le\sum_{k=0}^{M}\int_{v_{1}\in V_{k}}\dd v_1\int   \dd \mathbb P_{v_1}(O_{N-1})  \mathbb M (\mathcal C_{L,\delta}(x,\ve), v_{1},O_{N-1})\end{equation}
   We next estimate each term in the above right hand side.
   For every $L\in (2,\infty)$, let  $h_{L}$ be  a continuously differentiable function with values in $[0,1]$, derivative uniformly bounded by one,  equal to one for $|x|\le L/2$ and  to zero for  $|x|\ge L$. We set

\begin{equation*}
f_{L}(x):=h_{L}(x)\log(\max(|x|,L^{-1}))+ 4\frac{L^{-2}}{\log L} (L_{0}+1).
\end{equation*}
  Observe  that $\log |x|\le h_{L}(x)\log(\max(|x|,L^{-1}))$  if $|x|\le L/2$. Therefore, by Chebyshev's inequality,
\begin{align*}
\sum_{j>1}\log |\lambda_1-\lambda_j| &\le \sum_{j>1}  h_{L}(\lambda_{1}-\lambda_{j}) \log\bigl(\max(|\lambda_{1}-\lambda_{j}|,L^{-1})\bigr)+ \sum_{j:|\lambda_{i}-\lambda_{1}|>L/2} \log(|\lambda_{1}-\lambda_{j}|)\\
&\le  \sum_{j>1}  h_{L}(\lambda_{1}-\lambda_{j}) \log\bigl(\max(|\lambda_{1}-\lambda_{j}|,L^{-1})\bigr)+ \frac{L^{-2}}{\log L}\sum_{j:|\lambda_{i}-\lambda_{1}|>L/2} |\lambda_{1}-\lambda_{j}|^{2}\\
\end{align*}
Therefore, on  $\{d(\frac{1}{N}\sum \delta_{\lambda_{i}},\nu\boxplus\sigma_{t})\le \delta\}\cap\{\frac{1}{N}\sum \lambda_{i}^{2}\le L_{0}+1\}\cap \{|\lambda_{1}-x|\le \ve\}$ with $|x|+\ve\le \sqrt{ L_{0}}$,
$$\prod_{j>1}|\lambda_1-\lambda_j|\le \exp \Bigr\{ \sum_{i=1}^{N} f_{L}(\lambda_{1}-\lambda_{i})\Bigr\}\le \exp \Bigr\{ N\int f_{L}(\lambda_{1}-y) \dd(\sigma_{t}\boxplus \nu)(y)+ NC_{L}\delta\Bigr\},$$
where we finally used that $f_{L}$ is Lipschitz to find a finite constant $C_{L}$ {such that} the last inequality holds. Finally, if  $x+\ve<\ell_{\nu,t}$,  and therefore is at positive distance from the support of  $\sigma_{t}\boxplus\nu$,
 because $\sigma_{t}\boxplus \nu$ is compactly supported,  
\begin{equation}\label{approx} \int f_{L}(\lambda_{1}-y) \dd(\sigma_{t}\boxplus \nu)(y)=\int \log|x-y|\dd(\sigma_{t}\boxplus \nu)(y)+O(\ve)+O(L^{-1})\,.\end{equation}Similarly $\lambda_{1}^{2}/4t$ is very close to $x^{2}/4t$.
As a consequence, we find that 
\begin{equation}\label{eq2}
  \mathbb M (  \mathcal C_{L,\delta}(x,\ve), v_{1}, O_{N-1})    \le  e^{ N( \int \log|x-y|\dd(\sigma_{t}\boxplus \nu)(y)-\frac{x^{2}}{4t}+o(1))}   \mathbb L (  \mathcal C_{L,\delta}(x,\ve), v_{1}, O_{N-1}) \end{equation}
  with  $o(1)= C_{L}\delta+O(\ve)+O(L^{-1})$ and 
  $$
    \mathbb L ( A, v_{1}, O_{N-1}):=\frac{1}{Z^{t}_{N}} \int_{\lambda^{N}\in A}
     \prod_{2\le i< j}|\lambda_i-\lambda_j|e^{-N L_{\nu,t}^{\Lambda}(\lambda_1,Y(v_{1}))  -\frac{N}{4t}\tr((O_{N-1}D_{N-1}(\lambda)O_{N-1}^{T}-B_{N-1}')^2) }
    \dd \lambda^{N}
\,.$$

Let us  finally remark that $O_{N-1}D_{N-1}(\lambda)O_{N-1}^{T}$ follows {the law} of a GOE of dimension $N-1$ and variance $t(N-1)/N$   under $\mathbb P_{N-1}(\dd \lambda,\dd O_{N-1})$ defined in \eqref{def:prob}
so that 
\begin{equation}\label{eq3}  \int \dd\mathbb P_{v_1}(O_{N-1})   \int_{\lambda_{2}\le \cdots\le\lambda_{N}
    } \prod_{2\le i< j}|\lambda_i-\lambda_j|e^{  -\frac{N}{4t}\tr((O_{N-1}D_{N-1}(\lambda)O_{N-1}^{T}-B_{N-1}' )^2) }\dd \lambda_2\ldots\dd \lambda_N =Z_{N-1}^{t\frac{N-1}{N}}\,.\end{equation}
Moreover, $Z_{N}^{t}$ can be exactly computed by Selberg integral, see e.g \cite[(2.5.11)]{AGZ}, resulting, for every $t>0$ and $N\in \mathbb N$,  in
$$Z_{N}^{t}=N!\Bigr(\frac{2t}{N}\Bigr)^{\frac{N(N+1)}{4}}(2\pi)^{N}\prod_{j=0}^{N-1}\frac{\Gamma((j+1)/2)}{\Gamma(1/2)},$$
from which the limit of $\frac{1}{N}\log \frac{Z_{N-1}^{t\frac{N-1}{N}}}{Z_{N}^{t}}$ is easily  computed  to be equal to $C_t$ as in \eqref{def:C_t}.

We  next estimate the first term in the right-hand side of \eqref{exp}. Putting together \eqref{eq2} and \eqref{eq3} we see that there are finite constants $C, C'$ so that
$$\int_{v_{1}\in V_{0}}\dd v_1\int   \dd \mathbb P_{v_1}(O_{N-1})  \mathbb M (\mathcal C_{L,\delta}(x,\ve), v_{1},O_{N-1})\le e^{CN} \mathbb P(y_{1}(v_{1})\le 2\kappa)\le e^{-C'N }\Bigr(\frac{2\kappa}{\alpha_{1}}\Bigr)^{\frac{1}{2}\alpha_{1}N}$$
where we finally used Lemma \ref{lemma:Dirichlet} and assumed $N$ large enough.  The constant $C$ comes from the supremum of the mass of $ \mathbb M (  \mathcal C_{L,\delta}(x,\ve), v_{1}, O_{N-1})$ over $v_{1}$ which is easily seen to be at most of order $CN$ by the previous discussion, whereas the constant $C'$ integrates $C$ and the term in the rate function for the large deviation principle  for $Y(v_{1})$  given in Lemma \ref{lemma:Dirichlet} which does not depend on $\kappa$.

We now estimate the next terms in the right hand side of \eqref{exp} and fix $k\in \{1,\ldots,M\}$.
When $v_{1}\in V_{k}= V_{Y_{k},\zeta_{k},\chi}$ and  $|\lambda_{1}-x|\le\ve$, $ L_{\nu,t}^{\Lambda}(\lambda_1,Y(v_{1})) $ is at a distance from  $L_{\nu,t}^{\Lambda}(x,Y_{k})$ of order $\zeta_{k}+\ve$. Moreover, $y_{1}(v_{1})\ge \kappa$ if $\zeta_{k}\le\zeta\le\kappa$. 
By Lemma \ref{lemma:spectrum}, the matrix $B_{N-1}'$ has  eigenvalues $\eta_1,\ldots,\eta_p$ of respective multiplicities between $M_i-1$ and $M_i$ with $p$ possible outliers $\Lambda_1\leq \cdots\leq \Lambda_p$ with {for every} $k=2,\ldots,p$, $\Lambda_k\in [\eta_{k-1},\eta_k]$.  Moreover, by the continuity of $\Lambda_{1}$ proven in Lemma
\ref{lemma:spectrum} (1) and (2), $\Lambda_{1}$ is at distance at most  of order $o_{\zeta}(1)$ going to zero with $\zeta$  from  $\Phi(\Lambda,Y_{k})$. Note that $o_{\zeta}(1)$ can be chosen uniformly as $Y$ belongs to a compact set and hence $\Phi$ is uniformly continuous. We assume it is smaller than $\ve/2$. 
 We can   bound  
$B_{N-1}'$ from above by a matrix  $B_{N-1}''+\ve/2 I$,  that is $B_{N-1}''+\ve/2 I-B_{N-1}'$ is positive semidefinite, in $v_{1}^{\perp}$ with the  same eigenspaces, the outlier $\Phi(\Lambda,Y_{k})$ and eigenvalues $\eta_{i}$ with multiplicity 
$M_{i}'=M_{i}$ except for $\eta_{1}$ which has multiplicity $M_{1}'=M_{1}-1$.
 We therefore get 
 \begin{align}
 &\int_{v_{1}\in V_{k}}\dd v_1\int   \dd \mathbb P_{v_1}(O_{N-1})  \mathbb M (\mathcal C_{L,\delta}(x,\ve), v_{1},O_{N-1})\nonumber\\
&\qquad \le e^{ N( \int \log|x-y|\dd(\sigma_{t}\boxplus \nu)(y)-\frac{x^{2}}{4t}-L_{\nu,t}^\Lambda(x,Y_{k})+o(1))} \frac{Z_{N-1}^{t\frac{N-1}{N}} }{Z_{N}^{t}} \int_{v_{1}\in V_{k}} \dd v_{1}\mathbb K (x,\ve, v_{1})\label{jk}\end{align}
where   $o(1)= C_{L}\delta+O(\ve)+O(L^{-1})$ and $ \mathbb K (x,\ve, v_{1})$ is equal to

\begin{align}&=
\Bigl(Z_{N-1}^{t\frac{N-1}{N}} \Bigr)^{-1}\int   \dd \mathbb P_{v_1}(O_{N-1})\int_{ x-2\ve\leq \lambda_2\leq \cdots \leq \lambda_N }\prod_{2\le i<j}|\lambda_i-\lambda_j|e^{  -\frac{N}{4t}\tr((O_{N-1}D_{N-1}(\lambda)O_{N-1}^{T}-B_{N-1}')^2) }\dd \lambda^{N-1}\nonumber\\
& = \mathbb P_{N-1,\Lambda_{1},\bM'}(\lambda_{1}(\sqrt{N(N-1)^{-1}} G_{N-1}+B_{N-1}')\ge x-2\ve)\nonumber\\
&\le  \mathbb P_{N-1,\Lambda_{1},\bM'}(\lambda_{1}(G_{N-1}+B_{N-1}'')\ge x-3\ve)=: \mathbb K (x,\ve, Y_{k})\label{d-1}
 \end{align}
where  we denoted $\dd\lambda^{N-1}=\dd\lambda_{2 }\ldots\dd \lambda_N$, and in the last line we incorporated the error due to the small change in the variance (which can be removed after rescaling and boundedness of $B_{N}'$) in an additional error of order $\ve$ for $N$ large enough.  Note here
that  $\mathbb K (x,\ve, Y_{k})$ only depends on $Y_{k}$ and $x$. 
Observe that Theorem \ref{convth} and Lemma \ref{lemma:spectrum} imply that, { if  the empirical measure of the $\lambda_{i}$ is at distance smaller than $\delta$ from $\nu\boxplus\sigma_{t}$,} the smallest eigenvalue of $G_{N-1}+B_{N-1}''$ is close to $ H_{\nu,t}(\Phi(\Lambda,Y_{k}))$ if $\Phi(\Lambda,Y_{k})<\omega_{\nu,t}(\ell_{\nu,t})$ and close to $\ell_{\nu,t}$ otherwise, with some error going to zero when $N$ goes to infinity  and then $\delta$ to zero.
We therefore have three cases:

\smallskip

$\bullet$ If $\Phi(\Lambda,Y_{k})\ge \omega_{\nu,t}(\ell_{\nu,t})$, then the smallest eigenvalue of $G_{N-1}+B_{N-1}''$ converges towards $\ell_{\nu,t}$ and $\mathbb P_{N-1,\Lambda_{1},\bM'}(\lambda_{1}(G_{N-1}+B_{N-1}'')\ge x-3\ve)$ goes to one as soon as $x\le \ell_{\nu,t}$ and $\ve$ is small enough.
\smallskip

$\bullet$ If $x\le  H_{\nu,t}(\Phi(\Lambda,Y_{k}))$ and $ \Phi(\Lambda,Y_{k})\le \omega_{\nu,t}(\ell_{\nu,t})$, then the smallest eigenvalue of $G_{N-1}+B_{N-1}''$ goes  to $ H_{\nu,t}(\Phi(\Lambda,Y_{k}))$ when $N$ goes to infinity and is therefore greater than $x-3\ve$. Hence, we see that $\mathbb P_{N-1, \Lambda_{1}, \bM'}(\lambda_{1}(G_{N-1}+B_{N-1}'')\ge x-2\ve)$ goes to one as $N$ goes to infinity.
\smallskip

$\bullet$ If $x> H_{\nu,t}(\Phi(\Lambda,Y_{k}))$ and $ \Phi(\Lambda,Y_{k})\le \omega_{\nu,t}(\ell_{\nu,t})$,  then $\mathbb P_{N-1, \Lambda_{1}, \bM'}(\lambda_{1}(G_{N-1}+B_{N-1}'')\ge x-3\ve)$ goes to zero exponentially fast,  hence this term contributes to the probability of large deviations and needs to be evaluated.
Observe that $\lambda_{1}(G_{N-1}+B_{N}'')\le \ell_{\nu,t}+\ve$ with overwhelming probability by \eqref{con12}. We can cut the compact set $[x-4\ve,\ell_{\nu,t}]$ into a finite number $K$ of intervals of size $\gamma$  to deduce that
$$\mathbb P_{N-1, \Lambda_{1},\bM'}(\lambda_{1}(G_{N-1}+B_{N}'')\ge x-3\ve)\le \sum_{p=1}^{K}P_{N-1, \Phi(\Lambda,Y_{k}),\bM'}(x-4\ve+p\gamma, \gamma)\,.$$
According to the proof of Lemma \ref{monot}, there exists $c>0$ such that
$$P_{N-1, \Phi(\Lambda,Y_{k}),\bM'}(x-4\ve+p\gamma, \gamma)\le 2N^{2}P_{N-1, \Phi(\Lambda,Y_{k}),\bM'}(x-4\ve, \gamma)+2N^{2}e^{-cN^{2}}$$ so that we obtain

$$\mathbb P_{N-1, \Lambda_{1},\bM'}(\lambda_{1}(G_{N-1}+B_{N}'')\ge x-3\ve)\le 2KN^{2}P_{N-1, \Phi(\Lambda,Y_{k}),\bM'}(x-4\ve, \gamma)+2N^{2}Ke^{-cN^{2}}\,.$$
Hence, for $x> H_{\nu,t}(\Phi(\Lambda,Y_{k}))$  and $ \Phi(\Lambda,Y_{k})\le \omega_{\nu,t}(\ell_{\nu,t})$, after taking the large $N$ limit and $\gamma$ to zero, we find
$$\limsup_{N\rightarrow\infty}\frac{1}{N}\log \mathbb P_{N-1, \Lambda_{1}, \bM'}(\lambda_{1}(G_{N-1}+B_{N}'')\ge x-3\ve) \le -F^{+}_{\nu,t}(\Phi(\Lambda,Y_{k}),x-4\ve )\,.$$

To conclude we need to first take the supremum over $\bM\in\mc{B}_{\delta}^{N}(\alpha)$, so $\bM'\in  \mc{B}^{N-1}_{\delta}(\alpha)$ before the large $N$ limit.  We have to be careful because $\nu$ depends implicitly on $\bM/N$ and $Y_{N}$ is not equal to $Y$ but is  close to it. But the condition
 $\{ \Phi(\Lambda,Y)\le \omega_{\nu,t}(\ell_{\nu,t})\}\cap \{x\ge  H_{\nu,t}(\Phi(\Lambda,Y))\}=\{\omega_{\nu,t}(x)\ge \Phi(\Lambda,Y)\}$ is continuous in $\nu$ (as $\omega_{\nu,t}(x)$ is) and $Y$. Moreover, $F^{+}_{\nu,t}(\Phi(\Lambda,Y),x-4\ve )$ is lower semicontinuous  in $Y$. 
  The argument above is therefore {valid} uniformly on $\bM\in\mc{B}_{\delta}^{N}(\alpha)$ as long as we are in the interior of the set $\{\omega_{\nu,t}(x)\ge \Phi(\Lambda,Y)\}$ for $\delta$ small enough.
 We therefore conclude that, with    $o_{\zeta}(1)$ going to zero with $\zeta$,
\begin{equation}\label{Bk}\hspace{-0.3cm}
 \lim_{\delta\downarrow 0}   \limsup_{N\to \infty}\sup_{v_{1}\in V_k\atop \bM\in\mc{B}_{\delta}^{N}(\alpha)}
 \frac{1}{N}\log  \mathbb K (x,\ve, v_{1})
  \leq -F_{\nu,t}^+(\Phi(\Lambda,Y_{k}),x-4\ve )1_{\omega_{\nu,t}(x) > \Phi(\Lambda,Y_{k})}+o_{\zeta}(1).
\end{equation}
By Lemma \ref{lemma:Dirichlet}, for every $k$,
\begin{equation*}
\limsup_{N\rightarrow\infty}\frac{1}{N}\log \mathbb P(V_{k})\le -I_{\nu}(Y_{k})+\chi.
\end{equation*}
We therefore deduce from \eqref{exp}, \eqref{jk}, \eqref{Bk} that,  with $o_{\chi,\kappa, \ve}(1)$ going to zero when $\chi,\kappa,\ve$ go to zero, {
\begin{multline*}
 \lim_{\delta\downarrow 0} \limsup_{N\to\infty}\sup_{\bM\in\mc{B}_{\delta}^{N}(\alpha)}\frac{1}{N}\log P_{N,\Lambda,\bM}(x,\ve)\\ \leq \max\Big\{ \sup_{k}\Bigr\{-L_{\nu,t}^{\Lambda}(x,Y_k)-I_\nu(Y_k) - F_{\nu,t}^+(\Phi(\Lambda,Y_k),x-4\ve)1_{\omega_{\nu,t}(x)> \Phi(\Lambda,Y_{k})}\Bigr\}\\
   +\int \log|\lambda-x|\dd(\nu\boxplus\sigma_t)(\lambda)-\frac{x^2}{4t}-C_t +o_{\chi,\kappa, \ve}(1), \frac{1}{2}\alpha_{1}\log \frac{\kappa}{\alpha_{1}}+C'\Big\}.\\
   \end{multline*}}
  {Letting $\ve$ go to zero, we obtain}
   \begin{multline*}
 \lim_{\ve\downarrow 0} {  \lim_{\delta\downarrow 0}} \limsup_{N\to\infty}\sup_{\bM\in\mc{B}_{\delta}^{N}(\alpha)}\frac{1}{N}\log P_{N,\Lambda,\bM}(x,\ve)
 \\ \leq \max\Bigl\{ \sup_{k}\Bigr\{-K_{\nu,t}(\Lambda,x,Y_{k})- F_{\nu,t}^+(\Phi(\Lambda,Y_k),x^{-})1_{\omega_{\nu,t}(x)\ge \Phi(\Lambda,Y_{k})}\Bigr\}+o_{\chi,\kappa}(1),\frac{1}{2}\alpha_{1}\log\frac{ \kappa}{\alpha_{1}}+C'\Bigr\}
 \\
   \leq \max\Bigl\{ \sup_{Y\in\Delta^{p+1}}\Bigr\{-K_{\nu,t}(\Lambda,x,Y) - F_{\nu,t}^+(\Phi(\Lambda,Y),x^{-})1_{\omega_{\nu,t}(x)>\Phi(\Lambda,Y)}\Bigr\}+o_{\chi,\kappa}(1), \frac{1}{2}\alpha_{1}\log\frac{\kappa}{\alpha_{1}}+C'\Bigr\}.
\end{multline*}
We can finally let $\chi$, $\kappa$ go to zero . Now, we take the supremum over $\Lambda\in [\Lambda'-\ve,\Lambda'+\ve]$ on both sides. In the right-hand side it can be interchanged with the supremum over $Y$. The map $J_{\nu,t}^{\Lambda}(\lambda,Y)=L_{\nu,t}^{\Lambda}(\lambda,Y)+I_\nu(Y)$ is uniformly continuous in $\Lambda$. Moreover 
$$\lim_{\ve\downarrow 0}\inf_{\Lambda\in [\Lambda'-\ve,\Lambda'+\ve]}F_{\nu,t}^+(\Phi(\Lambda,Y),x^{-})1_{\omega_{\nu,t}(x)> \Phi(\Lambda,Y)}\ge F_{\nu,t}^+(\Phi(\Lambda',Y),x^{-})1_{\omega_{\nu,t}(x)> \Phi(\Lambda',Y)},$$
by lower semi-continuity of $F_{\nu,t}^+(.,x^{-})$, and the continuity of $\Lambda\mapsto \Phi(\Lambda,Y)$.  Note that the boundary point when $\omega_{\nu,t}(x)= \Phi(\Lambda,Y)$ is also such that $\omega_{\nu,t}(x^{-})= \Phi(\Lambda,Y)$ by continuity of $\omega_{\nu,t}$ and then since $F_{\nu,t}^{+}(\omega_{\nu,t}(x),x)$ vanishes at every $x$ (by the identity $H_{\nu,t}(\omega_{\nu,t}(x))=x$ and the definition of $F_{\nu,t}^{+}$), so that the above quantity vanishes at this point. We therefore conclude that \eqref{eq:ub} holds. 

\end{proof}
\subsection{A functional inequality for the large deviation lower bound}
{For every $\Lambda\leq \ell_\nu$ and $x\leq \ell_{\nu,t}$}, define
\begin{equation}\label{eq:defF-}
    F_{\nu,t}^-(\Lambda,x):=-\liminf_{{\Lambda'\to\Lambda}}\lim_{\ve\downarrow 0}\lim_{\delta\downarrow 0}
    \liminf_{N\to\infty}\inf_{\bM\in \mc{B}_{\delta}^{N}(\alpha)} \frac{1}{N}\log P_{N,\Lambda',\bM}(x,\ve).
\end{equation}
{Denoting $f(x^{+})=\lim_{\ve\downarrow 0} f(x+\ve)$, our goal is to prove the following:}

\begin{lemma}[Functional lower bound]\label{lemma:lower bound}
Let $\Lambda\leq \ell_{\nu}$. {Let $\Phi(\Lambda,\cdot)$ be as in \eqref{def:PhiL} and $\ell_{\nu,t}^\Lambda$ be as in \eqref{def:ellnut}}. For every $x< \ell_{\nu,t}$, we have
\begin{equation}\label{eq:lb}
    F_{\nu,t}^-(\Lambda,x)\leq \inf_{Y\in\Delta^{p+1}}\Bigr\{K_{\nu,t}(\Lambda,x,Y)+F_{\nu,t}^-(\Phi(\Lambda,Y),x^{+})1_{\omega_{\nu,t}(x)\ge \Phi(\Lambda,Y)} \Bigr\}.
\end{equation}
Moreover, $0\le F_{\nu,t}^+\leq F_{\nu,t}^-$. Furthermore,
$x\mapsto  F_{\nu,t}^-(\Lambda,x)$ decreases 
    on $(-\infty, \ell^{\Lambda}_{\nu,t})$ and increases on $(\ell^{\Lambda}_{\nu,t},\ell_{\nu,t})$, while it vanishes at ${\ell^{\Lambda}_{\nu,t}}$. It is  bounded  from below by $c(x-\ell^{\Lambda}_{\nu,t})^{2}$ for some $c>0$.
\end{lemma}

\begin{proof}
The only remaining task is to prove \eqref{eq:lb}, as the rest of the lemma is established in Lemma \ref{Fprop}. Let $x\in (-\infty,\ell_{\nu,t})$ and $\ve>0$.  

Again $B_{N}$ has eigenvalues $\eta_{i}$ with multiplicity $M_{i}$ and $\Lambda'$ with multiplicity $1$. Let $Y\in \Delta^{p+1}$ and $\kappa,\zeta, \chi>0$. We may and shall assume that $y_{1}>2\kappa$ since we saw in the upper bound that small $y_{1}$ have exponentially small contributions. We also can take $\zeta$ small enough so that $y_{1}$ is bounded below by $\kappa$  uniformly on $V_{Y,\zeta,\chi}$ defined in  \eqref{eq:defEY}.
One can bound $P_{N,\Lambda',\bM}(x,\ve)$ from below by
\begin{multline}\label{eq:start l}
P_{N,\Lambda',\bM}(x,\ve)\geq \frac{1}{Z_{N}^{t}}
\int_{V_{Y,\zeta,\chi}}\dd v_1
\int \dd\mathbb P_{v_1}(v_2,\ldots,v_N) \\
\int_{ B(\delta,\ve,x)}
\exp\Big\{-N L_{\nu,t}^{\Lambda'}\big(\lambda_{1}, Y(v_1)\big)
-\frac{N}{4t}\lambda_{1}^{2}\Big\}
\prod_{1\le i< j}|\lambda_i-\lambda_j|\,
\exp\Big\{-\frac{N}{4t}\tr\Big(\big(p_{v_1}^\perp (O D(\lambda)O^{T}-B)\,p_{v_1}^\perp\big)^{2}\Big)\Big\}
\,\dd \lambda^{N}.
\end{multline}
where $B(\delta,\ve,x):=\{|\lambda_{1}-x|\le\ve\}\cap C(\delta,\ve,x)$ with, if $\nu=\sum\alpha_{i}\delta_{\eta_{i}}$,
\begin{equation*}
    C(\delta,\ve,x):=\{x+2\ve<\lambda_{2}<\cdots<\lambda_{N}\}\cap \Bigr\{d\Bigr(\frac{1}{N-1}\sum_{i=2}^N \delta_{\lambda_i},\nu\boxplus\sigma_{t}\Bigr)\le 2\delta\Bigr\}\cap\Bigr\{\frac{1}{N-1}\sum_{i=2}^N \lambda_{i}^{2}\le L_{0}+1\Bigr\}.
\end{equation*}

Using the continuity of $L_{\nu,t}^{\Lambda'}$ and the fact that $\lambda_{1}$ is at distance at least $\ve$ from the rest of the spectrum  (if $x+2\ve<\ell_{\nu,t}$) so that, by arguments similar to those used in \eqref{approx},   with $o_{\delta}(1)$ going to zero as $\delta$ goes to zero  (while $\delta\ll \ve$ and $L_{0}$ is fixed)
$$\frac{1}{N}\sum_{i=2}^{N}\log|\lambda_{1}-\lambda_{i}|=\int \log|\lambda_{1}-y|\dd(\nu\boxplus\sigma_{t})(y)+o_{\delta}(1)\quad \text{on $ B(\delta,\ve,x)$},$$
we find some  $o_{\ve,\delta}(1)$ going to zero when $\delta\ll \ve$ go to zero, 
\vspace{-0.2cm}
\begin{multline*}
    P_{N,\Lambda',\bM}(x,\ve)\geq \frac{2\ve Z^{t (1-N^{-1})}_{N-1} }{Z^{t}_{N}}
    e^{N (\int \log|x-\lambda|  \dd(\sigma_{t}\boxplus \nu)(\lambda)
    - L_{\nu,t}^{\Lambda'}(x,Y)-\frac{1}{4t}x^2+o_{\ve,\delta}(1))}\\
 \times  \int_{V_{Y,\zeta,\chi}}\dd v_1
   \mathbb P(\lambda( G_{N-1}+B_{N-1}')\in   C(\delta,\ve,x))
 \end{multline*}
 where finally we compared the remaining probability to that of a perturbed GOE matrix in $v_{1}^{\perp}$ as in \eqref{d-1}.
We recall  that $B_{N-1}'$, as a $(N-1)\times (N-1)$ matrix in $v_{1}^{\perp}$, has a unique outlier in $[\Lambda,\ell_{\nu}]$, which is close to  $\Phi(\Lambda',Y)$, and eigenvalues $\eta_{i}$ with multiplicities in $\{M_{i}, M_{i}-1\}$.  Let $B_{N-1}''$ be the matrix with outlier $\Phi(\Lambda',Y)$ and eigenvalues $\eta_{i}$ with multiplicity $M_{i}'=M_{i}$ except for $\eta_{p}$ with multiplicity $M_{p}'=M_{p}-1$ (with the same eigenspaces as $B_{N-1}'$). Then $B_{N-1}'\ge B_{N-1}''-\ve$  for $N$ large enough so  that
 \begin{eqnarray*}
  \mathbb P(\lambda(G_{N-1}+B_{N-1}')\in   C(\delta,\ve,x))&\ge &\mathbb P(\lambda_{1}( G_{N-1}+B_{N-1}')\ge x+2\ve )\\
  &&-
  \mathbb P\Bigr(d\Bigr(\frac{1}{N-1}\sum_{i=2}^N \delta_{\lambda_i},\nu\boxplus\sigma_{t}\Bigr)\ge \delta\Bigr)-\mathbb P\Bigr(\frac{1}{N-1}\sum_{i=2}^N \lambda_{i}^{2}\ge L_{0}+1\Bigr)
  \\
  &\ge&
   \mathbb P(\lambda_{1}( G_{N-1}+B_{N-1}'')\ge x+3\ve)- e^{-cN^{2}},\\
    \end{eqnarray*}
     where we finally used \eqref{con12} and \eqref{con2}. $c$ depends on $\delta$ but we can choose $\delta$ going to zero with $N$ so that $cN^{2}\gg N$. As in the proof of the upper bound, the {right-hand} side goes to $1$ for $\ve$ small enough except when $\omega_{\nu,t}(x)> \Phi(\Lambda,Y)$.
 Notice that $B_{N-1}''$ satisfies the same assumptions as $B_{N}$ except that its outlier is now going to $\Phi(\Lambda',Y)$ (when $\kappa$ goes to zero) so that, for every $\Lambda<\ell_{\nu}$, for every $Y\in\Delta^{p+1}$,
$$\lim_{\delta\to 0}\liminf_{N\to\infty} \inf_{\bM\in \mc{B}_{\delta}^{N}(\alpha)}\frac{1}{N}\log P_{N,\Lambda',\bM}(x,\ve)
\ge -K_{\nu,t}(\Lambda',x,Y)- F_{\nu,t}^-(\Phi(\Lambda',Y),x+3\ve)+o(\ve,\kappa),
$$
where we took $\delta$ to zero. We can then let $\ve$ going to zero and take the supremum over all possible $Y$. 
We can finally take the infimum over $\Lambda'\in [\Lambda-\ve,\Lambda+\ve]$ and use the upper semi-continuity of $F_{\nu,t}^-$. 
This completes the proof of \eqref{eq:lb} as soon as $x+3\ve< \ell_{\nu,t}$, which is true for $\ve$ small enough as long as $x<\ell_{\nu,t}$.
\end{proof}

\subsection{The direct regime}
Recall that for  $\Lambda\leq \ell_{\nu}$ and $\lambda\leq \ell_{\nu,t}$, 
\begin{equation}\label{def:gamma}
    \gamma_\Lambda(\lambda):=\begin{cases}
        \Lambda & \text{if $\Lambda\leq \omega_{\nu,t}(\ell_{\nu,t})$, or $\Lambda\geq \omega_{\nu,t}(\ell_{\nu,t})$ and $\lambda\leq H_{\nu,t}(\Lambda)$}\\
        \omega^*_{\nu,t}(\lambda) & \text{if $\Lambda\geq \omega_{\nu,t}(\ell_{\nu,t})$ and $\lambda\geq H_{\nu,t}(\Lambda)$},
    \end{cases}
\end{equation}
{where $\omega^*_{\nu,t}(\lambda)$ is as in \eqref{def:omega*} and $H_{\nu,t}(\Lambda)$ as in \eqref{def:rhonut}}. Also define
\begin{equation}\label{def:ALambda}
  I_{\nu,t}^{\Lambda}(\lambda,\gamma):=\frac{1}{2}(S_{\nu}(\omega_{\nu,t}(\lambda))-S_\nu(\gamma))+\frac{1}{4t}((\lambda-\gamma)^2-(\lambda-\omega_{\nu,t}(\lambda))^2).
\end{equation}
and recall  ${I_{\nu,t}^{\Lambda}(\lambda):=I_{\nu,t}^{\Lambda}(\lambda, \gamma_\Lambda(\lambda))}$.
We first prove that $F_{\nu,t}^{+}$ and $F_{\nu,t}^{-}$ equal  $I_{\nu,t}^{\Lambda}$ in the simplest case where the functional equations \eqref{eq:ub} and \eqref{eq:lb} are such that the infimum is taken at $Y$ so that  the term depending on $F_{\nu,t}^{+}$ or $F_{\nu,t}^{-}$ {vanishes}.
\begin{lemma}\label{lemma:boundaryV}
{Let $H_{\nu,t}(\Lambda)$ be as in \eqref{def:rhonut}.} We have that for every $\lambda<\ell_{\nu,t}$ and $\Lambda\le\ell_{\nu}$,
\begin{equation}\label{eq:b1}
F_{\nu,t}^+(\Lambda,\lambda)\leq F_{\nu,t}^-(\Lambda,\lambda)\leq I_{\nu,t}^{\Lambda}(\lambda).
\end{equation}

 Moreover, if $\Lambda\geq \omega_{\nu,t}(\lambda)$, then
    \begin{equation}\label{eq:equality}
        F_{\nu,t}^-(\Lambda,\lambda)=F_{\nu,t}^+(\Lambda,\lambda)=I_{\nu,t}^{\Lambda}(\lambda)=\inf K_{\nu,t}(\Lambda,\lambda,\cdot).    \end{equation}
\end{lemma}

\begin{figure}
\centering
\newcommand{\figW}{0.46\textwidth}

\begin{subfigure}[b]{\figW}
	\centering
	\begin{tcolorbox}[colframe=black, colback=white, boxrule=0.45pt]		\includegraphics[width=\linewidth]{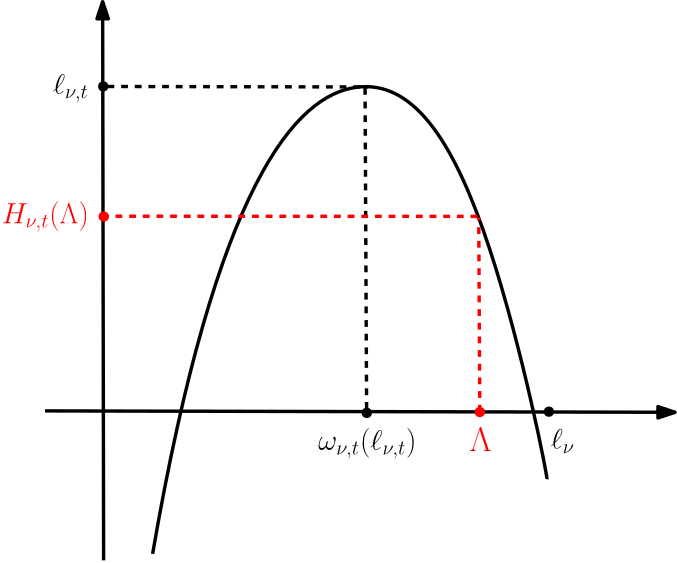}
	\end{tcolorbox}
	\subcaption{$\Lambda \geq \omega_{\nu,t}(\ell_{\nu,t})$}
\end{subfigure}
\hfill
\begin{subfigure}[b]{\figW}
	\centering
	\begin{tcolorbox}[colframe=black, colback=white, boxrule=0.45pt]
		\includegraphics[width=\linewidth]{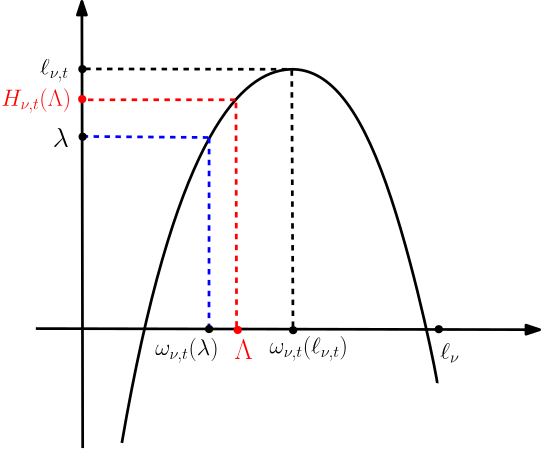}
	\end{tcolorbox}
	\subcaption{$\Lambda \leq \omega_{\nu,t}(\ell_{\nu,t})$ and $\lambda\leq H_{\nu,t}(\Lambda)$}
\end{subfigure}

\vspace{0.35cm}

\begin{center}
\begin{subfigure}[b]{\figW}
	\centering
	\begin{tcolorbox}[colframe=black, colback=white, boxrule=0.45pt]
		\includegraphics[width=\linewidth]{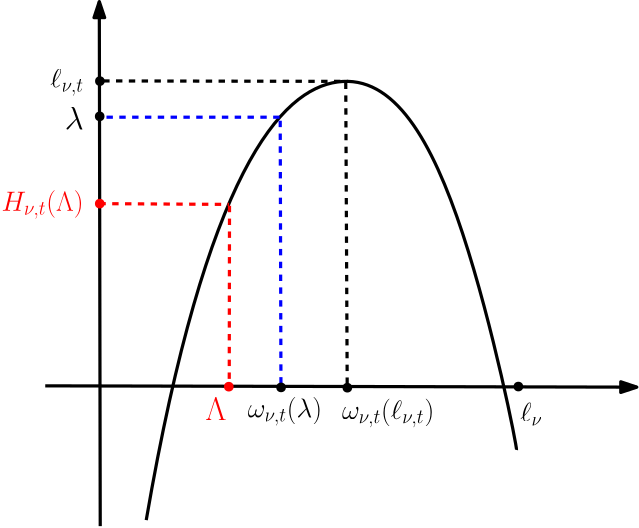}
	\end{tcolorbox}
	\subcaption{$\Lambda< \omega_{\nu,t}(\lambda)$}
\end{subfigure}
\end{center}

\caption{The graph of the function $H_{\nu,t}$}
\label{figure:3cases bis}
\end{figure}
The condition $\Lambda \ge \omega_{\nu,t}(\lambda)$ corresponds to cases $\mathrm{(a)}$ and $\mathrm{(b)}$ in Figure \ref{figure:3cases bis}. A central step in the proof of Lemma \ref{lemma:boundaryV} is to characterize the minimizers of $K_{\nu,t}(\Lambda,\lambda,\cdot)$, equivalently those of the function $J_{\nu,t}^{\Lambda}(\lambda,\cdot)$ defined in \eqref{def:JnutL}. In cases $\mathrm{(a)}$ and $\mathrm{(b)}$, the minimizer $Y$ of $J_{\nu,t}^{\Lambda}(\lambda,\cdot)$ can be identified explicitly; this relies crucially on the fact that $\Phi(\Lambda,Y)\geq \omega_{\nu,t}(\ell_{\nu,t})$, where $\Phi(\Lambda,Y)$ is given by \eqref{def:PhiL}. Once this minimizer is determined, we complete the proof of Lemma \ref{lemma:boundaryV}.

\begin{lemma}\label{lemma:opt J}
Fix $\lambda\in (-\infty,\ell_{\nu,t}]$. {Let $H_{\nu,t}(\Lambda)$ be as in \eqref{def:rhonut} and $\omega_{\nu,t}^*$ be as in \eqref{def:omega*}.}
\begin{enumerate}
 \item {For every $\Lambda\leq \ell_{\nu}$} and 
 with either $\gamma=\omega_{\nu,t}^*(\lambda)$ or $\gamma=\Lambda$, let {$Y(\gamma):=(y_{0}(\gamma),y_{1}(\gamma),\ldots, y_{p}(\gamma))$} be given by
\begin{equation}\label{def:mug}
    y_{i}(\gamma):=\frac{t {
   \alpha_{i}}}{(\eta_i-\omega_{\nu,t}(\lambda))(\eta_i-\gamma)}, \quad {\text{for {every} $1\le i\le p$}},
\end{equation}
{and {$y_0(\gamma)=1-\sum_{i=1}^{p}y_{i}(\gamma)$}. }{The local minimizers of $J^\Lambda_{\nu,t}(\lambda,.)$ are included in $\{Y(\omega_{\nu,t}^*(\lambda)),Y(\Lambda)\}$ whenever $y_0(\Lambda)\geq 0$ or equal to {
$Y(\omega_{\nu,t}^*(\lambda))=(0, (y_{i}(\omega_{\nu,t}^*(\lambda)))_{1\le i\le p})$} whenever $y_0(\Lambda)<0$. Moreover, {for $\gamma=\omega_{\nu,t}^*(\lambda)$ and $\gamma=\Lambda$},} we have
\begin{equation}\label{eq:valueYK}
 K_{\nu,t}(\Lambda,\lambda,Y(\gamma))=I_{\nu,t}^{\Lambda}(\lambda,\gamma).
\end{equation}
\item
The function $J_{\nu,t}^{\Lambda}(\lambda,\cdot)$ admits a unique minimizer $Y\in \Delta^{p+1}$ in  the following cases:
\begin{itemize}
\item If 
 $\Lambda\le \omega_{\nu,t}(\ell_{\nu,t})$ and $\lambda\leq H_{\nu,t}(\Lambda)$, 
 or $\Lambda\geq \omega_{\nu,t}(\ell_{\nu,t})$ and $\lambda\le H_{\nu,t}(\Lambda)$, or {$\Lambda<\omega_{\nu,t}(\ell_{\nu,t})$} and $\lambda=\ell_{\nu,t}$,
then the minimizer is given by $Y(\Lambda)$ {and}
\begin{equation*}
    y_0(\Lambda)=\begin{cases}
        1+tG_\nu'(\Lambda) & \text{if $\omega_{\nu,t}(\lambda)=\Lambda$}\\
        \frac{\lambda-H_{\nu,t}(\Lambda)}{\omega_{\nu,t}(\lambda)-\Lambda} & \text{if $\omega_{\nu,t}(\lambda)\neq \Lambda$}.
    \end{cases}
\end{equation*}
\item If $\Lambda\geq \omega_{\nu,t}(\ell_{\nu,t})$ and $\lambda\in (H_{\nu,t}(\Lambda),\ell_{\nu,t}]$,
then the unique minimizer is $Y({\omega_{\nu,t}^*(\lambda)})$ {with $y_0(\omega_{\nu,t}^*(\lambda))$ equal to zero}. 
\end{itemize}
\end{enumerate}
\end{lemma}

\medskip

\begin{proof}

{\bf{Step 1: existence of a minimizer.}}
By Cauchy-Schwarz inequality, since $t> 0$ and $Y\in \Delta^{p+1}$, 
\begin{equation*}
    J_{\nu,t}^{\Lambda}(\lambda,Y)\geq -\frac{\lambda}{2t}\Bigr(\Lambda y_0+\sum_{k=1}^p\eta_k y_k\Bigr)+\frac{1}{4t}\Bigr(\sum_{k=1}^p\eta_k^{2} y_k+\Lambda^2 y_0\Bigr)-\frac{1}{2}\sum_{k=1}^p \alpha_{k}\log (y_k/\alpha_{k})\,.
\end{equation*}
Thus, since the last  term is non-negative by Jensen's inequality, one can see by completing the square that 
\begin{equation*}
    {J_{\nu,t}^{\Lambda}(\lambda,Y)}\geq \frac{1}{4t}\Bigr(\sum_{k=1}^{p} ( \eta_{k}-\lambda)^2 y_{k}+(\Lambda-\lambda)^2 y_0\Bigr)-\frac{1}{4 t} \lambda^{2}\ge -\frac{1}{4 t} \lambda^{2}.
    \end{equation*}
 {Therefore $J_{\nu,t}^\Lambda(x,\cdot)$ is uniformly bounded from below on $\Delta^{p+1}$.}
By Lemma \ref{lemma:Dirichlet},
 $Y\in(1-y_{0})\Delta^{p}\mapsto I_{\nu}(Y)$ is lower semicontinuous. Moreover,\begin{equation*} Y\in (1-y_{0})\Delta^{p}
    \mapsto \sum_{k=1}^{p} \eta_{k}y_{k}\quad \text{ and}\quad  Y\in ({1-y_0})\Delta^{p}\mapsto \sum_{k=1}^{p}\eta_{k}^2y_{k}
\end{equation*}
are continuous. 
Thus, $J_{\nu,t}^{\Lambda}(x,\cdot)$ is lower semicontinuous  and since it is bounded from below, {it achieves its minimal value.} 
The function $J_{\nu,t}^{\Lambda}(x,\cdot)$ can be differentiated and its critical points analyzed. We  describe them in the following.

\bigskip

{\bf{Step 2: interior critical points.}}
Let $Y\in \Delta^{p+1}$ be a critical point of $J_{\nu,t}^{\Lambda}(x,\cdot)$ such that $y_0>0$. By variational calculus, there exist constants $R_{1},R_{2}$ such that  for every 
$k\in \{1,\ldots,p\}$\begin{equation*}
    -x \eta_{k}+\eta_{k}^2-R_1 \eta_{k} -\frac{t \alpha_{k}}{y_{k}}=R_2
\end{equation*}
and
\begin{equation*}
    -x \Lambda+\Lambda^2-R_1\Lambda=R_2.
\end{equation*}
Therefore, {for every $k=1,\ldots,p$,}
\begin{equation*}
y_{k}=\frac{t\alpha_{k}}{\eta_{k}^2-x \eta_{k}-R_1\eta_{k}-R_2},
\end{equation*}
where $R_{1}$ and $R_{2}$ are such that
\begin{equation}\label{eq:twocond2}
    \begin{cases}
       y_0+\sum_{i=1}^{p}y_{i}=1\\
       \Lambda y_0+ \sum_{i=1}^{p} \eta_i y_{i}=R_1\\
       -x \Lambda+\Lambda^2-R_1\Lambda=R_2.
    \end{cases}
\end{equation}
One may write $Y$ in the form
\begin{equation}\label{eq:formyk}
   y_{k}=\frac{t\alpha_{k}}{(\eta_{k}-\gamma_1)(\eta_{k}-\gamma_2)} {\quad \text{for {every} $k=1,\ldots,p$}}
\end{equation}
with constants $\gamma_1, \gamma_2\in \mathbb{C}$ satisfying $\gamma_1+\gamma_2=x+R_1$ and $\gamma_1\gamma_2=-R_2$. The last equation in \eqref{eq:twocond2} therefore gives
\begin{equation*}
    (\Lambda-\gamma_1)(\Lambda-\gamma_2)=0.
\end{equation*}
It follows that $\gamma_1=\Lambda$ or $\gamma_2=\Lambda$ and $\gamma_1, \gamma_2\in \dR$.  We assume without loss of generality that $\gamma_{1}\le \gamma_{2}$. For $\gamma\in\{\gamma_{1},\gamma_{2}\}$ we set {{for every} $k=1,\ldots,p$},
\begin{equation}\label{def:falpha}
    y_{k}(\gamma)=\frac{t\alpha_{k}}{(\eta_{k}-\Lambda)(\eta_{k}-\gamma)}, \qquad y_{0}(\gamma)=1-\sum_{k=1}^{p} y_{k}(\gamma)\,.
\end{equation}
Then, the previous argument shows that $y$ is either equal to $y(\gamma_{1})$ or $y(\gamma_{2})$. Note that $y_{k}(\gamma)$ is nonnegative  for every $k$ {such that} $\gamma$ belongs to $(-\infty,\ell_{\nu})$ since $\Lambda\leq \ell_{\nu}\le \eta_{k}$. If $\gamma\neq \Lambda$, the first equation of \eqref{eq:twocond2} 
gives
\begin{equation}\label{eq:mean}
\sum_{k=1}^{p} y_{k}(\gamma)=   \frac{t}{\Lambda-\gamma}(G_\nu(\gamma)-G_\nu(\Lambda))=1-y_0(\gamma),
\end{equation}
otherwise if $\gamma=\Lambda$, we have
\begin{equation*}
    1+t G_\nu'(\Lambda)=y_0(\Lambda).
\end{equation*}
As a consequence,
\begin{equation}\label{def:y0alpha}
    y_0(\gamma):=\begin{cases}
        1+tG_\nu'(\Lambda) & \text{if $\gamma=\Lambda$}\\
        \frac{H_{\nu,t}(\gamma)-H_{\nu,t}(\Lambda)}{\gamma-\Lambda} & \text{if $\gamma\neq \Lambda$}.
    \end{cases}
\end{equation}
Furthermore the second equation of \eqref{eq:twocond2} gives 
\begin{align}
\gamma+\Lambda-x= R_{1}&= \Lambda y_{0}(\gamma) +\frac{t}{\gamma-\Lambda} \sum_{k=1}^{p} \eta_{k} \alpha_{k}\Bigl( \frac{1}{\eta_{k}-\gamma}-\frac{1}{\eta_{k}-\Lambda}\Bigr)\nonumber\\
&= \Lambda y_{0}(\gamma)  -\frac{t}{\gamma-\Lambda} \Bigl( \gamma G_{\nu}(\gamma)-\Lambda G_{\nu}(\Lambda)\Bigr)\nonumber\\
&= \Lambda y_{0}(\gamma) -t G_{\nu}(\gamma)-\frac{t\Lambda}{\gamma-\Lambda}(G_{\nu}(\gamma)-G_{\nu}(\Lambda))\nonumber\\
&= \Lambda-t G_{\nu}(\gamma)\label{dfg}\end{align}
where we finally used \eqref{eq:mean}. Thus we find

\begin{equation}\label{eq:equation alpha}
H_{\nu,t}(\gamma)=x,\quad {\gamma < \ell_{\nu}}.
\end{equation}
By Remark \ref{remark:twosol}, whenever $x<\ell_{\nu,t}$, the above equation has two distinct solutions $\gamma_1<\gamma_2$ in $(-\infty,\ell_{\nu})$ and moreover $\gamma_1=\omega_{\nu,t}(x)$ and $\gamma_{2}=\omega^{*}_{\nu,t}(x)$.
If $x=\ell_{\nu,t}$, then Equation \eqref{eq:equation alpha} has a unique solution $\gamma=\omega_{\nu,t}(\ell_{\nu,t})$. 
 Observe that if $\Lambda>\omega_{\nu,t}(\ell_{\nu,t})$ and $x>H_{\nu,t}(\Lambda)$, then for $\gamma\in \{\gamma_1,\gamma_2\}$, $y_0(\gamma)<0$ by \eqref{def:y0alpha}, showing that $Y({\gamma})$ is not a critical point in $\Delta^{p+1}$. Let us now assume that $\Lambda\leq \omega_{\nu,t}(\ell_{\nu,t})$ or $\Lambda\geq \omega_{\nu,t}(\ell_{\nu,t})$ and $x\leq H_{\nu,t}(\Lambda)$. We check whether $y_{0}({\gamma})$ belongs to $[0,1]$.
Note that we always have $y_0({\gamma})<1$. Moreover, in the case $\gamma=\Lambda$,
$$y_{0}=1+t G'_{\nu}(\Lambda),$$ 
which is strictly positive if and only if $\Lambda<\omega_{\nu,t}(\ell_{\nu,t})$. Therefore if $x=H_{\nu,t}(\Lambda)$ ({giving $\gamma_1=\gamma_2=\Lambda$}) and $\Lambda<\omega_{\nu,t}(\ell_{\nu,t})$, then $y_0\in (0,1)$. If ${\Lambda\notin \{\gamma_1,\gamma_2\}}$, then Equation \eqref{def:y0alpha} shows that 
$$y_{0}({\gamma})=\frac{x-H_{\nu,t}(\Lambda)}{\gamma-\Lambda}$$
and $y_{0}({\gamma})$ is positive  only if

\begin{itemize}
\item $\Lambda<\omega_{\nu,t}(\ell_{\nu,t})$ (since $H_{\nu,t}$ is increasing on $(-\infty,\omega_{\nu,t}(\ell_{\nu,t}))$),
 \item or $\Lambda>\omega_{\nu,t}(\ell_{\nu,t})$ and $x< H_{\nu,t}(\Lambda)$.
 \end{itemize} 
 {Note that $y_{0}(\gamma)\le 1$ by Definition \eqref{eq:mean} since $\gamma$ and $\Lambda$ are smaller than or equal to $\ell_{\nu}$. }
Hence, in this range of parameters $Y({\gamma_{i}})_{i=1,2}$  are critical points in $\Delta^{p+1}$, but none of them belong to $\Delta^{p+1}$ otherwise. In the following we compute the value of $J_{\nu,t}^{\Lambda}(x,\cdot)$ {for each} of these two solutions in order to select the minimizer.
Let $\gamma\in\{\gamma_1,\gamma_2\}$. One  computes as in \eqref{dfg} 
\begin{equation}\label{eq:value}
M(\gamma):=\Lambda y_{0}(\gamma)+\sum_{i=1}^{p} \eta_{i} y_{i}(\gamma)=\Lambda-tG_\nu(\gamma)=\Lambda+\gamma-x,
\end{equation}
Moreover,  
\begin{equation}
\begin{split}\sigma(\gamma)&:=   \Lambda^2 y_0(\gamma)+\sum_{k=1}^{p} \eta_k^2 y_{k}(\gamma)\nonumber\\
     &=\Lambda^2 y_0(\gamma)+t\sum_{k=1}^{p} \frac{\alpha_{k} }{(\eta_{k}-\gamma)(\eta_{k}-\Lambda)}  \left( (\eta_{k}-\gamma)(\eta_{k}-\Lambda) +(\gamma+\Lambda)\eta_{k}-\gamma\Lambda\right)
     \nonumber\\
     &=\Lambda^2 y_0(\gamma)+
  t+(\gamma+\Lambda)\sum \eta_{k}y_{k}(\gamma)
     -\gamma \Lambda(1-y_0(\gamma))\nonumber\\
        &=t+(\gamma+\Lambda)M(\gamma)-\gamma \Lambda.\label{eq:sigma}
\end{split}
\end{equation}{
Moreover, since $\sum \alpha_{k}=1$, 
\begin{equation}\label{eq:I}
I_{\nu}(Y(\gamma))=-\frac{1}{2}\sum_{k=1}^{p}\alpha_{k}\ln\left( \frac{t}{(\eta_{k}-\gamma)(\eta_{k}-\Lambda)} \right)=-\frac{1}{2}\left( S_{\nu}(\gamma)+S_{\nu}(\Lambda)\right)
-\frac{1}{2}\ln t\,.\end{equation}
Thus, we find
\begin{equation}\label{eq:phicase1}
\begin{split}
  J_{\nu,t}^{\Lambda}(x, Y(\gamma)
  )&=L^{\Lambda}_{\nu,t}(x,Y(\gamma))+I_{\nu}(Y(\gamma))=-\frac{x}{2t}M(\gamma)+\frac{1}{2t}\sigma(\gamma)-\frac{1}{4t}M(\gamma)^{2}+I_{\nu}(Y(\gamma))
  \\
  &
   =\frac{1}{4t}(-2x-M(\gamma)+2(\gamma+\Lambda))M(\gamma)+\frac{1}{2}-\frac{\gamma\Lambda}{2t}-\frac{1}{2}(S_\nu(\Lambda)+S_\nu(\gamma))-\frac{1}{2}\ln t\\
   & =\frac{1}{4t}((\gamma+\Lambda-x)^2-2\gamma\Lambda)+\frac{1}{2}-\frac{1}{2}(S_\nu(\Lambda)+S_\nu(\gamma))-\frac{1}{2}\ln t\\&=\frac{1}{4t}(\gamma^2+\Lambda^2+x^2-2\gamma x-2\Lambda x)+\frac{1}{2}-\frac{1}{2}(S_\nu(\Lambda)+S_\nu(\gamma))-\frac{1}{2}\ln t.
\end{split}
\end{equation}}
Define the map
\begin{equation*}
    \phi:\gamma\in (-\infty,\ell_{\nu})\mapsto J(x,Y(\gamma))= \frac{1}{4t}(\gamma^2+\Lambda^2+x^2-2\gamma x-2\Lambda x)+\frac{1}{2}-\frac{1}{2}(S_\nu(\Lambda)+S_\nu(\gamma))-\frac{1}{2}\ln t.
\end{equation*}
One {can compute} that for all $\gamma\in (-\infty,\ell_{\nu})$,
\begin{equation*}
    \phi'(\gamma)=\frac{1}{2t}(H_{\nu,t}(\gamma)-x).
\end{equation*}
Therefore $\phi$ is strictly increasing on $[\gamma_1,\gamma_2]$ and {therefore since we assumed  $x\neq \ell_{\nu,t}$,}
\begin{equation*}
{J_{\nu,t}^\Lambda(x,Y({\gamma_1}))=\phi(\gamma_1)<\phi(\gamma_2)=J_{\nu,t}^\Lambda(x, Y({\gamma_2}))}.
\end{equation*}
We deduce that when {$x\neq \ell_{\nu,t}$,} $Y({\gamma_2})$ is not {a minimizer} of $J_{\nu,t}^{\Lambda}(x,\cdot)$. 

\bigskip  

{\bf{Step 3: study of the critical points on the boundary. }} We next consider the critical points so that $y_{0}$ equal to zero. Let $Y\in\Delta^{p}$ be a critical point of ${J_{\nu,t}^{\Lambda}(x,0,\cdot)}$. By variational calculus, there exist $R_1, R_2$ such that
\begin{equation*}
    -x \eta_k+\eta_k^2-R_1\eta_k-\frac{t\alpha_{k}}{y_{k}}=R_2,{\quad \text{{for every} $1\le k\le p$}},
\end{equation*}
where $R_1=\sum \eta_k y_{k} $. Therefore
\begin{equation}\label{def:mustep4}
    y_{k}=\frac{t \alpha_{k}}{(\eta_k-\gamma_1)(\eta_k-\gamma_2)},{\quad \text{{for every} $1\le k\le p$}}
\end{equation}
where $\gamma_1,\gamma_2\in \mathbb{C}$ satisfy $\gamma_1+\gamma_2=x+R_1$ and $\gamma_1\gamma_2=-R_2.$ The equation $\sum y_{k}=1$ gives
\begin{equation}\label{eq:g1g2b}
    H_{\nu,t}(\gamma_1)=H_{\nu,t}(\gamma_2).
\end{equation}
Moreover
\begin{equation*}
    \sum_{k=1}^{p} \eta_k y_{k}=-tG_\nu(\gamma_2)+\gamma_1=R_1=\gamma_1+\gamma_2-x.
\end{equation*}
Inserting \eqref{eq:g1g2b}, this shows that
\begin{equation*}
    H_{\nu,t}(\gamma_1)=H_{\nu,t}(\gamma_2)=x.
\end{equation*}
Therefore $\gamma_1=\omega_{\nu,t}(x)$ and $\gamma_2=\omega_{\nu,t}^*(x)$ {where $\omega_{\nu,t}^*(x)$ is as in \eqref{def:omega*}.} Let us now compute $J_{\nu,t}^{\Lambda}(x,0,Y)$. {Let $M(\gamma):=\sum \eta_k y_{k}(\gamma)=\gamma_1+\gamma_2-x$}. We have
\begin{equation*}
{\sum_{k=1}^p} \eta_k^2 y_{k}=t+(\gamma_1+\gamma_2)M(\gamma)-\gamma_1\gamma_2.
\end{equation*}
Therefore we get
\begin{equation}
\begin{split}
    J_{\nu,t}^{\Lambda}(x,0,Y)&=\frac{1}{4t}(-2x-M(\gamma)+2(\gamma_1+\gamma_2))M(\gamma)+\frac{1}{2}-\frac{\gamma_1\gamma_2}{2t}-\frac{1}{2}(S_\nu(\gamma_1)+S_\nu(\gamma_2))-\frac{1}{2}\ln t\\
&=\frac{1}{4t}(\gamma_1^2+\gamma_2^2+x^2-2\gamma_1x-2\gamma_2x)+\frac{1}{2}-\frac{1}{2}(S_\nu(\gamma_1)+S_\nu(\gamma_2))-\frac{1}{2}\ln t.\label{eqJ}
    \end{split}
\end{equation}

{\bf{Step 4: Identification of the {minimizer}. }}\\

$\bullet$ If $\Lambda>\omega_{\nu,t}(\ell_{\nu,t})$ and $x>H_{\nu,t}(\Lambda)$, then by Step 2, there is {no critical point} with $y_{0}>0$ and the only minimizer is 
therefore given by $Y({\omega^{*}_{\nu,t}(x)})$ where ${Y(\gamma)}$ is as defined in \eqref{def:mug}.

$\bullet$ If {$\Lambda\leq \omega_{\nu,t}(\ell_{\nu,t})$ and $x\leq H_{\nu,t}(\Lambda)$, or $\Lambda\geq \omega_{\nu,t}(\ell_{\nu,t})$ and $x\leq H_{\nu,t}(\Lambda)$,} then the minimizer is either $Y(\gamma_{1})$ defined in \eqref{def:falpha} or $(0,Y)$ defined in \eqref{def:mustep4}. We claim it is $Y(\gamma_{1})$. In fact, recall
\begin{equation*}
  J_{\nu,t}^{\Lambda}(x,0,Y)=\frac{1}{4t}(\gamma_1^2+\gamma_2^2+x^2-2\gamma_1x-2\gamma_2x)+\frac{1}{2}-\frac{1}{2}(S_\nu(\gamma_1)+S_\nu(\gamma_2))-\frac{1}{2}\ln t,
\end{equation*}
{while}
\begin{equation*}
  J_{\nu,t}^{\Lambda}(x,Y(\gamma_1))=\frac{1}{4t}(\gamma_1^2+\Lambda^2+x^2-2\gamma_1x-2\Lambda x)+\frac{1}{2}-\frac{1}{2}(S_\nu(\gamma_1)+S_\nu(\Lambda))-\frac{1}{2}\ln t{.}
\end{equation*}
Define
\begin{equation*}
    \psi:\gamma\in (-\infty,\ell_{\nu}]\mapsto \frac{1}{4t}(\gamma_1^2+\gamma^2+x^2-2\gamma_1x-2\gamma x)+\frac{1}{2}-\frac{1}{2}(S_\nu(\gamma_1)+S_\nu(\gamma))-\frac{1}{2}\ln t
\end{equation*}
Note that for all $\gamma\in (-\infty,\ell_{\nu})$, 
\begin{equation*}
    \psi'(\gamma)=\frac{1}{2t}(H_{\nu,t}(\gamma)-x).
\end{equation*}
Hence $\psi$ is strictly increasing on $[\gamma_1,\gamma_2]$. {Recall that $\Lambda\in [\gamma_{1},\gamma_{2}]$. Therefore if $\Lambda<\gamma_2$, then
\begin{equation*}
    J_{\nu,t}^{\Lambda}(x,0,Y)=\psi(\gamma_2)>\psi(\Lambda)=J_{\nu,t}^{\Lambda}(x,Y(\gamma_{1}))
\end{equation*}
which implies that $Y({\gamma_{1}})$ is the unique minimizer of $J_{\nu,t}^{\Lambda}(x,\cdot)$. If $\Lambda=\gamma_2$, then $Y(\gamma_1)$ is again the unique minimizer of $J_{\nu,t}^{\Lambda}(x,\cdot)$.}
We can also conclude in the case where  $\Lambda<\omega_{\nu,t}(\ell_{\nu,t})$ and $x=\ell_{\nu,t}$. {In this case,} one can notice that $\gamma_2=\gamma_1=\omega_{\nu,t}(\ell_{\nu,t})$. Moreover $\psi$ is strictly increasing on $(-\infty,\gamma_1)$. Therefore
\begin{equation*}
    J_{\nu,t}^{\Lambda}(x,0,Y)=\psi(\gamma_1)>J_{\nu,t}^{\Lambda}(x,Y(\Lambda))=\psi(\Lambda).
\end{equation*} 
Therefore $Y(\Lambda)$ is the unique minimizer of $J_{\nu,t}^{\Lambda}(x,\cdot)$.

\bigskip

{\bf{Step 5: Proof of \eqref{eq:valueYK}}.}
Let $Y(\gamma)$ be as in \eqref{def:mug}. Recall that by \eqref{eq:phicase1} and \eqref{eqJ}, for $\gamma\in \{\omega_{\nu,t}^*(x),\Lambda\}$,
\begin{equation*}
J_{\nu,t}^\Lambda(x,Y(\gamma))= \frac{1}{4t}(\gamma_1^2+\gamma^2+x^2-2\gamma_1x -2\gamma x)+\frac{1}{2}-\frac{1}{2}(S_\nu(\gamma_1)+S_\nu(\gamma))-\frac{1}{2}\ln t.
\end{equation*}
By Lemma \ref{lemma:exp det}, 
\begin{equation*}
    -\int \log|y-x|\dd (\nu\boxplus\sigma_t)(y)=S_{\nu\boxplus\sigma_t}(x)=S_\nu(\gamma_1)-\frac{1}{2t}(\gamma_1^2+x^2-2x\gamma_1).
\end{equation*}
Therefore, combining this with \eqref{eq:phicase1} and recalling that $C_t=\frac{1}{2}-\frac{1}{2}\ln t$, we obtain
\begin{eqnarray*}
 K_{\nu,t}(\Lambda, x,Y(\gamma))&=&\frac{1}{4t}(\gamma_1^2+\gamma^2+x^2-2\gamma_1x -2\gamma x)-\frac{1}{2}(S_\nu(\gamma_1)+S_\nu(\gamma))+S_\nu(\gamma_1)\\
 &&-\frac{1}{4t}(2\gamma_1^2+2x^2-4x\gamma_1-x^2)\\
& =&\frac{1}{2}(S_\nu(\gamma_1)-S_\nu(\gamma))+\frac{1}{4t}(\gamma^2-\gamma_1^2-2\gamma x+2\gamma_1x)\\
& =&\frac{1}{2}(S_\nu(\gamma_1)-S_\nu(\gamma))+\frac{1}{4t}((\gamma-x)^2-(\gamma_1-x)^2)=I_{\nu,t}^{\Lambda}(x,\gamma),
\end{eqnarray*}
which proves \eqref{eq:valueYK} {by} recalling \eqref{def:ALambda}.
\end{proof}
Before giving the proof of Lemma \ref{lemma:boundaryV} we {establish the following:}
\begin{lemma}\label{lemma:eigenvalue proj}{Let $H_{\nu,t}(\Lambda)$ be as in \eqref{def:rhonut} and $\omega_{\nu,t}^*$ be as in \eqref{def:omega*}.}
\begin{enumerate}
\item Let $x\in (-\infty,\ell_{\nu,t}]$.{ { If  $\Lambda< \omega_{\nu,t}(\ell_{\nu,t})$, or $\Lambda> \omega_{\nu,t}(\ell_{\nu,t})$ and $x< H_{\nu,t}(\Lambda)$, 
let $Y(\Lambda)=(y_0(\Lambda),\ldots,y_p(\Lambda))\in \Delta^{p+1}$ be given by
\begin{equation*}
    y_k(\Lambda)=\frac{t\alpha_{k}}{(\eta_k-\omega_{\nu,t}(x))(\eta_k-\Lambda)},\quad {\text{{for every} $k=1,\ldots,p$}},\quad y_{0}(\Lambda)=1-\sum_{k=1}^p y_{k}(\Lambda)
\end{equation*}
Then $y_0(\Lambda)$ is strictly positive}}.  Moreover,  $\Phi(\Lambda,Y(\Lambda))=\omega_{\nu,t}^*(x)$. 
\item Let $Y=(0,y_1,\ldots,y_p)\in \Delta^{p+1}$.  Then, $\Phi(\Lambda,Y)=\Lambda$.
\end{enumerate} 
\end{lemma}

\medskip

\begin{proof}{Let us prove (1)}. By Lemma \ref{lemma:spectrum} and the proof of the previous lemma, it is enough 
to prove that $\Phi(\Lambda,Y(\Lambda))=\omega_{\nu,t}^*(x)$ where we recall that $\Phi(\Lambda,Y)$ is defined as the smallest solution $\Lambda_{1}$ in \eqref{eq:sm}.  We denote in short
$\alpha=\omega_{\nu,t}(x)$.
Since $y_{0}(\Lambda)\neq 0$,  $\Lambda_1\neq\Lambda$. We claim that also $\Lambda_{1}\neq\alpha$ because \begin{equation}\label{neqa} \sum_{k=0}^p \frac{y_k}{\eta_k-\alpha}\neq 0,\end{equation}
where we denote $\eta_0:=\Lambda$. Indeed, using that for every  real numbers 
$a\neq b$ and $x\notin \{a,b\}$,
\begin{equation}\label{eq:ab}
\frac{1}{(x-a)^2(x-b)} = -\frac{1}{(a-b)^2} \left( \frac{1}{x-a} + \frac{b-a}{(x-a)^2} - \frac{1}{x-b} \right),
\end{equation}
{and setting} $a=\alpha$ and $b=\Lambda$, one gets
\begin{equation*}
\begin{split}
    \sum_{k=0}^p \frac{y_k}{\eta_k-\alpha}&=\frac{y_0}{\Lambda-\alpha}-\frac{t}{(\Lambda-\alpha)^2}(-G_\nu(\alpha)+G_\nu(\Lambda)-(\Lambda-\alpha)G_\nu'(\alpha))\\&=-\frac{1}{(\Lambda-\alpha)^2}(H_{\nu,t}(\alpha)-H_{\nu,t}(\Lambda)-tG_\nu(\alpha)+tG_\nu(\Lambda)-t(\Lambda-\alpha)G_\nu'(\alpha))\\&=\frac{1+tG_\nu'(\alpha)}{\Lambda-\alpha}\,.
\end{split}
\end{equation*}
Since $x<\ell_{\nu,t}$, we have $\alpha=\omega_{\nu,t}(x)<\omega_{\nu,t}(\ell_{\nu,t})$ and therefore $1+tG_\nu'(\alpha)>0$, which yields
\begin{equation*}
      \sum_{k=0}^p \frac{y_k}{\eta_k-\alpha}\neq 0,
\end{equation*}
which proves \eqref{neqa}. 
We can thus take $\Lambda_{1}\notin\{\Lambda,\alpha\}$.
Then, by partial fraction decomposition,
\begin{multline*}
    \sum_{k=1}^p \frac{\alpha_{k}}{(\eta_k-\Lambda_1)(\eta_k-\Lambda)(\eta_k-\alpha)}\\=-\frac{1}{(\Lambda_1-\Lambda)(\Lambda_1-\alpha) }G_\nu(\Lambda_1)-\frac{1}{(\Lambda-\alpha)(\Lambda-\Lambda_1)}G_\nu(\Lambda)-\frac{1}{(\Lambda_1-\alpha)(\Lambda-\alpha)}G_\nu(\alpha),
\end{multline*}
which implies
\begin{equation*}
    \frac{y_0}{\Lambda-\Lambda_1}-t\frac{1}{(\Lambda-\Lambda_1)(\alpha-\Lambda_1) }G_\nu(\Lambda_1)+t\frac{1}{(\alpha-\Lambda)(\Lambda-\Lambda_1)}G_\nu(\Lambda)-t\frac{1}{(\alpha-\Lambda_1)(\alpha-\Lambda)}G_\nu(\alpha)=0.
\end{equation*}
Multiplying by $(\Lambda-\Lambda_1)(\alpha-\Lambda_1)(\alpha-\Lambda)$  gives
\begin{equation*}
    y_0(\alpha-\Lambda_1)(\alpha-\Lambda)-t(\alpha-\Lambda)G_\nu(\Lambda_1)+t(\alpha-\Lambda_1)G_\nu(\Lambda)-t(\Lambda-\Lambda_1)G_\nu(\alpha)=0.
\end{equation*}
Recalling that $y_0(\alpha-\Lambda)=x-H_{\nu,t}(\Lambda)$, this can be simplified into
\begin{align}
&(x-H_{\nu,t}(\Lambda))(\alpha-\Lambda_1)-t(\alpha-\Lambda)G_\nu(\Lambda_1)+t(\alpha-\Lambda_1)G_\nu(\Lambda)-t(\Lambda-\Lambda_1)G_\nu(\alpha)=0 \notag\\
 \Longleftrightarrow \ &(\alpha-\Lambda)(\alpha-\Lambda_1)+tG_\nu(\alpha)(\alpha-\Lambda)-tG_\nu(\Lambda_1)(\alpha-\Lambda)=0 \notag\\
\Longleftrightarrow \ & (\alpha-\Lambda_1)+tG_\nu(\alpha)-tG_\nu(\Lambda_1)=0\notag \\
\Longleftrightarrow \ & H_{\nu,t}(\Lambda_1)=x\label{eq:lambda1'}\,.
\end{align}
Since $\Lambda_{1}\neq \alpha$, it follows that
$\Lambda_1$ equals  $\omega_{\nu,t}^*(x)$. 
{The proof of (2) is immediate.}
\end{proof}

{We can now give the proof of Lemma \ref{lemma:boundaryV}.}

\begin{proof}[Proof of Lemma \ref{lemma:boundaryV}]
 Let us prove the upper bound \eqref{eq:b1}. Let $Y(\gamma)\in \Delta^{p+1}$ be given by  \eqref{def:mug} with $\gamma$ as in \eqref{def:gamma}. By \eqref{eq:valueYK}, we know that $
K_{\nu,t}(\Lambda,x,Y(\gamma))=I_{\nu,t}^{\Lambda}(x).$
Using \eqref{eq:lb}, and {taking} $Y=Y(\gamma)$ in the infimum yields 
\begin{equation*}
   { F_{\nu,t}^-(\Lambda,x)\leq K_{\nu,t}(\Lambda,x,Y(\gamma))+F_{\nu,t}^-(\Phi(\Lambda,Y),x^{+})1_{\omega_{\nu,t}(x)\ge \Phi(\Lambda,Y(\gamma))}.}
\end{equation*}
As shown in Lemma \ref{lemma:eigenvalue proj}, $\Phi(\Lambda,Y(\gamma))\geq \omega_{\nu,t}(\ell_{\nu,t})\geq \omega_{\nu,t}(x)$ so that the last term in the {right-hand} side vanishes. It follows that
\begin{equation*}
    F_{\nu,t}^-(\Lambda,x)\leq K_{\nu,t}(\Lambda,x,Y(\gamma))
    =I_{\nu,t}^{\Lambda}(x)
\end{equation*}
where we finally used \eqref{eq:valueYK}.
Using $F_{\nu,t}^+\leq F_{\nu,t}^-$ we get the desired upper bounds.
Let us now assume that  $\Lambda\geq \omega_{\nu,t}(x)$ or  $\Lambda\le \omega_{\nu,t}(\ell_{\nu,t})$ and $x\leq H_{\nu,t}(\Lambda)$. 
Since $F_{\nu,t}^+\geq 0$, one has 
\begin{equation*}
    F_{\nu,t}^+(\Lambda,x)\geq \inf_{Y} K_{\nu,t}(\Lambda,x,Y).
\end{equation*}
By Lemma \ref{lemma:opt J}, in the range of parameters that we consider, $K_{\nu,t}(\Lambda,x, .)$ and 
 $J_{\nu,t}^{\Lambda}$ {have} a unique minimizer $Y(\gamma)$ with $\gamma=\gamma_\Lambda(x)$. Hence
\begin{equation*}
    F_{\nu,t}^+(\Lambda,x)\geq K_{\nu,t}(\Lambda,x,Y(\gamma))=I_{\nu,t}^{\Lambda}(x),
\end{equation*}
which gives $F_{\nu,t}^+(\Lambda,x)=I_{\nu,t}^{\Lambda}(x)$ using \eqref{eq:b1}. Since $F_{\nu,t}^-(\Lambda,x)\geq F_{\nu,t}^+(\Lambda,x)$ we get that $F_{\nu,t}^-(\Lambda,x)=I_{\nu,t}^{\Lambda}(x)$.
\end{proof}

\subsection{The iterated regime}
{In the next lemma, we prove that the equality \eqref{eq:equality} in fact holds for any $\Lambda\leq\ell_{\nu}$ and $x\leq \ell_{\nu,t}$. In view of Lemma \ref{lemma:boundaryV}, it only remains to prove it for $\Lambda<\omega_{\nu,t}(x)$, i.e., in case $(\mathrm{c})$ of Figure \ref{figure:3cases}. This will complete the proof of the weak large deviation principle \eqref{wldp}.}

\begin{lemma}\label{lemma:F+}
Let $\Lambda\leq \ell_{\nu}$ and $x< \ell_{\nu,t}$. Let $I_{\nu,t}^{\Lambda}(x)$ be as in \eqref{def:ALambda}. We have
\begin{equation*}
F_{\nu,t}^-(\Lambda,x)=F_{\nu,t}^+(\Lambda,x)=I_{\nu,t}^{\Lambda}(x).
\end{equation*}
\end{lemma}
{To prove Lemma \ref{lemma:F+}, we first demonstrate that $I_{\nu,t}^{\Lambda}$ satisfies the functional equation necessary for it to be equal to $F_{\nu,t}^{\pm}$, as suggested by \eqref{eq:ub} and \eqref{eq:lb}.}

\begin{lemma}\label{lemma:solves}
{Let $\Lambda\leq \ell_{\nu}$. Let $\Phi(\Lambda,Y)$ be as in \eqref{eq:sm}, $H_{\nu,t}(\Lambda)$ be as in \eqref{def:rhonut} and $\omega_{\nu,t}^*$ be as in \eqref{def:omega*}. Let $I_{\nu,t}^{\Lambda}$ be as in \eqref{def:ALambda} and $K$ as in \eqref{defK}.} Let $x< \ell_{\nu,t}$. Let $\omega_{\nu,t}^*(x)$ be the unique solution $\gamma$ in $[\omega_{\nu,t}(\ell_{\nu,t}),\ell_{\nu})$ of the equation $H_{\nu,t}(\gamma)=x$. 
Then, the function $G(Y):=K_{\nu,t}(\Lambda,x,Y)+I_{\nu,t}^{\Phi(\Lambda,Y)}(x)1_{\omega_{\nu,t}(x)\ge\Phi(\Lambda,Y)}$ admits a unique minimizer $Y^{*}\in \Delta^{p+1}$. 
\begin{enumerate}
 \item Moreover,
\begin{equation*}
   \Phi(\Lambda,Y^{*})\geq \omega_{\nu,t}(\ell_{\nu,t})
\end{equation*}
and 
\begin{equation*}
   I_{\nu,t}^{\Lambda}(x)= \inf_{Y\in \Delta^{p+1}} G(Y)=K_{\nu,t}(\Lambda,x,Y^{*}).
\end{equation*}
\item Furthermore,
\begin{enumerate}
\item If 
 $\Lambda\le \omega_{\nu,t}(\ell_{\nu,t})$, 
 or $\Lambda\geq \omega_{\nu,t}(\ell_{\nu,t})$ and $x\le H_{\nu,t}(\Lambda)$,
then $Y^{*}=Y(\Lambda)$ {with $Y(\Lambda)$ as defined in \eqref{def:mug}},
\item If $\Lambda\geq \omega_{\nu,t}(\ell_{\nu,t})$ and $x\in (H_{\nu,t}(\Lambda),\ell_{\nu,t}]$,
then $y_{0}^{*}$ vanishes. Moreover, $Y^{*}=Y(\omega^{*}_{\nu,t}(x))$ {with $Y(\omega^{*}_{\nu,t}(x))$ as defined in \eqref{def:mug}}. 
\end{enumerate}
\end{enumerate}
\end{lemma}

\medskip 
\begin{remark}
\begin{itemize}
\item
 In view of our proof, conditionally on  $\lambda_{1}$ being in a $\delta$-neighborhood of $x$,  $Y(v_{1})$ concentrates in a $\varepsilon_{\delta}$- neighborhood of the minimizers of $K_{\nu,t}(\Lambda,x,\cdot)+I_{\nu,t}^{\Phi(\Lambda,\cdot)}(x)1_{\omega_{\nu,t}(x)\ge\Phi(\Lambda,\cdot)} $, namely $Y^{*}$, for some $\varepsilon_{\delta}$ going to zero with $\delta$. Indeed, 
 the contribution of the complement of this set  is exponentially small.  Obtaining large deviations for the joint distribution of $(\lambda_{1}, Y(v_{1}))$ may be possible but would require additional work.

\item When $\omega_{\nu,t}(x)=\Phi(\Lambda,Y)$, $I_{\nu,t}^{\Phi(\Lambda,Y)}(x)=0$ so that in the definition of $G$ we can equivalently take the indicator function $\omega_{\nu,t}(x)\ge \Phi(\Lambda,Y)$ or $\omega_{\nu,t}(x)>\Phi(\Lambda,Y)$.
\end{itemize}
\end{remark}

\begin{proof}
Since $I_{\nu,t}^{\Phi(\Lambda,Y)}\geq 0$, we deduce from the first step of the proof of Lemma \ref{lemma:opt J} that $G$ is bounded from below (by $K_{\nu,t}$). We claim that it is lower semicontinuous  in $Y\in \Delta^{p+1}$. We already checked in the previous proof that $J_{\nu,t}^{\Lambda}(x,.)$ is lower semicontinuous, and so is $K_{\nu,t}(\Lambda,x,.)$ by Definition  \eqref{defK}. 
Moreover $Y\mapsto \Phi(\Lambda,Y)$ is continuous on $y_{1}\ge \kappa$   for every $\kappa>0$ by Lemma \ref{lemma:spectrum}. Recall that when $y_{1}$ goes to zero, $\Phi( \Lambda,Y)$ goes to $\eta_{1}$ which is larger than  $\omega_{\nu,t}(x)$ unless $\ell^{\Lambda}_{\nu,t}=\ell_{\nu,t}=\eta_{1}$ and $x$ goes to $\ell_{\nu,t}$.  Since we assume $x<\ell_{\nu,t}$, we may and shall assume $y_{1}>\kappa_{x}$ for some $\kappa_{x}>0$ and hence $\Phi(\Lambda,.)$ continuous. 
Furthermore, when $\Phi(\Lambda,Y)\le \omega_{\nu,t}(x)$, we have by Lemma \ref{lemma:boundaryV}
\begin{equation}\label{defI}
I_{\nu,t}^{\Phi(\Lambda,Y)}(x)= \frac{1}{2}(S_{\nu}(\omega_{\nu,t}(x))-S_\nu(\Phi(\Lambda,Y)))+\frac{1}{4t}((x-\Phi(\Lambda,Y))^2-(x-\omega_{\nu,t}(x))^2){.}\end{equation}
Since $S_{\nu}$ is continuous on $(-\infty,\ell_{\nu})$ and   $\Phi(\Lambda,Y)\le \ell_{\nu,t}\le \ell_{\nu}$, the lower semi-continuity of $$Y\mapsto I_{\nu,t}^{\Phi(\Lambda,Y)}(x)$$ follows.  The lower semi-continuity of
$Y\mapsto I_{\nu,t}^{\Phi(\Lambda,Y)}(x)1_{\Phi(\Lambda,Y)<\omega_{\nu,t}(x)}$ then follows from the fact that $I_{\nu,t}^{\Phi(\Lambda,Y)}(x)$ vanishes {when} $\Phi(\Lambda,Y)=\omega_{\nu,t}(x)$. 
 It follows that $G$ is bounded from below and lower semicontinuous , and hence achieves its minimal value. {Next, we characterize the minimizers of $G$.}
\medskip

\paragraph{\bf{Step 1: study of the {local minima} such that $\Phi(\Lambda,\cdot)\ge \omega_{\nu,t}(x)$. }}
{On the set $\{Y\in \Delta^{p+1}:\Phi(\Lambda,Y)\ge \omega_{\nu,t}(x)\}$, the last term in $G$ vanishes and therefore the minimizers of $G$ on this set are the minimizers $Y$ of $J_{\nu,t}^{\Lambda}(x,\cdot)$ that satisfy
$\Phi(\Lambda,Y)\ge \omega_{\nu,t}(x)$.} 

{The minimizers of $J_{\nu,t}^\Lambda(x,\cdot)$ are in $\{Y(\Lambda),Y(\omega_{\nu,t}^*(x))\}$ by Lemma \ref{lemma:opt J}. In view of Lemma \ref{lemma:eigenvalue proj}, $\Phi(\Lambda,Y(\Lambda))=\omega_{\nu,t}^{*}(x)\geq \omega_{\nu,t}(x)$. Besides $\Phi(\Lambda,Y(\omega_{\nu,t}^*(x)))=\Lambda$. Therefore the possible minimizers of $G$ on the set $\{Y\in \Delta^{p+1}:\Phi(\Lambda,Y)\ge \omega_{\nu,t}(x)\}$ are $Y(\Lambda)$, and $Y(\omega_{\nu,t}^*(x))$ whenever $\Lambda\geq \omega_{\nu,t}(x)$. Additionally, \eqref{eq:valueYK} shows that the value   of $G(Y({\gamma}))$ for $\gamma\in \{\omega_{\nu,t}^*(x),\Lambda\}$ is $I_{\nu,t}^{\Lambda}(x,\gamma)$.} 
\medskip

\paragraph{\bf{Step 2: study of the critical points such that $\Phi(\Lambda,\cdot)< \omega_{\nu,t}(x)$ and $y_0>0$}}
Let $Y\in \Delta^{p+1}$ be a critical point of $G$ such that $y_0>0$ and
\begin{equation}\label{eq:Phi1a}
\Phi(\Lambda,Y)<\omega_{\nu,t}(x).
\end{equation}
By variational calculus similar to the proof of Lemma \ref{lemma:opt J}, but with the additional term  $I^{\Phi(\Lambda, Y)}_{\nu,t}(x)$ given by \eqref{defI}, 
there exist real numbers $R_1, R_2$ and $C_0$ such that for ${k\in\{1,\ldots,p\}}$, 
\begin{equation*}
    -x \eta_k+\eta_k^2-R_1\eta_k-\frac{t\alpha_{k}}{y_{k}}-C_0 \frac{1}{\eta_k-\Lambda_1}=R_2
\end{equation*}
and
\begin{equation}\label{eq:Lambdaroot}
    -x \Lambda+\Lambda^2-R_1\Lambda-\frac{C_0}{\Lambda-\Lambda_1}=R_2,
\end{equation}
with $R_1=\Lambda y_0+\sum \eta_{k}y_{k}$, and $\Lambda_{1}=\Phi(\Lambda,Y)$. The new term in $C_{0}$ comes from the derivative of $I_{\nu,t}^{\Phi(\Lambda,Y)}$ and 
the remark that 
$$\partial_{y_{k}}( I_{\nu,t}^{\Phi(\Lambda,Y)}(x))= \partial_{\Lambda'}I_{\nu,t}^{\Lambda'}(x)|_{\Lambda'=\Phi(\Lambda,Y)} \partial_{y_{k}} \Phi(\Lambda,Y)$$
whereas by Equation \eqref{eq:sm},
$$\partial_{y_{k}} \Phi(\Lambda,Y)=\left( \frac{y_{0}}{(\Lambda-  \Phi(\Lambda,Y))^{2}}-\sum_{k=1}^{p}\frac{y_{k}}{(\eta_{k}- \Phi(\Lambda,Y))^{2}}\right)^{-1}\frac{1}{ \Phi(\Lambda,Y)-\eta_{k}}=:\frac{C_{0}}
{\partial_{\Lambda'}I_{\nu,t}^{\Lambda'}(x)|_{\Lambda'=\Phi(\Lambda,Y)} }\frac{1}{ \Lambda_{1}-\eta_{k}}
\,.$$
It follows that {{for every} $k=1,\ldots,p$,}
\begin{equation}\label{defy}
    y_{k}=\frac{t\alpha_{k}}{\eta_k^2-(x+R_1)\eta_k-R_2-\frac{C_0}{\eta_k-\Lambda_1}}=\frac{t\alpha_{k}(\eta_k-\Lambda_1)}{(\eta_k^2-(x+R_1)\eta_k-R_2)(\eta_k-\Lambda_1)-C_0}. 
\end{equation}
According to \eqref{eq:Lambdaroot}, $\Lambda$ is one of the roots of the denominator  in the last display. 
Let us write this  denominator as
\begin{equation}\label{defiP}
P(\eta_{k}):= (\eta_k^2-(x+R_1)\eta_k-R_2)(\eta_k-\Lambda_1)-C_0=(\eta_k-\gamma_1)(\eta_k-\gamma_2)(\eta_k-\Lambda),   
\end{equation}
with $\gamma_2, \gamma_1\in \mathbb{C}$. Identifying the coefficient in front of $\eta_{k}^{2}$, {we observe that}
\begin{equation}\label{eq:eqR1}
\Lambda+\gamma_1+\gamma_2=x+R_1+\Lambda_1.
\end{equation}

$\bullet$ Assume that $\Lambda\neq \gamma_1\neq \gamma_2$. {By definition of $\Phi(\Lambda,Y)=\Lambda_{1}$, see \eqref{eq:sm}}, we have 
    \begin{equation}\label{eq:eqy0}
     \frac{y_0}{\Lambda-\Lambda_1}+\sum_{k=1}^{p}\frac{t\alpha_{k}}{(\eta_k-\Lambda)(\eta_k-\gamma_1)(\eta_k-\gamma_2)}=0.
    \end{equation}
    Using the partial fraction decomposition
\begin{eqnarray*}
\frac{1}{(\eta_k-\Lambda)(\eta_k-\gamma_1)(\eta_k-\gamma_2)} &=& \frac{1}{(\Lambda - \gamma_1)(\Lambda - \gamma_2)}  \frac{1}{\eta_k - \Lambda} + \frac{1}{(\gamma_1 - \Lambda)(\gamma_1 - \gamma_2)} \frac{1}{\eta_k - \gamma_1} \\
&&+ \frac{1}{(\gamma_2 - \Lambda)(\gamma_2 - \gamma_1)} \frac{1}{\eta_k - \gamma_2},
\end{eqnarray*}
one may rewrite \eqref{eq:eqy0} as
\begin{equation}\label{eq:simplified}
    \frac{y_0}{\Lambda-\Lambda_1}-\frac{1}{(\Lambda-\gamma_1)(\Lambda-\gamma_2)}G_\nu(\Lambda)+\frac{1}{(\Lambda-\gamma_1)(\gamma_1-\gamma_2)}G_\nu(\gamma_1)-\frac{1}{(\Lambda-\gamma_2)(\gamma_1-\gamma_2)}G_\nu(\gamma_2)=0.
\end{equation}
The condition $\sum_{k=0}^{p}y_k=1$ reads
\begin{equation*}
    y_0+t\sum_{k=1}^{p}\alpha_{k} \frac{\eta_k-\Lambda_1}{(\eta_k-\Lambda)(\eta_k-\gamma_1)(\eta_k-\gamma_2)}=1,
\end{equation*}
which gives
\begin{equation*}
   t\sum_{k=1}^{p}\frac{\alpha_{k}}{(\eta_k-\gamma_1)(\eta_k-\gamma_2)}+t(\Lambda-\Lambda_1)\sum_{k=1}^{p}\frac{\alpha_{k}}{(\eta_k-\Lambda)(\eta_k-\gamma_1)(\eta_k-\gamma_2)}=1-y_0.
\end{equation*}
Hence, using \eqref{eq:eqy0}, one gets
\begin{equation}\label{eq:=1}
   \sum_{k=1}^{p} \frac{t\alpha_{k}}{(\eta_k-\gamma_1)(\eta_k-\gamma_2)}=1,
\end{equation}
which gives
\begin{equation}\label{eq:Hg1g2}
 \frac{1}{\gamma_1-\gamma_2}(-tG_\nu(\gamma_1)+tG_\nu(\gamma_2))=1\Leftrightarrow    H_{\nu,t}(\gamma_1)=H_{\nu,t}(\gamma_2).
\end{equation}
In view of \eqref{eq:eqR1}, we have
\begin{equation*}
    \Lambda y_0+t\sum_{k=1}^{p} \frac{\alpha_{k}(\eta_k-\Lambda_1)\eta_k}{(\eta_k-\Lambda)(\eta_k-\gamma_1)(\eta_k-\gamma_2) }=R_1=\Lambda+\gamma_1+\gamma_2-x-\Lambda_1,
\end{equation*}
which implies
\begin{equation*}
    \Lambda y_0+t\sum_{k=1}^{p} \frac{\alpha_{k}(\eta_k-\Lambda_1)}{(\eta_k-\gamma_1)(\eta_k-\gamma_2)} +\Lambda (1-y_0)=\Lambda+\gamma_1+\gamma_2-x-\Lambda_1.
\end{equation*}
Simplifying further we obtain
\begin{equation*}
   -tG_\nu(\gamma_2)+(\gamma_1-\Lambda_1)=\gamma_1+\gamma_2-x-\Lambda_1.
\end{equation*}
Therefore, inserting \eqref{eq:Hg1g2}, we conclude that 
\begin{equation*}
    H_{\nu,t}(\gamma_2)=H_{\nu,t}(\gamma_1)=x.
\end{equation*}
Therefore $\gamma_1=\omega_{\nu,t}(x)$ and $\gamma_2=\omega^*_{\nu,t}(x)$. Moreover we can compute using \eqref{eq:simplified} that
$$\frac{y_0}{\Lambda-\Lambda_1}(\Lambda-\gamma_1)(\Lambda-\gamma_2)=\Lambda+tG_\nu(\Lambda)-x$$ 
which {implies} 
\begin{equation}\label{eq:y0here}
\Lambda_{1}=\Lambda -y_{0 }\frac{(\Lambda-\gamma_1)(\Lambda-\gamma_2)}{H_{\nu,t}(\Lambda)-x}\,.
\end{equation}

$\bullet$
Assume that $\Lambda=\gamma_1=\gamma_2$. Then proceeding as in \eqref{eq:=1}, we get
\begin{equation*}
    \sum_{k=1}^{p} \frac{t \alpha_{k}}{(\eta_k-\Lambda)^2}=1,
\end{equation*}
implying that $\Lambda=\omega_{\nu,t}(\ell_{\nu,t})$. In particular, $\Phi(\Lambda,Y)\geq \omega_{\nu,t}(x)$, which contradicts \eqref{eq:Phi1a}. 

$\bullet$ Assume that $\Lambda=\gamma_1$ and $\gamma_1\neq \gamma_2$. Proceeding as above we easily get $H_{\nu,t}(\Lambda)=H_{\nu,t}(\gamma_2)=x$, from which it is inferred that $\Lambda=\omega_{\nu,t}(x)$, implying that $\Phi(\Lambda,Y)\geq \omega_{\nu,t}(x)$, thus contradicting \eqref{eq:Phi1a}.

$\bullet$ Assume that $\Lambda\neq \gamma_1$ and $\gamma_1=\gamma_2=\gamma$. Then we get 
\begin{equation*}
    \sum_{k=1}^{p} \frac{t\alpha_{k}}{(\eta_k-\gamma)^2}=1
\end{equation*}
and $H_{\nu,t}(\gamma)=x$, which implies that $x=\ell_{\nu,t}$ and $\gamma=\omega_{\nu,t}(\ell_{\nu,t})$, contradicting the assumption that $x<\ell_{\nu,t}$.

We conclude from \eqref{defy} and \eqref{defiP} and the above discussion  that if $Y$ is a critical point of $G$ satisfying \eqref{eq:Phi1a}, then it is of the form
\begin{equation}\label{eq:muGG}
   y_{k}=\frac{t\alpha_{k}(\eta_k-\Lambda_1)}{(\eta_k-\Lambda)(\eta_k-\omega_{\nu,t}(x))(\eta_k-\omega_{\nu,t}^*(x))},\quad \text{{for every} $1\le k\le p$},
\end{equation}
for some $\Lambda_1\in(\Lambda,\ell_{\nu})$ such that letting $y_0:=1-\sum_{k=1}^{p} y_{k}$, the relations \eqref{eq:y0here}, \eqref{eq:eqy0} hold. In the sequel we denote $\gamma_1:=\omega_{\nu,t}(x)$ and $\gamma_2:=\omega_{\nu,t}^*(x)$. We next show that this critical point will never be a minimizer of $G$. 

\medskip

\paragraph{\bf{Step {3}: study of the minimizers. }}
{A minimizer $Y$ of $G$ is either such that $\Phi(\Lambda,Y)<\omega_{\nu,t}(x)$ or a critical point of ${J_{\nu,t}^{\Lambda}(x,\cdot)}$ such that $\Phi(\Lambda,Y)\geq \omega_{\nu,t}(x)$. Therefore it is either of the form \eqref{eq:muGG} or of the form $Y(\gamma)$, for $\gamma\in\{\Lambda,\omega^{*}_{\nu,t}(x)\}$ by Lemma \ref{lemma:opt J}.} We will study these minimizers in the different regimes and show that  $G$ in fact has a
 unique minimizer  and it is  the same as that of $J_{\nu,t}^{\Lambda}$, except in the previously undecided case
 $\Lambda<\omega_{\nu,t}(x)$, where it is now given by $Y({\Lambda})$. This will complete the proof of the Lemma.

$\bullet$ Assume that $\Lambda\ge \omega_{\nu,t}(x)$. Then, since $\Phi(\Lambda,Y)\ge \Lambda\ge  \omega_{\nu,t}(x)$, 
$$
G(Y)= K_{\nu,t}(\Lambda,x,Y)=J_{\nu,t}^{\Lambda}(x,Y)-\int \log|x-y|\dd(\nu\boxplus\sigma_t)(y)+\frac{x^2}{4t} -C_t{.}$$
{The function $G(Y)$ is therefore uniquely minimized by the unique minimizer of $J_{\nu,t}^\Lambda(x,\cdot)$ described in Lemma \ref{lemma:opt J}},  which is equal to $Y(\Lambda)$ if $x\le \ell^{\Lambda}_{\nu,t}$ and $Y(\omega^{*}_{\nu,t}(x))$ otherwise. This proves the second part of  Lemma \ref{lemma:solves}(2)(a) and Lemma \ref{lemma:solves}(2)(b).

$\bullet$ {Assume that $\Lambda<\omega_{\nu,t}(x)$ and $x=\ell_{\nu,t}$. By Lemma \ref{lemma:opt J}, the minimum of $J_{\nu,t}^\Lambda(\ell_{\nu,t},\cdot)$ is uniquely attained at $Y=Y(\Lambda)$. Moreover $\Phi(\Lambda,Y(\Lambda))=\omega^*_{\nu,t}(\ell_{\nu,t})=\omega_{\nu,t}(\ell_{\nu,t})$. Therefore $G(Y(\Lambda))=I_{\nu,t}^\Lambda(\ell_{\nu,t})$. Hence, using $I_{\nu,t}^\Lambda\geq 0$ and $I_{\nu,t}^{\Phi(\Lambda,Y(\Lambda))}(\ell_{\nu,t})=0$, we see that if $Y$ is a minimizer of $G$, then $J^\Lambda_{\nu,t}(\ell_{\nu,t},Y)=J^\Lambda_{\nu,t}(\ell_{\nu,t},Y(\Lambda))$. Therefore, $Y=Y(\Lambda)$.}

$\bullet$ {Assume that $\Lambda< \omega_{\nu,t}(x)$ and $x<\ell_{\nu,t}$. Then, for every  $Y\in \Delta^{p}$, $\Phi(\Lambda,(0, Y))=\Lambda$ and 
$$
G((0,Y))= J_{\nu,t}^{\Lambda}(x,(0,Y))-\int \log|y-x|\dd(\nu\boxplus\sigma_t)(y)+\frac{x^2}{4t}+I_{\nu,t}^{\Lambda}(x)1_{\Lambda<\omega_{\nu,t}(x)} -C_t{.}$$
Since the term depending on $I_{\nu,t}^{\Lambda}$ does not depend on $Y$ we conclude that $G(0,\cdot)$ is {minimal} at the minimum of $J_{\nu,t}^{\Lambda}(x,(0,\cdot))$, namely at $(y_1(\omega_{\nu,t}^{*}(x)),\ldots,y_p(\omega_{\nu,t}^{*}(x)))$ and therefore
\begin{equation*}
    \inf_{Y\in \Delta^p} G((0,Y))= I_{\nu,t}^{\omega_{\nu,t}^{*}(x)}(x)+I_{\nu,t}^{\Lambda}(x).
\end{equation*}
One can check that  for every $\lambda<\ell_{\nu,t}$,
\begin{equation*}
    \frac{\partial}{\partial \lambda}I_{\nu,t}^{\omega_{\nu,t}^{*}(\lambda)}(\lambda)=\frac{1}{2t}(\omega_{\nu,t}(\lambda)-\omega_{\nu,t}^*(\lambda))<0.
\end{equation*}
Since $I_{\nu,t}^{\omega_{\nu,t}^{*}(\ell_{\nu,t})}(\ell_{\nu,t})=0$ we deduce that for all $ x<\ell_{\nu,t}$,
\begin{equation}\label{eq:posI*}
    I_{\nu,t}^{\omega_{\nu,t}^{*}( x)}( x)>0
\end{equation}
and therefore
\begin{equation*}
    \inf_{Y\in \Delta^p} G((0,Y))>I_{\nu,t}^\Lambda( x)=G(Y(\Lambda))
\end{equation*}
with $Y(\Lambda)\in \Delta^{p+1}$. It follows that minimizers of $G$ are not of the form $(0,Y)$ with $Y\in \Delta^{p}$.}

{Let us prove that $Y(\Lambda)$ uniquely minimizes $G$.} {Since $\Lambda\leq \omega_{\nu,t}(\ell_{\nu,t})$, $Y(\Lambda)\in \Delta^{p+1}$}. By \eqref{eq:valueYK}, we have 
$
    G(Y({\Lambda}))=I_{\nu,t}^{\Lambda}( x).
$
Therefore the minimizer of $G$ is either $Y(\Lambda)$
 or $(y_0,Y)$ of the form \eqref{eq:muGG}.

Let $(y_0,Y)$ be of the form \eqref{eq:muGG}. Let us compute $G(y_0,Y)$. Define $M( x):=\Lambda y_0+\sum_{k=1}^{p} \eta_k y_{k}$. Recall that $M( x)=\Lambda+\gamma_1+\gamma_2- x-\Lambda_1$. We have
\begin{equation*}
\begin{split}
  \sigma( x) &:=\Lambda^{2}y_{0}+  \sum_{k=1}^{p}\eta_k^2 y_{k}=\Lambda^{2}y_{0}+t-(\Lambda_1-\Lambda)tG_\nu(\Lambda)+ (\gamma_1+\gamma_2)M( x) -\Lambda (\gamma_1+\gamma_2)y_0-\gamma_1\gamma_2 (1-y_0)\\
  &=(\Lambda-\gamma_1)(\Lambda-\gamma_2)y_0-\gamma_1 \gamma_2-(\Lambda_1-\Lambda)tG_\nu(\Lambda)+(\gamma_1+\gamma_2)M( x)\\
  &=\gamma_1 \gamma_2+(\gamma_1+\gamma_2)M( x)+(\Lambda-\Lambda_1)(\Lambda- x).
\end{split}
\end{equation*}
where we inserted
 \eqref{eq:y0here}. It follows that
\begin{equation*}
\begin{split}
   G(y_0,Y)&=\frac12-\frac{1}{2}\ln t+\frac{1}{4t}(-2 x-M( x)+2(\gamma_1+\gamma_2))M( x) -\frac{\gamma_1\gamma_2}{2t}+\frac{1}{2t}(\Lambda-\Lambda_1)(\Lambda- x)\\
   &-\frac{1}{2}(S_\nu(\Lambda)+S_\nu(\gamma_1)
   +S_\nu(\gamma_2)-S_\nu(\Lambda_1))
   +  \frac{1}{2}(S_\nu(\gamma_1)-S_\nu(\Lambda_1))+\frac{1}{4t}(\Lambda_1^2-2 x\Lambda_1-\gamma_1^2+2 x \gamma_1)\\
   &=-\frac{1}{2}(S_\nu(\gamma_2)+S_\nu(\Lambda))+\frac{Z}{4t}+\frac{1}{2}-\frac{1}{2}\ln t,
\end{split}
\end{equation*}
with
\begin{equation*}
\begin{split}
Z&:=(-2 x-M( x)+2(\gamma_1+\gamma_2))M( x)-2\gamma_1\gamma_2+2(\Lambda^2+\Lambda_1 x-\Lambda x-\Lambda_1  x)+\Lambda_1^2-2 x \Lambda_1-\gamma_1^2+2 x \gamma_1\\
&=\Lambda^2+ x^2+\gamma_2^2-2\gamma_2 x-2 x \Lambda.
\end{split}
\end{equation*}
We therefore find that 
\begin{equation}\label{eq:final value}
G(y_0,Y)=\frac{1}{4t}(\Lambda^2+ x^2+\gamma_2^2-2\gamma_2 x-2 x \Lambda)-\frac{1}{2}(S_\nu(\gamma_2)+S_\nu(\Lambda))+\frac{1}{2}-\frac{1}{2}\ln t.
\end{equation}
We next {show} that this is greater than $G(Y({\Lambda}))$. Let us define 
\begin{equation}\label{def:phigamma}
    \phi:\gamma\in (-\infty,\ell_{\nu})\mapsto \frac{1}{4t}(\Lambda^2+ x^2+\gamma^2-2\gamma x-2 x \Lambda)-\frac{1}{2}(S_\nu(\gamma)+S_\nu(\Lambda))+\frac{1}{2}-\frac{1}{2}\ln t.
\end{equation}
For all $\gamma<\ell_{\nu}$, one has
\begin{equation*}
    \phi'(\gamma)=\frac{H_{\nu,t}(\gamma)- x}{2t}.
\end{equation*}
Therefore $\phi$ is strictly increasing on $[\gamma_1,\gamma_2]$. It follows that if $\gamma_1<\gamma_2$,
\begin{equation*}
 \phi(\gamma_2)=G(y_0,Y)> G( Y(\Lambda))=\phi(\gamma_1)
\end{equation*}
and $Y(\Lambda)$ is the unique minimizer of $G$. If $\gamma_1=\gamma_2$, then $Y(\Lambda)$ is still the unique minimizer of $G$. This proves the first part of  Lemma \ref{lemma:solves} (2)(a) and completes the proof of the lemma.

\end{proof}

\medskip
We finally prove Lemma \ref{lemma:F+} by iterating  the functional inequality
\eqref{eq:ub} and showing that after a finite number of {steps} it {does not depend} on $F^{+}_{\nu,t}$, so that $F^{+}_{\nu,t}$ {is bounded from below} by $I_{\nu,t}^{\Lambda}$.

\begin{proof}[Proof of Lemma \ref{lemma:F+}] By Lemma \ref{lemma:boundaryV}, we only need to consider {the case where} $\Lambda< \omega_{\nu,t}( x)$ {and $ x<\ell_{\nu,t}$} which we  assume in the sequel.
Recall
\begin{equation*}
F_{\nu,t}^+(\Lambda, x)\geq \inf_{Y\in\Delta^{p+1}}\{K_{\nu,t}(\Lambda, x,Y)+F_{\nu,t}^+(\Phi(\Lambda,Y), x^-)1_{\Phi(\Lambda,Y)<\omega_{\nu,t}( x)}\}.
\end{equation*}
Recall from Lemma \ref{Fprop} that {$ x\mapsto F_{\nu,t}^+(\Phi(\Lambda,Y), x^-)$ increases on $\{ x<\ell_{\nu,t}:\Phi(\Lambda,Y)<\omega_{\nu,t}( x)\}$.}
Hence, for $\ve>0$ small enough,
\begin{equation}\label{eq:ine}
F_{\nu,t}^+(\Lambda, x)\geq \inf_{Y\in\Delta^{p+1}}\{K_{\nu,t}(\Lambda, x,Y)+F_{\nu,t}^+(\Phi(\Lambda,Y), x-\ve)1_{\Phi(\Lambda,Y)<\omega_{\nu,t}( x)}\},
\end{equation}
By iteration define $\Phi_{0}(\Lambda,Y)=\Lambda, \Phi_{1}(\Lambda,Y)=\Phi(\Lambda,Y)$ and  for  each $k\geq 1$, 
\begin{equation*}
    \Phi_k(\Lambda,Y_1,\ldots,Y_k):=\Phi(\Phi_{k-1}(\Lambda,Y_1,\ldots,Y_{k-1}),Y_k)\quad {\text{for all $Y_1,\ldots,Y_k\in \Delta^{p+1}$.}}
\end{equation*}
Recall that for every $\Lambda, Y$, $\Phi(\Lambda,Y)\geq \Lambda$. Consequently {for every} $k\geq 1$  
\begin{equation}\label{tf}
 \bigcap_{i=1}^k \{Y_{1},\ldots,Y_{k}\in (\Delta^{p+1})^{k}:\Phi_i(\Lambda,Y_1,\ldots,Y_i)<\omega_{\nu,t}( x)\}=\{Y_{1},\ldots,Y_{k}\in (\Delta^{p+1})^{k}:\Phi_k(\Lambda,Y_1,\ldots,Y_k)<\omega_{\nu,t}( x)\}.
\end{equation}
Also recall that by Lemma \ref{lemma:opt J}, $\inf K_{\nu,t}(\Phi_{k-1}(\Lambda,Y_1,\ldots,Y_{k-1}), x, \cdot)=0$ when $\Phi_{k-1}(\Lambda,Y_1,\ldots,Y_{k-1})=\omega_{\nu,t}( x)$. {Let $n\geq 1$. Applying \eqref{eq:ine} $n$ times for $\ve$ small enough, gives}
\begin{multline}
F_{\nu,t}^+(\Lambda, x)\geq \inf_{Y_1,\ldots,Y_n}\Bigr(\sum_{k=1}^n K_{\nu,t}(\Phi_{k-1}(\Lambda,Y_1,\ldots,Y_{k-1}),  x-(k-1)\ve,Y_k)1_{\cap_{\ell=1}^{k-1}\{\Phi_{\ell}(\Lambda,Y_1,\ldots,Y_{\ell})<\omega_{\nu,t}( x-(\ell-1)\ve)\}}\\+F_{\nu,t}^+(\Phi_n(\Lambda,Y_1,\ldots,Y_n), x-n\ve )1_{\cap_{\ell=1}^n\{\Phi_\ell(\Lambda,Y_1,\ldots,Y_\ell)<\omega_{\nu,t}( x-(\ell-1)\ve)\}}\Bigr).\label{opt}
\end{multline}
Recall that $\Lambda{\mapsto} F_{\nu,t}^{+}(\Lambda, x)$ is lower semicontinuous  and $K_{\nu,t}(., x,Y)$ is continuous so that the above minimum is achieved on the compact set $(\Delta^{p+1})^{n}$. Let $(Y_1^\ve,\ldots,Y_n^\ve)$ be a minimizer.  By compactness, one can extract a subsequence $(Y_1^{\ve_l},\ldots,Y_n^{\ve_l})$ converging as $\ve$ goes to zero to some $(Y_1,\ldots,Y_n)\in (\Delta^{p+1})^n$.

$\bullet$ Assume that $\Phi_n(\Lambda,Y_1,\ldots,Y_n)>\omega_{\nu,t}( x)$. Then, by continuity of $\Phi$ and $\omega_{\nu,t}$, for $\ve$ small enough, 
$\Phi_n(\Lambda,Y_1^{\ve},\ldots,Y_n^{\ve})>\omega_{\nu,t}( x-(n-1) \ve)$.
{Hence, in this case, by letting $\ve$ go to zero and using \eqref{tf}, we find}
\begin{eqnarray*}
 F_{\nu,t}^+(\Lambda, x)&\geq &K_{\nu,t}(\Lambda, x,Y_1)+ \sum_{k=2}^n K_{\nu,t}(\Phi_{k-1}(\Lambda,Y_1,\ldots,Y_{k-1}), x,Y_k)1_{\{\Phi_{k-1}(\Lambda,Y_1,\ldots,Y_{k-1})<\omega_{\nu,t}( x)\}}\\
    &&+I_{\nu,t}^{\Phi_n(\Lambda,Y_1,\ldots,Y_n)}( x )1_{\{\Phi_n(\Lambda,Y_1,\ldots,Y_n)<\omega_{\nu,t}( x)\}}\\
    &\ge&\inf_{Y_1,\ldots,Y_n}\Bigr(K_{\nu,t}(\Lambda, x,Y_1)+\sum_{k=2}^n K_{\nu,t}(\Phi_{k-1}(\Lambda,Y_1,\ldots,Y_{k-1}), x,Y_k)1_{\{\Phi_{k-1}(\Lambda,Y_1,\ldots,Y_{k-1})<\omega_{\nu,t}( x)\}}\\
    &&+I_{\nu,t}^{\Phi_n(\Lambda,Y_1,\ldots,Y_n)}( x )1_{\{\Phi_n(\Lambda,Y_1,\ldots,Y_n)<\omega_{\nu,t}( x)\}}\Bigr),
\end{eqnarray*}
where in the first line we kept the last term and replaced $F^{+}_{\nu,t}$ by $I_{\nu,t}$ since it is multiplied by zero. {We can then take the infimum over $Y_{n}$.} By Lemma \ref{lemma:solves},
$$\inf_{Y_{n}}\{ K_{\nu,t}(\Phi_{n-1}(\Lambda,Y_1,\ldots,Y_{n-1}), x,Y_n)1_{\{\Phi_{n-1}(\Lambda,Y_1,\ldots,Y_{n-1})<\omega_{\nu,t}( x)\}}+I_{\nu,t}^{\Phi_n(\Lambda,Y_1,\ldots,Y_n)}( x )1_{\{\Phi_n(\Lambda,Y_1,\ldots,Y_n)<\omega_{\nu,t}( x)\}}\}\qquad$$
$$\qquad=
I_{\nu,t}^{
\Phi_{n-1}(\Lambda,Y_1,\ldots,Y_{n-1})}( x )1_{\{\Phi_{n-1}(\Lambda,Y_1,\ldots,Y_{n-1})<\omega_{\nu,t}( x)\}}$$
and proceeding inductively we deduce that
$$ F_{\nu,t}^+(\Lambda, x)\ge I_{\nu,t}^{\Lambda}( x).$$

$\bullet$ We show that it is impossible that  $\Phi_n(\Lambda,Y_1,\ldots,Y_n)\le \omega_{\nu,t}( x)$ if $n$ is large enough  by contradiction. 
We first notice that by definition for every $\Lambda,Y$, 
\begin{equation}\label{bv}\frac{y_{0}}{\Phi(\Lambda,Y)-\Lambda}=\sum_{i=1}^{p} \frac{y_{i}}{\eta_{i}-\Phi(\Lambda,Y)}{.}\end{equation}
{We can assume without loss of generality that  $\Lambda$ is such that $\eta_{i}-\Phi(\Lambda,Y)>\gamma$ for some $\gamma>0$; otherwise if $\Phi(\Lambda,Y)$ is close to $\ell_{\nu}$, then $\Phi(\Lambda,Y)$ would be greater than $ \omega_{\nu,t}( x)$, placing us in the previously discussed scenario. Hence, by \eqref{bv},}
$$\Phi(\Lambda,Y)\ge \Lambda +\gamma y_{0}\,.$$
We therefore have two cases,  being given some small $\delta>0$,
\begin{itemize}
\item Either $y_{0}\ge \delta/\gamma$, and then \begin{equation}\label{lb2}\Phi(\Lambda,Y)\ge \Lambda +\delta\,.\end{equation}
\item {Or $y_{0}\le \delta/\gamma$. Then  denote by $\tilde Y$ the element of $\Delta^{p+1}$ {such} that $\tilde y_{0}=0$ and $\tilde y_{i}=y_{i}(1-y_{0})^{-1}$ for $i\in\{1,\ldots,p\}$. The continuity of $K_{\nu,t}(\Lambda, x,.)$ implies that for $\delta$ small enough, uniformly on $Y,\Lambda, x$, 
$$ K_{\nu,t}(\Lambda, x, Y)= K_{\nu,t}(\Lambda, x, \tilde Y)+O(\delta)\ge \min_{Y\in\Delta^{p}}\{K_{\nu,t}(\Lambda, x, (0, Y))\}+O(\delta).$$
By Lemma \ref{lemma:opt J}, the minimum of $K_{\nu,t}(\Lambda, x,\cdot)$ over elements of the form $(0,Y)$ is taken at $Y=Y(\omega^{*}_{\nu,t}( x))$ and therefore
\begin{equation}\label{lb1}
 \min_{Y\in\Delta^{p}}\{K_{\nu,t}(\Lambda, x, (0, Y))\}\ge  I_{\nu,t}^{\omega^{*}_{\nu,t}( x)}( x).\end{equation}
Notice that the {right-hand} side does not depend on $\Lambda$ and moreover by \eqref{eq:posI*}, $I_{\nu,t}^{\omega^{*}_{\nu,t}( x)}( x)>0$. We hereafter choose $\delta$ small enough such that 
$$\kappa( x):= \min_{Y\in\Delta^{p}}\{K_{\nu,t}(\Lambda, x, (0, Y))\}+O(\delta)$$ is {strictly} positive and independent of $\Lambda$}.
\end{itemize}
{Returning to \eqref{opt}, and recalling that $(Y_{1}^{\ve},\ldots,Y_{n}^{\ve})$ is a minimizer, and that $F_{\nu,t}^{+}$ is lower semicontinuous  in the variable $\Lambda$ while $K_{\nu,t}$ and $\omega_{\nu,t}$ are continuous, we obtain the following: if again $(Y_{1}
,\cdots,Y_{n})$ denotes a limiting point of the optimizing sequence $(Y_{1}^{\ve},\ldots,Y_{n}^{\ve})$}, {then}
\begin{eqnarray*}
    F_{\nu,t}^+(\Lambda, x)&\geq& \sum_{k=1}^n K_{\nu,t}(\Phi_{k-1}(\Lambda,Y_1,\ldots,Y_{k-1}), x, Y_k)1_{\{\Phi_{k-1}(\Lambda,Y_1,\ldots,Y_{k-1})
   < \omega_{\nu,t}( x)\}}\\
   &&+F_{\nu,t}^+(\Phi_n(\Lambda,Y_1,\ldots,Y_n), x^{-})1_{\{\Phi_n(\Lambda,Y_1,\ldots,Y_n)< \omega_{\nu,t}( x)\}}.
   \end{eqnarray*}
From the above discussion, we see from \eqref{lb2} that if there are $\ell$ indices $m$ such that $(Y_{m})_0\ge \delta/\gamma$ then $\Phi_n(\Lambda,Y_1,\ldots,Y_n)\ge \Lambda+\ell \delta$ which implies that $\Phi_n(\Lambda,Y_1,\ldots,Y_n)> \omega_{\nu,t}( x)$ if $\ell$ is greater than some constant, which contradicts our {assumption}. Therefore, we conclude that there exists $\ell_{0}$ finite such that for $n-\ell_{0}$ indices $i$, $(Y_{i})_0\le \delta/\gamma$ . But then, since $K_{\nu,t}$ and $F_{\nu,t}^{+}$ are non-negative, \eqref{lb1} implies
   $$ F_{\nu,t}^+(\Lambda, x)\ge (n-\ell_{0}) \kappa( x),$$
  {which is impossible for $n$ large enough since $F_{\nu,t}^+(\Lambda, x)$ is finite}. {Therefore, we deduce that there exists $n$ large enough, such that {for every} $k\geq n$, $\Phi_{k}(\Lambda,Y_1,\ldots,Y_{k})
> \omega_{\nu,t}( x)$. We can argue as in the first case  to conclude that}
\begin{equation*}
    F_{\nu,t}^+(\Lambda, x)\geq I_{\nu,t}^{\Lambda}( x).
\end{equation*}
Using the upper bound \eqref{eq:b1}, this concludes the proof of {Lemma \ref{lemma:F+}}.
\end{proof}

\section{Large deviations in the general case}
In this section, we suppose that the sequence of matrices $(B_N)$ satisfies Assumption \ref{assumption:B}. The measure $\nu$ is now allowed to have a continuous support.   We first remark that we may restrict ourselves to the case where $b_{1}^{N}=\Lambda$ and $b_{2}^{N}=\ell_{\nu}$  up to replacing $B_{N}$ by another matrix  $B_{N}'$ satisfying this hypothesis, so that  they are equal except for these two entries. In fact, we then see that $\|B_{N}-B_{N}'\|_{\mathrm{op}}$ goes to zero with $N$ so that also $\lambda_{1}(G_{N}+B_{N})-\lambda_{1}(G_{N}+B_{N}')$ goes to zero with $N$ everywhere. Thus, large deviations for  the distribution of  $\lambda_{1}(G_{N}+B_{N}')$ imply the large deviations for the law of $\lambda_{1}(G_{N}+B_{N})$  since they are exponentially equivalent \cite[Theorem 4.2.13]{DZ}.
\subsection{Properties of the rate function}\label{propgrf}

We first prove the approximation and convexity properties of the rate function $I_{\nu,t}^{\Lambda}$ defined in \eqref{eq:expl formula}. Recall that $\eta_\nu$ is the transport map {such that} $\nu=\eta_{\nu}\#1_{[0,1]} \dd x$. 

\begin{lemma}[Properties of the rate function]\label{lemma:convexity}

\begin{enumerate}
\item (Continuity) Let $\nu$ be a probability measure on $\dR$ with bounded support. Let $(\nu_n)_{n\geq 1}$ be a sequence of probability measures on $\dR$ such that $(\eta_{\nu_{n}})$ converges uniformly towards $\eta_{\nu}$. Then for every $x\in (-\infty,{\ell_{\nu,t}})$,
\begin{equation*}
  \lim_{n\to\infty} I_{\nu_n,t}^{\Lambda}(x)=I_{\nu,t}^{\Lambda}(x).
\end{equation*}
\item (Convexity) Let $\nu$ be a probability measure on $\dR$ with support bounded from below. The map $I_{\nu,t}^{\Lambda}$ is strictly convex on $(-\infty, \ell_{\nu,t})$ and achieves its minimum uniquely, at $H_{\nu,t}(\Lambda)=\Lambda+tG_\nu(\Lambda)$ if $\Lambda\leq \omega_{\nu,t}(\ell_{\nu,t})$ or $\ell_{\nu,t}$ if $\Lambda\geq \omega_{\nu,t}(\ell_{\nu,t})$. 
\item (Bounds)   Moreover, for every $\ve>0$, there exists a positive constant $C_{\ve}$ so that for $x\le \ell_{\nu,t}-\ve$,
\begin{equation}\label{lowerq}
I_{\nu,t}^{\Lambda}(x)\ge C_{\ve}(x-\ell_{\nu,t}^{\Lambda})^{2}\,.
\end{equation}
At infinity, we have $I_{\nu,t}^{\Lambda}(x)\simeq \frac{x^{2}}{4t}$.
\end{enumerate}
\end{lemma}

\medskip

\begin{proof}\

\paragraph{\bf{Step 1: proof of continuity}} It is easy to check that, if $|\eta_{\nu}|\vee |x|\vee|\Lambda|\le M$, if $|\eta_{\nu_n}-\eta_{\nu}|\le \varepsilon$, and if
$x\le \ell_{\nu,t}-\kappa$, then there exist constants $C=C(M,t)>0$ and $c=c(M,t,\kappa)>0$ such that
\begin{equation*}
|\ell_{\nu,t}-\ell_{\nu_n,t}|\le C\varepsilon,
\qquad
|G_{\nu}(x)-G_{\nu_n}(x)|\le C\kappa^{-2}\varepsilon,
\qquad
|H_{\nu,t}(x)-H_{\nu_n,t}(x)|\le C\kappa^{-2}\varepsilon,
\end{equation*}
\begin{equation*}
H'_{\nu,t}(\omega_{\nu,t}(x))\ge c,
\end{equation*}
and therefore
\begin{equation*}
|\omega_{\nu,t}(x)-\omega_{\nu_n,t}(x)|+|\omega_{\nu,t}^{*}(x)-\omega_{\nu_n,t}^{*}(x)|
\le \frac{C}{c}\,\varepsilon.
\end{equation*}
See also Lemma~\ref{smoothness}. By continuity of $S_{\nu}(x)$ for $x<\ell_{\nu}$ (since $\nu$ has compact support), we deduce that
$I_{\nu_n,t}^{\Lambda}\to I_{\nu,t}^{\Lambda}$ uniformly on compact subsets of $(-\infty,\ell_{\nu,t})$, hence away from $\ell_{\nu,t}$.
\medskip
\paragraph{\bf{Step 2: proof of convexity}} 
Let $\nu$ be a probability measure on $\dR$ with support bounded from below. Assume first that $\Lambda\le \omega_{\nu,t}(\ell_{\nu,t})$ so that $\gamma_{\Lambda}(x)=\Lambda$. A direct derivation shows that  
\begin{equation*}
 \partial_{x}I_{\nu,t}^{\Lambda}(x)=\frac{1}{2t}(\omega_{\nu,t}(x)-\Lambda).
\end{equation*}
This implies that for all $x<\ell_{\nu,t}$,
\begin{equation}\label{eq:partial2}
\partial_{x}^{2}I_{\nu,t}^{\Lambda}(x)=\frac{1}{2t} \omega_{\nu,t}'(x)>0\,.
\end{equation}
Indeed the term  in the r.h.s is clearly positive since
\begin{equation*}
\omega_{\nu,t}'(x)=\frac{1}{1+tG_\nu'(\omega_{\nu,t}(x))}>0
\end{equation*}
on $(-\infty,\ell_{\nu,t})$ since $\omega_{\nu,t}(\ell_{\nu,t})$ satisfies
\begin{equation*}
    \omega_{\nu,t}(\ell_{\nu,t})=\inf\Bigr\{x:\int \frac{\dd\nu(\lambda)}{(\lambda-x)^2}\geq\frac{1}{t}\Bigr\}.
\end{equation*}
It follows that $I_{\nu,t}^{\Lambda}$ is strictly convex when $\Lambda\le\omega_{\nu,t}(\ell_{\nu,t})$. Moreover, since $\omega_{\nu,t}(H_{\nu,t}(\Lambda))=\Lambda$, we find $\partial_{x}I_{\nu,t}^{\Lambda}(H_{\nu,t}(\Lambda))=0$. Thus, $I_{\nu,t}^{\Lambda}$ attains its minimum uniquely at $\ell_{\nu,t}^\Lambda=H_{\nu,t}(\Lambda)$.

Now assume that $\Lambda\geq \omega_{\nu,t}(\ell_{\nu,t})$. For all $x<H_{\nu,t}(\Lambda)$, we have still $\gamma_{\Lambda}(x)=\Lambda$ and so as above,
\begin{equation}\label{ty}
   \partial_{x} I_{\nu,t}^{\Lambda}(x)=\frac{1}{2t}(\omega_{\nu,t}(x)-\Lambda)\leq 0.
\end{equation}
Moreover, for all $x\in [H_{\nu,t}(\Lambda),\ell_{\nu,t})$, we have $\gamma_{\Lambda}(x)=\omega_{\nu,t}^{*}(x)$ and we find 
\begin{equation}\label{eq:F2}
 \partial_{x}   I_{\nu,t}^{\Lambda}(x)=\frac{1}{2t}(\omega_{\nu,t}(x)-\omega^{*}_{\nu,t}(x))\leq 0.
\end{equation}
It follows that for all $x\in (H_{\nu,t}(\Lambda),\ell_{\nu,t})$,
\begin{equation}\label{partial22}
\partial_{x}^{2}    I_{\nu,t}^{\Lambda}(x)=\frac{1}{2t}\Bigr(\frac{1}{1+tG_{\nu}'(\omega_{\nu,t}(x))}-\frac{1}{1+tG_{\nu}'(\omega_{\nu,t}^{*}(x))}\Bigr)>0,
\end{equation}
since $\omega_{\nu,t}(x)<\omega_{\nu,t}(\ell_{\nu,t})<\omega_{\nu,t}^{*}(x)$. {Together with \eqref{eq:partial2},} this proves that $I_{\nu,t}^{\Lambda}$ is strictly convex and decreasing. Thus it reaches its minimum uniquely at $\ell^{\Lambda}_{\nu,t}=\ell_{\nu,t}$.  
Finally, the bound \eqref{lowerq} is a direct consequence from  uniform positive lower bounds on $\partial_{x}^{2} I_{\nu,t}^{\Lambda}(x)$ that are easily deduced from \eqref{eq:partial2} and \eqref{partial22} which degenerate only at $x=\ell_{\nu,t}$. The behaviour at infinity was already derived in \eqref{atinfinity}. 
\end{proof}

\subsection{Proof of Theorem \ref{theorem:LDP}}

\begin{proof}[Proof of Theorem \ref{theorem:LDP}]
Fix $\ve>0$ small enough so that $\ve<\ell_\nu-\Lambda$.
Since $\nu$ has at most countably many atoms, for each such $\ve$ we can choose a shift
$s\in(0,\ve)$ such that
\begin{equation*}
\nu\bigl(\{\ell_\nu+s+k\ve\}\bigr)=0\qquad\forall k\in\mathbb Z.
\end{equation*}
Define the maps
\begin{equation*}
D_{\ve}^-(x):=\ell_\nu+s+\ve\Big\lfloor \frac{x-(\ell_\nu+s)}{\ve}\Big\rfloor,
\qquad
D_{\ve}^+(x):=\ell_\nu+s+\ve\Big\lceil \frac{x-(\ell_\nu+s)}{\ve}\Big\rceil,
\end{equation*}
and the discretized matrices
\begin{equation*}
B_{N,\ve}^-:=\diag\bigl(\Lambda, D_{\ve}^-(b_2^N),\dots,D_{\ve}^-(b_N^N)\bigr),
\qquad
B_{N,\ve}^+:=\diag\bigl(\Lambda, D_{\ve}^+(b_2^N),\dots,D_{\ve}^+(b_N^N)\bigr).
\end{equation*}
Then $B_{N,\ve}^-\le B_N\le B_{N,\ve}^+$ and
$\|B_{N,\ve}^\pm-B_N\|_{\mathrm{op}}\le \ve$.
Let $\nu_\ve^\pm := (D_\ve^\pm)_\#\nu$.
By the choice of $s$, the discontinuity points of $D_\ve^\pm$  have $0$ measure under $\nu$, hence
$(D_\ve^\pm)_\#\mu_{B_N}\Rightarrow \nu_\ve^\pm$ weakly.

Since $B_{N,\ve}^-\le B_N\le B_{N,\ve}^+$, we have almost surely
\begin{equation*}
\lambda_1(B_{N,\ve}^-+G_N)\le \lambda_1(B_N+G_N)\le \lambda_1(B_{N,\ve}^++G_N),
\end{equation*}
and hence for all $x\in\mathbb R$,
\begin{equation}\label{sandwich}
\mathbb P\bigl(\lambda_1(B_{N,\ve}^-+G_N)\ge x\bigr)
\le
\mathbb P\bigl(\lambda_1(B_N+G_N)\ge x\bigr)
\le
\mathbb P\bigl(\lambda_1(B_{N,\ve}^++G_N)\ge x\bigr).
\end{equation}
Consequently we can use Proposition \ref{LDP:dis} to deduce from \eqref{sandwich} that
\begin{multline}
-\hspace{-0.2cm}\inf_{[x,+\infty)} I_{\nu_{\ve}^{-},t}^{\Lambda}\le \liminf_{N\rightarrow\infty}\frac{1}{N}\log  \mathbb P\left(  \lambda_{1}(B_{N} +G_N)\ge x\right)\le \limsup_{N\rightarrow\infty}\frac{1}{N}\log  \mathbb P\left(  \lambda_{1}(B_{N} +G_N)\ge x\right) \le -\hspace{-0.2cm}\inf_{[x,+\infty)} I_{\nu_{\ve}^{+},t}^{\Lambda}\,.\label{refer}
\end{multline}

We now observe that the function $D_{\ve}^{\pm}(\eta_{\nu})-\eta_{\nu}$ converges uniformly towards zero so that $I_{\nu_{\ve}^{\pm},t}^{\Lambda}-I_{\nu,t}^{\Lambda}$ goes to zero    by the first point of Lemma \ref{lemma:convexity}.  Therefore, 
since $I_{\nu_{\ve}^{\pm},t}^{\Lambda}$ goes to infinity like $x^{2}/4t$, we also deduce that 
$$\lim_{\ve\rightarrow 0} \inf_{[x,+\infty)}I_{\nu_{\ve}^{\pm},t}^{\Lambda}=\inf_{[x,+\infty)}I_{\nu,t}^{\Lambda}$$
and the same when the infimum is taken on $(x,+\infty)$. Therefore, letting $\ve$ going to zero,  \eqref{refer} implies that
\begin{equation}\label{lim}-\hspace{-0.2cm}\inf_{[x,+\infty)} I_{\nu,t}^{\Lambda}\le \liminf_{N\rightarrow\infty}\frac{1}{N}\log  \mathbb P\left(  \lambda_{1}(B_{N} +G_N)\ge x\right)\le \limsup_{N\rightarrow\infty}\frac{1}{N}\log  \mathbb P\left(  \lambda_{1}(B_{N} +G_N)\ge x\right) \le -\hspace{-0.2cm}\inf_{[x,+\infty)} I_{\nu,t}^{\Lambda}.\end{equation}

Let $\delta>0$ be such that $x+\delta\leq  \ell_\nu$.  
It follows that since $I_{\nu,t}^{\Lambda}$ is strictly increasing above ${\ell_{\nu,t}^\Lambda}$ by  Lemma  \ref{lemma:convexity}, for $\kappa>0$ small and  $N$ large enough
\begin{eqnarray*}
\mathbb   P\left(  \lambda_{1}(B_N +G_N)\ge x+\delta\right)&\le& e^{-N \inf_{(-\infty,x+\delta]} (I_{\nu,t}^{\Lambda}-\kappa) }\le  e^{-N \inf_{(-\infty,x]} (I_{\nu,t}^{\Lambda}+\kappa) }\\
  &\le& e^{-N\kappa/2}  \mathbb{P}\left(  \lambda_{1}(B_N +G_N)\ge x\right)\,.\end{eqnarray*}
Therefore
$$\mathbb P\left(  \lambda_{1}(B_N +G_N)\ge x\right)=\mathbb P\left(  \lambda_{1}(B_N +G_N)\in (x,x+\delta) \right)(1+o(1))\,.$$
As a consequence, we deduce from \eqref{lim}  that for $x\ge \ell_{\nu,t}^\Lambda$, 
\begin{multline}
 -\inf_{(x,x+\delta)}I_{\nu,t}^{\Lambda}\leq \liminf_{N\rightarrow\infty}\frac{1}{N}\log  \mathbb P\left(  \lambda_{1}(B_N +G_N)\in [x,x+\delta)\right)\\
 \leq   \limsup_{N\rightarrow\infty}\frac{1}{N}\log  \mathbb P\left(  \lambda_{1}(B_N +G_N)\in [x,x+\delta)\right) \leq -\inf_{[x,x+\delta)}I_{\nu,t}^{\Lambda}.\label{wer}
\end{multline}
Since $I_{\nu,t}^{\Lambda}$ is lower semicontinuous, we have
\begin{equation*}
    \lim_{\delta\to 0}\inf_{(x,x+\delta)}I_{\nu,t}^{\Lambda}=\lim_{\delta\to 0}\inf_{[x,x+\delta)}I_{\nu,t}^{\Lambda}=I_{\nu,t}^{\Lambda}(x).
\end{equation*}
Hence, letting $\delta$ going to zero, \eqref{wer} gives the weak large deviations estimate above ${\ell_{\nu,t}^\Lambda}$. 
Similarly we find by the above comparison that 
\begin{equation*}-\inf_{(-\infty, x]} I_{\nu_{\ve}^{+},t}^{\Lambda}\le \liminf_{N\rightarrow\infty}\frac{1}{N}\log  \mathbb P\left(  \lambda_{1}(B_N +G_N)\le x\right)\le \limsup_{N\rightarrow\infty}\frac{1}{N}\log  \mathbb P\left(  \lambda_{1}(B_N +G_N)\le x\right) \le -\inf_{(-\infty,x]} I_{\nu_{\ve}^{-},t}^{\Lambda}.
\end{equation*}
We can also use the convergence of $I_{\nu_{\ve}^{\pm},t}^{\Lambda}$ towards $I_{\nu,t}^{\Lambda}$
and the fact that $I_{\nu,t}^{\Lambda}$ is strictly decreasing below ${\ell_{\nu,t}^\Lambda}$ to deduce the weak large deviation principle below ${\ell_{\nu,t}^\Lambda}$.
This proves the weak LDP everywhere. This result {can be upgraded} to a full LDP using exponential tightness as in the proof of Proposition \ref{LDP:dis} and the fact that $I_{\nu,t}^{\Lambda}$ is a good rate function.
Lemma \ref{lemma:convexity} allows us to complete the proof of Theorem \ref{theorem:LDP}.

\end{proof}

\section{Comparison with prior results}\label{prior}

In this section we show that our rate function agrees with \cite{maida2007large,mckenna2021large}.

\subsection{The case of the GOE with an outlier}

We first consider the case where $B_N=\Lambda e_1 e_1^T$ for $\Lambda<0$ and compare our result with \cite{maida2007large},  more precisely with the formula obtained in the second arXiv version of this paper. Let $G_N$ be a GOE of variance $t>0$. We study the smallest eigenvalue of $G_N+\Lambda e_1 e_1^T$. Note that $\sigma_t=\sigma_t\boxplus\delta_0$. Therefore, the characteristic associated to the Burgers equation, see Remark \ref{remark:chara}, is 
\begin{equation*}
    H_{0,t}(x)=x+\frac{t}{x}.
\end{equation*}
Let $\omega_{0,t}$ be the inverse of $H_{0,t}$ on $(-\infty,-2\sqrt{t}]$, namely
\begin{equation*}
    \omega_{0,t}:\lambda\in (-\infty,-2\sqrt{t}]\mapsto \frac{\lambda-\sqrt{\lambda^2-4t}}{2}.
\end{equation*}
Notice that the support of $\sigma_t$ is $[-2\sqrt{t},2\sqrt{t}]$.
With the normalization of \cite{maida2007large}, taking $\beta=1$ gives an LDP for the smallest eigenvalue of $G_N+\Lambda e_1 e_1^T$ where $G_N$ is GOE of variance $t=1/2$. In the sequel we take $t=1/2$. Denote $\rho_{0,1/2}^\Lambda:=\Lambda+\frac{1}{2\Lambda}$. Let $J_\Lambda$ be the rate function found in \cite{maida2007large}, given if $\Lambda \leq -\frac{1}{\sqrt{2}}$ by
\begin{equation*}
J_\Lambda(x):=\begin{cases}M_\Lambda(x)& \text{if $x\leq -\sqrt{2}$}\\
    +\infty &\text{if $x>-\sqrt{2}$},
    \end{cases}
\end{equation*}
where
\begin{equation*}
   M_\Lambda(x):=\frac{1}{2}\int_{x}^{\rho_{0,1/2}^\Lambda} \sqrt{z^2-2}\dd z-\Lambda (x-\rho_{0,1/2}^\Lambda)+\frac{1}{4}(x^2-(\rho_{0,1/2}^\Lambda)^{ 2})
\end{equation*}
and if $\Lambda\geq -\frac{1}{\sqrt{2}}$ by
\begin{equation*}
    J_\Lambda(x):=\begin{cases}N_\Lambda(x) & \text{if $x\leq \rho_{0,1/2}^\Lambda$}\\ 
    \int_{x}^{-\sqrt{2}}\sqrt{z^2-2}\dd z & \text{if $\rho_{0,1/2}^\Lambda\leq x\leq -\sqrt{2}$}\\
  +\infty  & \text{if $x> -\sqrt{2}$},
    \end{cases}
\end{equation*}
where
\begin{equation*}
  { {N_\Lambda}(x):=\frac{1}{2}\int_{x}^{-\sqrt{2}}\sqrt{z^2-2}\dd z-\Lambda x+\frac{1}{4}x^2+\frac{1}{4}+\frac{1}{4}\log 2+\frac{\Lambda^2}{2}
  + \frac{1}{2}\ln |\Lambda|.}
\end{equation*}

\begin{lemma}
 We have
\begin{equation*}
  J_\Lambda=I_{0,1/2}^{\Lambda}
\end{equation*}
where we denoted in short $I_{0,1/2}^{\Lambda}$ for $I_{\delta_{0},1/2}^{\Lambda}$. 
\end{lemma}

\begin{proof}
For all $x<-\sqrt{2}$, let $\omega_{0,1/2}(x)\leq \omega_{0,1/2}^{*}(x)$ be the two solutions in $\dR^-$ of the equation $\gamma+\frac{1}{2\gamma}=x$, namely
\begin{equation*}
    \omega_{0,1/2}(x)=\frac{x-\sqrt{x^2-2}}{2}\quad \text{and}\quad \omega_{0,1/2}^{*}(x)=\frac{x+\sqrt{x^2-2}}{2}.
\end{equation*}
Let $x<-\sqrt{2}$. Assume that $\Lambda\leq -\frac{1}{\sqrt{2}}$, or $\Lambda\geq -\frac{1}{\sqrt{2}}$ and $x<\rho_{0,1/2}^\Lambda$. {One can notice that
\begin{equation*}
 J_\Lambda'(x)=-\frac{1}{2}\sqrt{x^2-2}+\frac{x}{2}-\Lambda.
\end{equation*}}Moreover, by \eqref{ty}, we have
\begin{equation*}
(I_{0,1/2}^{\Lambda})'(x)=\omega_{0,1/2}(x)-\Lambda=-\frac{1}{2}\sqrt{x^2-2}+\frac{x}{2}-\Lambda= J_\Lambda'(x).
\end{equation*}
Since $I_{0,1/2}^{\Lambda}(\rho^{\Lambda}_{0,1/2})=J_\Lambda(\rho^{\Lambda}_{0,1/2})$, equality follows for all $x\le -\sqrt{2}$ if $\Lambda\leq -\frac{1}{\sqrt{2}}$ and for all $x<\rho_{0,1/2}^\Lambda$ if $\Lambda\geq -\frac{1}{\sqrt{2}}$.

We finally consider the last case where $\Lambda\geq -\frac{1}{\sqrt{2}}$ and $x\in (\rho_{0,1/2}^\Lambda,-\sqrt{2})$. {By Equation \eqref{eq:F2}},
\begin{equation*}
(I_{0,1/2}^{\Lambda})'(x)=\omega_{0,1/2}(x)-\omega_{0,1/2}^{*}(x)=-\sqrt{x^2-2}=J_\Lambda'(x).
\end{equation*}
Again, we conclude that $J_\Lambda=I_{0,1/2}^{\Lambda}$.
\end{proof}

\subsection{The no outlier case}

Let us now compare our result with \cite{mckenna2021large} where no outlier is permitted. Again we take $\beta=1$ in \cite{mckenna2021large}. In the normalization of \cite{mckenna2021large}, the variance of the non-diagonal entry distribution is $1$. In our setting \eqref{var:GOEt}, this corresponds to taking $t=1$ (instead of $t=1/2$ in the normalization of \cite{maida2007large}).

We apply Theorem \ref{theorem:LDP} to $\Lambda:=\ell_{\nu}$. The rate function of the LDP of \cite{mckenna2021large} (stated there for the largest
eigenvalue) can be rewritten for the smallest eigenvalue by applying it to $-X_N$ and
changing variables $\theta\mapsto -\theta$. It takes the form
\begin{equation}\label{eq:equation1}
    J_\nu(x)=\sup_{\theta\le 0}J(x,\theta),
\end{equation}
where for all $x\in \dR$, $\theta\le 0$,
\[
    J(x,\theta):=J(\theta,\nu\boxplus\sigma_1,x)-J(\theta,\nu,x)-\theta^2,
\]
and for all $x\in \dR$, $\theta\le 0$ and probability measure $\mu$ on $\dR$,
\[
     J(\theta,\mu,x):=\begin{cases}
        \frac{1}{2} \int_0^{2\theta } R_\mu(s)\dd s
        & \text{if $G_\mu(x)\leq 2\theta\leq 0$}\\[0.2em]
         \theta x-\frac{1}{2}(1+\log(-2\theta))
         -\frac{1}{2}\int \log|x-y|\dd\mu(y)
         & \text{if $2\theta<G_\mu(x)$},
     \end{cases}
\] 
where $R_\mu$ is the Voiculescu $R$-transform, which is given by $R_\mu:y\mapsto K_\mu(y)-\frac{1}{y}$ with $K_\mu$ as in \eqref{def:Knu}.

\begin{lemma}
 We have
    \begin{equation*}
    {J_\nu=I_{\nu,1}^{\ell_{\nu}}.}
    \end{equation*}
\end{lemma}

\medskip

\begin{proof}
Let $\rho_{\nu,{1}}^{\ell_\nu}:=\ell_{\nu}+G_\nu(\ell_{\nu})\in \dR\cup\{-\infty\}$. For all $x\in (\rho_{\nu,1}^{\ell_\nu}, \ell_{\nu,1})$, let $\gamma_1(x)\leq \gamma_2(x)$ be the two solutions of the equation $H_{\nu,1}(\gamma)=x$ on $(-\infty,\ell_{\nu})$. 
First observe that for all $x<\ell_{\nu,1}$, 
\begin{equation}\label{eq:J'}
  J_\nu'(x)=\partial_x J(x,\theta(x))+\partial_\theta J(x,\theta)\theta'(x)=\partial_x J(x,\theta(x))=
\theta(x)-\frac{1}{2}G_{\nu\boxplus\sigma_1}(x)=\theta(x)-\frac{1}{2}G_\nu(\gamma_1(x)).
\end{equation}
where we have applied Lemma \ref{lemma:free co} in the last equality.

Let $x\in (\rho_{\nu,{1}}^{\ell_\nu},\ell_{\nu,1})$. As shown in \cite[(2.14)]{mckenna2021large}, for $x\in (\rho_{\nu,1}^{\ell_\nu},\ell_{\nu,1})$, the optimal $\theta(x)$ in \eqref{eq:equation1} is the unique solution of the equation
\begin{equation*}
    2\theta(x)+K_\nu(2\theta(x))=x,\quad \theta(x) \in \Bigr(\frac{1}{2}G_\nu(\ell_{\nu}),\frac{1}{2}G_{\nu\boxplus\sigma_1}(\ell_{\nu,1})\Bigr),
\end{equation*}
where $K_\nu$ is as in \eqref{def:Knu}. One can easily show 
\begin{equation*}
    \theta(x)=\frac{1}{2}G_\nu(\gamma_2(x)).
\end{equation*}
Substituting this into \eqref{eq:J'}, we find that for all $x\in(\rho_{\nu,1}^{\ell_\nu},\ell_{\nu,1})$,
\begin{equation*}
   J_\nu'(x)=\frac{1}{2}(G_\nu(\gamma_2(x))-G_\nu(\gamma_1(x)))=\frac{1}{2}(\gamma_1(x)-\gamma_2(x))=(I_{\nu,1}^{\ell_{\nu}})'(x),
\end{equation*}
where we have used \eqref{eq:F2} in the last equality.  

If $x\leq \rho_{\nu,1}^{\ell_\nu}$, then by \cite{mckenna2021large},
\begin{equation*}
    \theta(x)=\frac{1}{2}(x-\ell_{\nu}),
\end{equation*}
which implies that for all $x<\rho_{\nu,1}^{\ell_\nu}$,
\begin{equation*}
    J_\nu'(x)=\frac{1}{2}(x-{\ell_\nu}-G_\nu(\gamma_1(x)))=\frac{1}{2}(\gamma_1(x)-{\ell_\nu})=(I_{\nu,1}^{\ell_{\nu}})'(x).
\end{equation*}
Since $J_\nu$ and ${I_{\nu,1}^{\ell_\nu}}$ are $\mathcal{C}^1$ on $(-\infty,\ell_{\nu,1})$ and $\inf J_\nu=\inf {I_{\nu,1}^{\ell_\nu}}=0$, this proves that $J_\nu=I_{\nu,1}^{\ell_{\nu}}$.
\end{proof}

\section{Generalization to a deformed GUE matrix}\label{sec:GUE}
In this section we discuss why the proof of Theorem \ref{theorem:LDPGUE} is the same as for GOE (which is in general more complicated to study). It is enough to show that the volume of small balls satisfies the same functional equations and we discuss how the heuristics of Section \ref{sub:proof ideas}
are modified.  In fact, we have the same type of formula
\begin{equation*}
 \mathbb P( \lambda_1(X_N) \in (x - \varepsilon, x + \varepsilon))=\frac{1}{Z_{N}^{t}}  \int_{\lambda_1(X_N) \in (x - \varepsilon, x + \varepsilon)} e^{-\frac{N}{2t} \operatorname{Tr}((X_N - B_N)^2)} \, \dd X_N,
\end{equation*}
except that now $X_{N}$ is Hermitian,  we have a term $\tfrac{1}{2t}$ instead of $\tfrac{1}{4t}$ to keep entries with variance $t/N$ and the Vandermonde comes with a factor $2$: so in the formulas $t$ has to be divided by two. 
Equation \eqref{heu1} extends verbatim, with $O$ now unitary. $v_{1}$ is now a vector in $\mathbb C^{N}$ and as a result $Y_{N}$ will follow a Dirichlet distribution with parameters multiplied by $2$, see    \cite[Theorem 3.3]{HuGu2}. Moreover, because of the change in $t$, $L^{\Lambda}_{\nu,t}$ of \eqref{defL} is multiplied by $2$. Finally the constant $C_{t}$ computed by Selberg integral is also multiplied by two.
We therefore find that with the same notation as in \eqref{heu2}, we have the following generalization of \eqref{heu3}
$$F_{\nu,t}^-(\Lambda, x) \lesssim \inf_{Y \in \Delta^{p+1}} \Bigl\{2 L_{\nu,t}^{\Lambda}(x,Y)
    -  \sum_{k = 1}^p \alpha_k \log (y_k / \alpha_k)+ C_t - 2\int \log |\lambda - x| \, d(\nu \boxplus \sigma_t)(\lambda) + \frac{x^2}{2t} + \inf_{y \geq x} F_{\nu,t}^-(\Phi(\Lambda, Y), y) \Bigr\}\,.
$$
We therefore conclude that $F_{\nu,t}^-(\Lambda, x) /2$ satisfies the same inequality as $F_{\nu,t}^-(\Lambda, x)$ in the GOE case. Similarly  $F_{\nu,t}^+(\Lambda, x) /2$ satisfies the same inequality as $F_{\nu,t}^+(\Lambda, x)$ in the GOE case. Thus, solving these functional inequalities exactly as before yields the weak large deviation principle with rate function $2I^{\Lambda}_{\nu,t}$.

\begin{appendix}
\section{Reminder on {the} free convolution}\label{sec:freeconv}
We recall some classical results on free convolution from \cite{biane} that are useful throughout this article.
Recall the definitions of $v_{t}(u)$ and $\Omega_{\mu,t}$ in \eqref{eq:defvu}.
We have the following:

\begin{lemma}[Biane]\label{lemma:free co}
Let $\mu$ be a probability measure on $\dR$. Let $t>0$. The function $H_{\mu,t}$ is a homeomorphism, which is conformal from $\Omega_{\mu,t}$ onto $\mathbb{C}^+$. We let $\omega_{\mu,t}:{\mathbb{C}^+\cup\dR}\to \overline{\Omega_{\mu,t}}$ be its inverse, called the subordination function. Moreover, for every $z\in \Omega_{\mu,t}$, 
 \begin{equation}\label{eq:method of car}
     G_{\mu\boxplus \sigma_t}(z+tG_\mu(z))=G_\mu(z).
 \end{equation}   
 Assume that the support of $\mu$ is bounded from below. The left edge of $\mu\boxplus\sigma_t$ is given by $\ell_{\mu,t}=H_{\mu,t}(\omega)$ where 
 \begin{equation}\label{eq:shock}
     \omega:=\inf\Bigr\{u\le \ell_{\mu} : \int \frac{\dd\mu(\lambda)}{(\lambda-u)^2}\geq \frac{1}{t}\Bigr\}.
 \end{equation}
For every $x<\ell_{\mu,t}$, we have
 \begin{equation}\label{eq:omegax exp}
     \omega_{\mu,t}(x)=x-tG_{\mu\boxplus\sigma_t}(x).
 \end{equation}
\end{lemma}

\begin{remark}\label{remark:chara}
Lemma \ref{lemma:free co} admits a reformulation in PDE terms. The Stieltjes transform of $\mu_t:=\mu\boxplus \sigma_t$ indeed satisfies a complex Burgers equation: for any $z$ outside the support of $\mu_t$ we have
\begin{equation}\label{eq:Burgers}
    \partial_t G_{\mu_t}(z)+G_{\mu_t}(z)\partial_z G_{\mu_t}(z)=0.
\end{equation}

This equation can be solved via the method of characteristics. The characteristic starting from $z\in \mathbb{C}^+$ is given by the straight line $t\mapsto z+tG_\mu(z)=H_{\mu,t}(z)$. Equation \eqref{eq:method of car} reflects the fact that the solution is constant along characteristics (after restriction to suitable $z$). Let $\omega\in (-\infty,\ell_{\mu}].$ One may note that $H_{\mu,t}$ is injective on $(-\infty,\omega)$ provided 
\begin{equation*}
    \int \frac{\dd\mu(\lambda)}{(\lambda-\omega)^2}<\frac{1}{t}.
\end{equation*}
From \eqref{eq:shock}, the left edge of $\mu_t$ can be interpreted as the first value at which the solution develops a shock at time $t$. 
\end{remark}

	\begin{figure}[H] \centering
		\begin{subfigure}[b]{0.45\textwidth}
			\centering
			\begin{tcolorbox}[colframe=black, colback=white, boxrule=0.5pt] 
				\includegraphics[width=\textwidth]{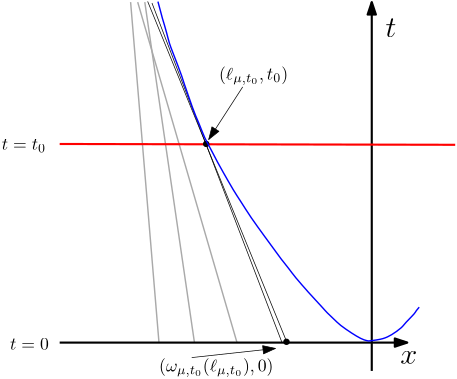}
			\end{tcolorbox}
                 \subcaption{Characteristics and left edge of $\mu\boxplus \sigma_t$.}
		\end{subfigure}
		\hspace{0.5cm}
		\begin{subfigure}[b]{0.45\textwidth}
			\centering
			\begin{tcolorbox}[colframe=black, colback=white, boxrule=0.5pt]
				\includegraphics[width=\textwidth]{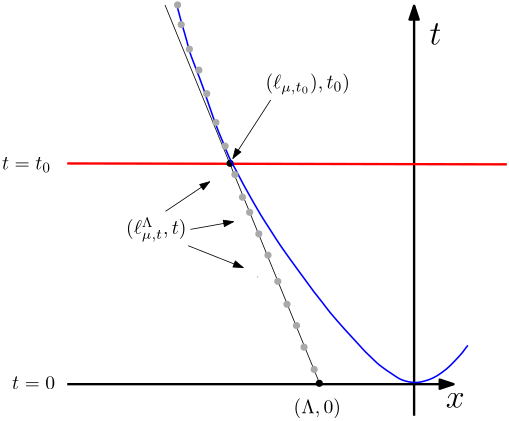}
			\end{tcolorbox}
             \subcaption{The BBP transition.}
		\end{subfigure}
        \caption{Characteristics of the free convolution flow and illustration of the BBP transition.}
        \label{figure:caract}
	\end{figure}

In Figure \((\mathrm{a})\), the gray lines show characteristics \(\{(H_{\mu,t}(x),t): t>0\}\) for fixed \(x\) (with \(x\) indicated on the horizontal axis). Characteristics that originate near \((\omega_{\mu,t}(\ell_{\mu,t}),t)\) coalesce at time \(t\); see the discussion in Remark \ref{remark:chara}. This coalescence corresponds to the point where the two solution branches merge. The blue curve is the locus of these merging points and traces the left edge \(\ell_{\mu,t}\) of \(\mu\boxplus\sigma_t\).
Figure \((\mathrm{b})\) illustrates the BBP transition. The \(x\)-coordinates of the gray dots are the limiting smallest eigenvalue \(\ell_{\mu,t}^\Lambda\) of \(G_N + D_N\), where \(G_N\) is a GOE matrix with variance \(t\) and \(D_N\) is as above; see Theorem \ref{convth}. The BBP transition occurs at the point where the gray dots move from lying off the blue curve to lying on it.

\begin{property}\label{remark:twosol}
By Lemma \ref{lemma:free co}, the map $H_{\mu,t}|_{(-\infty,\ell_{\mu})}$ is strictly concave, strictly increasing on the set $(-\infty,\omega_{\mu,t}(\ell_{\mu,t}))$ and strictly decreasing on the set $(\omega_{\mu,t}(\ell_{\mu,t}),\ell_\mu)$. In addition $H_{\mu,t}$ goes to $-\infty$ when $x$ goes to $-\infty$ and to $H_{\mu,t}(\ell_{\mu})\in \dR\cup\{-\infty\}$ when $x$ goes to $\ell_{\mu}$.

It follows that, for all $\lambda \in (H_{\mu,t}(\ell_{\mu}),\ell_{\mu,t})$, the equation 
\begin{equation*}
    H_{\mu,t}(\gamma)=\lambda, \quad \gamma <\ell_{\mu}
\end{equation*}
admits two solutions $\gamma_1<\gamma_2$. Moreover {if $\lambda \in (H_{\mu,t}(\ell_{\mu}),\ell_{\mu,t})$}, we have $\gamma_1=\omega_{\mu,t}(\lambda)<\omega_{\mu,t}(\ell_{\mu,t})$ and {$\gamma_2=\omega_{\mu,t}^*(\lambda)>\omega_{\mu,t}(\ell_{\mu,t})$}.
\end{property}

Throughout this paper, we used several smoothness results that we summarize and prove here:
\begin{lemma}\label{smoothness} 
\begin{enumerate}
\item  Let $M$ be a positive finite real number. Then $\mu\mapsto \mu\boxplus\sigma_{t}$ is continuous for the weak topology on $\mathcal P([-M,M])$. In fact recalling the distance $d$ from \eqref{def:distance}, we have
$$d(\mu\boxplus\sigma_{t},\mu'\boxplus\sigma_{t})\le \left(\int_{0}^{1}|\eta_{\mu}(s)-\eta_{\mu'}(s)|^{2}\dd s\right)^{\frac{1}{2}}.$$
\item Moreover, for $\eta_1<\cdots< \eta_{p}$, denoting  $\nu_{\alpha}=\sum\alpha_{i}\delta_{\eta_{i}}$,  we have for every $t>0$ and $\alpha,\alpha'\in\Delta^{p}$,  
$$d(\nu_{\alpha}\boxplus\sigma_{t}, \nu_{\alpha'}\boxplus\sigma_{t})\le \left( \sum_{i=1}^p \eta_{i}^{2} |\alpha_{i}-\alpha_{i}'|\right)^{1/2}\,.$$
\item  $\Lambda\in(-\infty,\ell_{\nu})\mapsto \ell_{\nu,t}^{\Lambda}$ is continuous.
\item For every probability measure $\nu\in \mathcal P([-M,M])$, $x\mapsto \omega_{\nu,t}(x)$ is continuous and increasing on $(-\infty,\ell_{\nu,t}]$ with values in $(-\infty,\omega_{\nu,t}(\ell_{\nu,t})]$, and $x\mapsto \omega_{\nu,t}^{*}(x)$ is continuous and decreasing on   $[H_{\nu,t}(\ell_{\nu}),\ell_{\nu,t})$ with values in $[\omega_{\nu,t}(\ell_{\nu,t}),\ell_{\nu}]$.
\item  For every probability measure $\nu\in \mathcal P([-M,M])$, let $\eta_{\nu}$ be {such that} $\nu=\eta_{\nu}\# 1_{[0,1]}\dd x$. Then,  for every $\nu,\nu'\in  \mathcal P([-M,M])$, 
$$d(\nu\boxplus\sigma_{t}, \nu'\boxplus\sigma_{t})\le \|\eta_{\nu}-\eta_{\nu'}\|_{\infty} \mbox{ and } |\ell_{\nu,t}-\ell_{\nu',t}|\le\|\eta_{\nu}-\eta_{\nu'}\|_{\infty}\,.$$
Moreover if $x< \ell_{\nu,t}$ (resp. $x\in [H_{\nu,t}(\ell_{\nu}),\ell_{\nu,t})$)
, $\omega_{\nu,t}(x)-\omega_{\nu',t}(x)$ (resp.   $\omega^{*}_{\nu,t}(x)-\omega^{*}_{\nu',t}(x)$)  go to zero when $\|\eta_{\nu}-\eta_{\nu'}\|_{\infty}$ goes to zero. 
\end{enumerate}
\end{lemma}

\medskip

\begin{proof}

$\bullet$  The  first point is a consequence of \eqref{con200} by taking sequences of matrices $B_{N},B_{N}'$ so that the empirical measure of their eigenvalues converge towards $\mu$ and $\mu'$ respectively.   Similarly, the
second point is an application 
  of \eqref{con20} by taking $N$ going to infinity while $\bM$ and $\bM'$ are so that $M_{i}/N$ (resp. $M_{i}'/N$) goes to $\alpha_{i}$ (resp. $\alpha_{i}'$).  
\medskip

$\bullet$ To prove the third  point, observe that  the function $H_{\nu,t}$ is smooth on $(-\infty,\ell_{\nu})$, strictly increasing on $(-\infty,\omega)$ and strictly decreasing on $(\omega,\ell_{\nu})$ with ${\omega\leq \ell_\nu}$ the point where
$$\int \frac{{\dd\nu(x)}}{(\omega-x)^{2}}=\frac{1}{t},$$
if it is strictly smaller than $\ell_{\nu}$, $\omega=\ell_{\nu}$ otherwise.   The continuity of $\ell^{\Lambda}_{\nu,t}$ follows.

$\bullet$
Because the derivative of $H_{\nu,t}$ is strictly positive on $(-\infty,\omega)$, the implicit function theorem shows that  $H_{\nu,t}$ has an inverse  $\omega_{\nu,t}$ with values in $(-\infty,\omega)$ which is continuous. Similarly, 
because the derivative of $H_{\nu,t}$ is strictly negative on $(\omega,\ell_{\nu})$, $H_{\nu,t}$
 has a second  continuous inverse with values in $(\omega,\ell_{\nu})$.
By \cite{biane}, we know that $\ell_{\nu,t}=H_{\nu,t}(\omega)$ so $\omega=\omega_{\nu,t}(\ell_{\nu,t})$. 
\medskip

$\bullet$ For the last point, we can prove the bound on $d(\nu\boxplus\sigma_{t}, \nu'\boxplus\sigma_{t})$ again by coupling them as the distribution of $\eta_{\nu}(\Lambda)+S_{t}$ and $\eta_{\nu'}(\Lambda)+S_{t}$ where $\Lambda$ is distributed according to $1_{[0,1]}\dd x$ and is free from $S_{t}$. The result follows. Moreover, $\ell_{\nu}=\inf\{\eta_{\nu}\}$ is continuous in $\eta_{\nu}$. Furthermore $\ell_{\nu,t}$ is the smallest eigenvalue of $X_{t}(\nu)=\eta_{\nu}(\Lambda)+S_{t}$ and $\|X_{t}(\nu)-X_{t}(\nu')\|_{\mathrm{op}}\le\|\eta_{\nu}-\eta_{\nu'}\|_{\infty}$ so that $|\ell_{\nu,t}-\ell_{\nu',t}|\le \|\eta_{\nu}-\eta_{\nu'}\|_{\infty}$. {The last statement regarding the continuity of $\omega_{\nu,t}(x)$ and $\omega^{*}_{\nu,t}(x)$ can be deduced from the following observations: First, we can assume $x<\min\{\ell_{\nu,t},\ell_{\nu',t}\}$, which ensures that both $\omega_{\nu,t}(x)$ and $\omega^{*}_{\nu,t}(x)$ are away from $\ell_{\nu}$. Additionally, the inverse $H_{\nu,t}$ has a non-vanishing derivative on this range. Therefore, by applying the implicit function theorem and considering that $H_{\nu,t}-H_{\nu',t}$ is small, we obtain the desired result.}
\end{proof}

{The PDE interpretation discussed in Remark \ref{remark:chara} allows one to compute the logarithmic potential of the measure $\mu\boxplus\sigma_t$ through the Hopf-Lax formula} (see also \cite[Theorem 1.1]{BeliusCG})

\begin{lemma}\label{lemma:exp det}
{Let $\mu$ be a probability measure on $\dR$ with support bounded from below and $t>0$.
\begin{enumerate}
\item Let $x\in\dR$ and $\omega:=\omega_{\mu,t}(x)\in\mathbb{C}^+\cup\dR$. Let $(z_s)_{s\in [0,t]}$ given for all $s\in [0,t]$ by
\begin{equation}\label{eq:zs}
    z_s=\omega+s G_\mu(\omega).
\end{equation}
If $x\leq \ell_{\mu,t}$, then for all $s\in [0,t]$,  
\begin{equation}\label{eq:omegamu}
   x\leq z_s\leq \ell_{\mu\boxplus\sigma_s}.
\end{equation}
\item Let $x\in \dR$. Assume that $x\leq \ell_{\mu,t}$. We have
\begin{equation}\label{eq:ch}
    \int \log|\lambda-x|\dd (\mu\boxplus \sigma_t)(\lambda)=\int \log |\lambda-\omega_{\mu,t}(x)|\dd \mu(\lambda) +\frac{(\omega_{\mu,t}(x)-x)^2 }{2t}.
\end{equation}
\end{enumerate}}
\end{lemma}

{\begin{proof}
Let $x\in\dR$ and $\omega:=\omega_{\mu,t}(x)$. Consider
\begin{equation*}
z_s=\omega+sG_\mu(\omega),\quad 0\leq s\leq t.
\end{equation*}

Assume first that $x\leq \ell_{\mu\boxplus\sigma_t}$. First by {\cite[Lemma 11]{guionnet2020largeb}}, we have
\begin{equation*}
\omega_{\mu,t}((-\infty,\ell_{\mu\boxplus \sigma_t}])\subset (-\infty,\ell_{\mu}]
\end{equation*}
and therefore $\omega\leq \ell_{\mu}$. This proves that $z_0\leq \ell_{\mu}$. Let $s\in (0,t)$. Recall that $\mu\boxplus\sigma_t=(\mu\boxplus\sigma_s) \boxplus \sigma_{t-s}$. We have
\begin{equation*}
    \omega_{\mu\boxplus\sigma_s,t-s}(x)=z_s.
\end{equation*}
Indeed,
\begin{equation*}
    z_s+(t-s)G_{\mu\boxplus\sigma_s}(z_s)=z_s+(t-s)G_\mu(\omega)=\omega+sG_\mu(\omega)+(t-s)G_\mu(\omega)=x.
\end{equation*}
Therefore, applying the previous result to $t'=t-s$ and $\mu'=\mu\boxplus\sigma_s$, we obtain
\begin{equation*}
    z_s\leq \ell_{\mu\boxplus \sigma_s}.
\end{equation*}
Because $\omega\leq \ell_{\mu}$, $G_{\mu}(\omega)$ is non-positive  and therefore $z_s$ is decreasing which shows that $z_s\geq z_t=x$ for all $0\leq s\leq t$. This completes the proof of \eqref{eq:omegamu}.

Fix $x\in\dR$ and let $\omega:=\omega_{\mu,t}(x)$. For all $s\geq 0$, let $m_s:=G_{\mu\boxplus\sigma_s}$. It is well-known, see Lemma \ref{lemma:free co}, that for all $z\in\mathbb{C}^+\cup\dR$ outside the support of $\mu\boxplus\sigma_s$,
\begin{equation*}
    \partial_s m_s(z)+m_s(z)\partial_zm_s(z)=0
\end{equation*}
and that $m_s(z_s)=m_0(z_0)=G_\mu(\omega)$.  
Integrating this equality over $z\in (-\infty, \ell_{\mu,s})$ we deduce that there exists a constant $C_{s}$ independent of $z$ such that
\begin{equation}\label{qwe}
   \partial_s  \int \log|\lambda-z|\dd(\mu\boxplus \sigma_s)(\lambda)+\frac{1}{2} m_{s}(z)^{2}=C_{s}\,.
\end{equation}
Furthermore, taking the limit $z$ going to $-\infty$ shows that $C_{s}$ must vanish. 
Define
\begin{equation*}
      F:(s,z)\in \bigcup_{s'> 0}\{(s',z'):z'\in \mathbb{C}^+\cup (-\infty,\ell_{\mu,s'}) \}\mapsto \int \log|\lambda-z|\dd(\mu\boxplus \sigma_s)(\lambda)-\frac{(z-\omega)^2}{2s},
  \end{equation*}
  extended to $s=0$ by $F(0,z):=\int \log|\lambda-z|\,\dd\mu(\lambda)$ for $z\in\dR\setminus\supp(\mu)$ (so that $F(0,z_{0})=\int\log|\lambda-\omega|\,\dd\mu(\lambda)$, since $z_{0}=\omega$).

Let $(z_s)_{s\in [0,t]}$ be as in \eqref{eq:zs}. For every $z\in \mathbb{C}^+\cup (-\infty,\ell_{\mu,t})$, we find by \eqref{qwe} that
  \begin{equation*}
      \partial_s F(s,z)=-\frac{1}{2}m_s(z)^2+\frac{(z-\omega)^2}{2s^2}.
  \end{equation*}
As proved in item (1), for all $s\in [0,t)$, $z_s<\ell_{\mu,s}$. Hence, since $m_{s}(z_{s})=m_0(z_0)=G_\mu(\omega)$ and $z_{s}-\omega=s G_{\mu}(\omega)$,
\begin{equation*}
  \partial_s F(s,z_s)=0\,.
\end{equation*}
It follows that
\begin{equation*}
\frac{\dd }{\dd s}F(s,z_s)=0,
\end{equation*}
which gives  after integrating $s\in [0,t]$,  $F(0,\omega)=F(0,z_{0})=F(t,z_{t})=F(t,x)$, or equivalently
\begin{equation*}
 \int \log|\lambda-x|\dd (\mu\boxplus \sigma_t)(\lambda)=\int \log |\lambda-\omega_{\mu,t}(x)|\dd \mu(\lambda) +\frac{(\omega_{\mu,t}(x)-x)^2}{2t},
\end{equation*}
which proves \eqref{eq:ch}.
\end{proof}}

\section{A priori properties of {the rate function}}

{\begin{lemma}\label{monot} 
Let $(B_N)$ satisfy Assumptions \ref{assumption:B} with $\nu$ and $\Lambda$ as in \eqref{eq:a_imu} and \eqref{eq:Lambdalim}. Let $\mathbb{P}$ be the law of the eigenvalues of $B_N+G_N$ where $G_N$ is an $N\times N$ GOE matrix with variance {$t$}. Recall 
\begin{equation*}
\ell_{\nu,t}^\Lambda:=\begin{cases}
        H_{\nu,t}(\Lambda) & \text{if $\Lambda\leq \omega_{\nu,t}(\ell_{\nu,t})$}\\
        \ell_{\nu,t} & \text{if $\Lambda\geq \omega_{\nu,t}(\ell_{\nu,t})$}.
    \end{cases}
\end{equation*}
Let
$x<y<\ell_{\nu,t}^\Lambda$ or $\ell_{\nu,t}^\Lambda <y<x$. Then, for every $\delta\in (0,|\ell^{\Lambda}_{\nu,t}-y|)$:
$$\limsup_{N\rightarrow\infty}\frac{1}{N}\log \mathbb P (|\lambda_{1}-x|\le\delta)\le \limsup_{N\rightarrow\infty}\frac{1}{N}\log \mathbb P(|\lambda_{1}-y|\le\delta)\,.$$
\end{lemma}}

\begin{proof} 
We may and shall assume that $\limsup_{N\rightarrow\infty}\frac{1}{N}\log \mathbb P(|\lambda_{1}-x|\le\delta)$ is finite since otherwise we are done. 
The proof follows  the arguments of {\cite[Lemma 5.2]{cook2023full}} based on {\cite[Lemma 8.1]{cook2023full}}  where a coupling is constructed. Because we  restrict ourselves to  Gaussian entries in this article, we propose a simpler version based on the symmetric Ornstein-Uhlenbeck 
$${\dd G_{N}(s)=\dd H_{N}(s)-\frac{1}{2t} G_{N}(s) \dd s}$$
where $H_{N}$ is an $N\times N$ symmetric Brownian motion and $G^{N}_{0}=G_{N}$ is a GOE with variance $t$.
We also denote by $\bar G_{N}(s)$ the solution constructed with the same Brownian motion but starting from zero. It is well known that $G_{N}(s)$ has the law of a GOE  with variance $t$ at any time $s$, whereas $\bar G_{N}(s)$ has  the law of a GOE with variance $t(1- e^{-s/t})$.  Moreover,
$${\dd (G_{N}(s)-\bar G_{N}(s))^{2}=-\frac{1}{t}(G_{N}(s)-\bar G_{N}(s))^{2} \dd s}$$
implies that
\begin{equation}\label{dis}\|G_{N}(s)-\bar G_{N}(s)\|_{\mathrm{op}}\le \|G_{N}\|_{\mathrm{op}} e^{-s/2t}\end{equation}
where $\|.\|_{\mathrm{op}}$ denotes the operator norm. {We will later take $s=N$.}
{Moreover we work on the event where $\|G_{N}\|_{\mathrm{op}}\le N$ so that $\|G_{N}(s)-\bar G_{N}(s)\|_{\mathrm{op}}$ is of order $Ne^{-s/2t}\le \delta/2$ if $s$ is of order $N$ (here $t$ is fixed, independently of $N$). We finally set $X_{N}(s):=B_{N}+G_{N}(s)$ and $\bar X_{N}(s):=B_{N}+\bar G_{N}(s)$.}
{Note that for every integer $p$,  $X_{N}(\frac{p}{N})$ has the same law as $X_{N}(0)$ and in particular its smallest eigenvalue converges towards $\ell_{\nu,t}^\Lambda$ almost surely and in expectation.} Observe also that by \eqref{dis}, $\|\bar X_{N}(N)-X_{N}(N)\|_{\mathrm{op}} $ is bounded by $\|G_{N}\|_{\mathrm{op}}e^{-N/2t}$ so that
$$| \mathbb E[\lambda_{1}(\bar X_{N}(N))]-\mathbb E[\lambda_{1}( X_{N}(N))]|\le \mathbb E[ \|G_{N}\|_{\mathrm{op}}]e^{-N/2t}$$
goes to zero as $N$ goes to infinity.
Hence, 
$\mathbb E[\lambda_{1}(\bar X_{N}(N))]$ converges as $N$ goes to infinity towards ${\ell}_{\nu, t }^{\Lambda}$. By concentration of measure for $\lambda_{1}(\bar X_{N}(N))$ {(see \cite[Theorem 2.3.5]{AGZ})}, we also have {that} for every $\eta>0$, 
\begin{equation}\label{conc}
\mathbb P( |\mathbb E[\lambda_{1}(\bar X_{N}(N))]-\lambda_{1}(\bar X_{N}(N))|\ge \eta)\le 2e^{-\frac{N\eta^{2}}{4t}}{.}
\end{equation}
We deduce that for $N\gg \log \delta^{-1}$ large enough,
$$Q_{N}:=\mathbb P\left( |\lambda_{1}( X_{N}(N))-{\ell}_{\nu,t}^{\Lambda}|\le \delta\big| |\lambda_{1}(X_N(0))-x|\le\delta, \|G_{N}\|_{\mathrm{op}}\le N\right)\ge 1/2{.}$$
Indeed,  \eqref{dis} implies that $\|X_{N}(N)-\bar X_{N}(N)\|_{\mathrm{op}}$ is of order $Ne^{-N/2t}$ when the operator norm of $G_{N}$ is bounded by $N$, for $N$ large enough so that $Ne^{-N/2t}\le \delta/2$, 
\begin{equation*}
\begin{split}
Q_{N}&\ge\frac{ \mathbb P(\{ |\lambda_{1}( \bar X_{N}(N))-{\ell}_{\nu,t}^{\Lambda}|\le \delta/2\}\cap\{  |\lambda_{1}(X_N(0))-x|\le\delta, \|G_{N}\|_{\mathrm{op}}\le N\})}{ \mathbb P( |\lambda_{1}(X_N(0))-x|\le\delta, \|G_{N}\|_{\mathrm{op}}\le N\})}\\
&=\mathbb P(\{ |\lambda_{1}( \bar X_{N}(N))-{\ell}_{\nu,t}^{\Lambda}|\le \delta/2\}){.}
\end{split}
\end{equation*}
where we finally  used that $\bar X_{N}$ is independent of $X_{N}(0)=B_N+G_{N}$. Finally, we have seen 
 that the above  {right-hand} side goes to {$1$} as $N$ goes to infinity.
The fact that $Q_{N}$ is bounded {from} below by $1/2$  implies that
\begin{equation*}
\begin{split}
& \mathbb P( |\lambda_{1}(X_N(0))-x|\le\delta, \|G_{N}\|_{\mathrm{op}}\le N)\\
&\le 2  \mathbb P(\{ |\lambda_{1}(  X_{N}(N))-{\ell}_{\nu,t}^{\Lambda}|\le \delta\}\cap\{  |\lambda_{1}(X_N(0))-x|\le\delta, \|G_{N}\|_{\mathrm{op}}\le N\})\\
 &\le 2  \mathbb P(\{ |\lambda_{1}( X_{N}(N))-{\ell}_{\nu,t}^{\Lambda}|\le \delta\}\cap\{  |\lambda_{1}(X_N(0))-x|\le\delta, \|G_{N}\|_{\mathrm{op}}\le N\}\cap\Omega_{N })+\mathbb P(\Omega_{N}^{c
 }),
 \end{split}
\end{equation*}
where
$${\Omega_{N}=\bigcap_{0\le p\le N^{2}-1}
 \Bigr\{\Bigr\|X_{N}\Bigr(\frac{p+1}{N}\Bigr)-X_{N}\Bigr(\frac{p}{N}\Bigr)\Bigr\|_{\mathrm{op}}\le\delta
 \Bigr\}.} $$
Since $X_{N}(\frac{p+1}{N})-X_{N}(\frac{p}{N})$ is a GOE matrix with variance $2t(1-e^{-1/(2tN)})\sim \frac1N$,  we find by \eqref{conc} that $\mathbb P(\Omega_{N}^{c})$ is bounded above by $2N^{2}e^{-\frac{1}{4t}N^{2}\delta^{2}}$. Moreover, since $y$ lies in between $x$ and ${\ell}_{\nu,t}^{\Lambda}$, we see that
\begin{multline*}
\{|\lambda_{1}(  X_{N}(N))-{\ell}_{\nu,t}^{\Lambda}|\le \delta\}\cap\{  |\lambda_{1}(X_N(0))-x|\le\delta, \|G_{N}\|_{\mathrm{op}}\le N\}\cap \Omega_{N}\subset {\bigcup_{0\le p\le  N^{2}}\Bigr\{\Bigr|\lambda_{1}\Bigl(X_{N}\Bigr(\frac{p}{N}\Bigr)\Bigr)-y\Bigr|\le\delta\Bigr\}.}
\end{multline*}
Hence, we deduce that
\begin{equation*}
\begin{split}
 \mathbb P( |\lambda_{1}(X_N(0))-x|\le\delta, \|G_{N}\|_{\mathrm{op}}\le N)&\leq 2 \mathbb P\Bigl( \bigcup_{0\le p\le N^{2}}
 \Bigl\{\Bigr|\lambda_{1}\Bigl(X_{N}\Bigl(\frac{p}{N}\Bigr)\Bigr)-y\Bigr|\le\delta\Bigr\}\Bigr)+4N^{2}e^{-c(\delta)N^{2}}
\\
&\leq 2N^{2 }\mathbb P(\{|\lambda_{1}(X_{N}(0))-y|\le\delta\})+4N^{2}e^{-\frac{1}{4t}N^{2}\delta^{2}},
\end{split}
\end{equation*}
where we used that $X_{N}(\frac{p}{N})$ has the same law as $X_{N}(0)$. 
Finally, since again by concentration of measure $\mathbb P( \|G_{N}\|_{\mathrm{op}}\ge N)\le e^{-cN^{2}}$,
we conclude that for some $c'(\delta)>0$ and all $N\in\mathbb N$
\begin{equation*}
{ \mathbb P( |\lambda_{1}(X_N(0))-x|\le\delta)\leq 2N^{2 }\mathbb P(\{|\lambda_{1}(X_{N}(0))-y|\le\delta\})+4N^{2}e^{-c'(\delta)N^{2}}.}
\end{equation*}
The result follows as long as $ \mathbb P( |\lambda_{1}(X_N(0))-x|\le\delta)$ is much larger than $e^{-cN^{2}}$ which is clear since we assumed that $\limsup\frac{1}{N}\log \mathbb P( |\lambda_{1}(X_N(0))-x|\le\delta)$ is finite. 

\end{proof}

\begin{lemma}\label{Fprop} The functions
$\lambda\mapsto  F_{\nu,t}^\pm(\Lambda,\lambda)$ are decreasing on $(-\infty,{\ell}^{\Lambda}_{\nu,t})$ and increasing on $({\ell}^{\Lambda}_{\nu,t},\ell_{\nu,t})$. Moreover, they vanish at ${\ell}^{\Lambda}_{\nu,t}$.
These functions are  bounded below by $c(\lambda-{\ell}^{\Lambda}_{\nu,t})^{2}$ {for some constant $c>0$ depending on $t$}.
\end{lemma}
\begin{proof}
The first point is a slight generalization of the previous proof where we take the supremum {over vectors} $\bM\in \mc{B}^{N}_{\delta}(\alpha)$.  Moreover, because $G_{N}\mapsto \lambda_{1}(B_{N}+G_{N})$ is Lipschitz with Lipschitz constant $1$, we deduce as in {\cite[Theorem 2.3.5]{AGZ}} that since $G_{N}$ has Gaussian entries with variance $t/N$, for every $\delta>0$,
$$\mathbb P(|\lambda_{1}(B_{N}+G_{N})-\mathbb E[\lambda_{1}(B_{N}+G_{N})]|\ge \ve)\le 2
e^{-\frac{\ve^{2}}{4t} N}\,.$$
{By Theorem \ref{convth}, we know that if $N_{i}/N$ goes to $\alpha_{i}$, 
$\mathbb E[\lambda_{1}(B_{N}+G_{N})]$ converges towards ${\ell}^{\Lambda}_{\nu,t}$.  Note that this limit is continuous in  $\nu$ provided $\Lambda<\ell_{\nu}$. As a consequence, for $\ve>0$ and  $\delta$ small enough so that 
$${\sup_{\|\alpha-\alpha'\|_{\infty}\le \delta}\Bigr|\ell^{\Lambda}_{\sum\alpha_{i}\delta_{\eta_{i}},t}-\ell^{\Lambda}_{\sum\alpha'_{i}\delta_{\eta_{i}},t}\Bigr|\le \ve},$$
we find that for $|x-\ell^{\Lambda}_{\nu,t}|\ge 2\ve$,
$$\sup_{\bM\in \mc{B}_{\delta}^{N}(\alpha)}\mathbb P( |\lambda_{1}(X_N)-x|\le\ve)\le \sup_{\bM\in \mc{B}_{\delta}^{N}(\alpha)}\mathbb P( |\lambda_{1}(X_N)-\E[\lambda_{1}(X_N)]|  \ge \tfrac{1}{2}|x- \ell^{\Lambda}_{\nu,t}|
)\le 2 e^{-\frac{|x- \ell^{\Lambda}_{\nu,t}|^{2}}{16t} N}.$$}
{Since $\Lambda\mapsto \ell^{\Lambda}_{\nu,t}$ is continuous, it follows by first taking the limit when $N$ goes to infinity and then the supremum over $\Lambda$ in a small ball, that}
$${F^{+}_{\nu,t}(\Lambda,\lambda)\ge \frac{|\lambda-\ell^{\Lambda}_{\nu,t}|^{2}}{16t}\,}.$$
The same result holds for $F^{-}_{\nu,t}$. Finally, $F^{\pm}_{\nu,t}$ has to vanish somewhere. Indeed, by exponential tightness \eqref{expt} we know that for some finite $M$,
 and $N$ large enough, we must have
 $$P(\lambda_{1}(X_N)\in [-M, M])\ge \frac{1}{2}\,.$$
 Moreover $M$ is chosen {independently} of the choice of {$\alpha\in\Delta^{p}$} governing the distribution of the eigenvalues of $B_{N}
$. Covering $[-M,M]$ {with} balls of width $\delta$ and using the previous monotonicity, we deduce that
 ${\inf F^{\pm}_{\nu,t}(\Lambda,\cdot)}$ equals {$0$}, {with the infimum achieved at ${\ell^{\Lambda}_{\nu,t}}$}.
\end{proof}    
\end{appendix}

\section*{Acknowledgments}
We would like to thank G\'erard Ben Arous for proposing the problem and for numerous useful and interesting discussions.  We are very grateful to the anonymous referees whose comments allowed to improve significantly a preliminary version of this article.

\section*{Funding}
This project was partly supported  by the ERC Project LDRAM : ERC-2019-ADG Project 884584

\bibliographystyle{plain}
\bibliography{main}

@article{AJBray,
doi = {10.1088/0022-3719/13/19/002},
url = {https://dx.doi.org/10.1088/0022-3719/13/19/002},
year = {1980},
month = {jul},
publisher = {},
volume = {13},
number = {19},
pages = {L469},
author = {A J Bray and  M A Moore},
title = {Metastable states in spin glasses},
journal = {Journal of Physics C: Solid State Physics}
}

@article {Hu22,
    AUTHOR = {Husson, Jonathan},
     TITLE = {Large deviations for the largest eigenvalue of matrices with
              variance profiles},
   JOURNAL = {Electron. J. Probab.},
  FJOURNAL = {Electronic Journal of Probability},
    VOLUME = {27},
      YEAR = {2022},
     PAGES = {Paper No. 74, 44},
      ISSN = {1083-6489},
   MRCLASS = {60B20 (60F10)},
  MRNUMBER = {4440063},
       DOI = {10.1214/22-ejp793},
       URL = {https://doi.org/10.1214/22-ejp793},
}

@article{biane,
 ISSN = {00222518, 19435258},
 URL = {http://www.jstor.org/stable/24899639},
 author = {Philippe Biane},
 journal = {Indiana University Mathematics Journal},
 number = {3},
 pages = {705--718},
 publisher = {Indiana University Mathematics Department},
 title = {On the Free Convolution with a Semi-circular Distribution},
 urldate = {2023-03-14},
 volume = {46},
 year = {1997}
}

@article{Thouless1977-THOSOS-3,
	author = {D. J. Thouless and P. W. Anderson and R. G. Palmer},
	doi = {10.1080/14786437708235992},
	journal = {Philosophical Magazine},
	number = {3},
	pages = {593--601},
	title = {Solution of 'Solvable Model of a Spin Glass'},
	volume = {35},
	year = {1977}
}

@article{plefka,
author = {Plefka, Timm},
year = {1999},
month = {01},
pages = {1971},
title = {Convergence condition of the TAP equation for the infinite-ranged Ising spin glass model},
volume = {15},
journal = {Journal of Physics A: Mathematical and General},
doi = {10.1088/0305-4470/15/6/035}
}

@article{parisi2004supersymmetry,
  title={On supersymmetry breaking in the computation of the complexity},
  author={Parisi, G and Rizzo, T},
  journal={Journal of Physics A: Mathematical and General},
  volume={37},
  number={33},
  pages={7979},
  year={2004},
  publisher={IOP Publishing}
}

@article{maida2007large,
 AUTHOR = {Maida, Myl\`ene},
     TITLE = {Large deviations for the largest eigenvalue of rank one
              deformations of {G}aussian ensembles},
   JOURNAL = {Electron. J. Probab.},
  FJOURNAL = {Electronic Journal of Probability},
    VOLUME = {12},
      YEAR = {2007},
     PAGES = {1131--1150},
      ISSN = {1083-6489},
   MRCLASS = {60F10 (15A52)},
  MRNUMBER = {2336602},
MRREVIEWER = {Peter\ Eichelsbacher},
       DOI = {10.1214/EJP.v12-438},
       URL = {https://doi.org/10.1214/EJP.v12-438},
}

@article{mckenna2021large,
AUTHOR = {McKenna, Benjamin},
     TITLE = {Large deviations for extreme eigenvalues of deformed {W}igner
              random matrices},
   JOURNAL = {Electron. J. Probab.},
  FJOURNAL = {Electronic Journal of Probability},
    VOLUME = {26},
      YEAR = {2021},
     PAGES = {Paper No. 34, 37},
      ISSN = {1083-6489},
   MRCLASS = {60B20 (60F10)},
  MRNUMBER = {4235485},
MRREVIEWER = {Ji\v r\'i\ \v Cern\'y},
       DOI = {10.1214/20-EJP571},
       URL = {https://doi.org/10.1214/20-EJP571},
}

@book {AGZ,
    AUTHOR = {Anderson, Greg W. and Guionnet, Alice and Zeitouni, Ofer},
     TITLE = {An introduction to random matrices},
    SERIES = {Cambridge Studies in Advanced Mathematics},
    VOLUME = {118},
 PUBLISHER = {Cambridge University Press},
   ADDRESS = {Cambridge},
      YEAR = {2010},
     PAGES = {xiv+492},
      ISBN = {978-0-521-19452-5},
   MRCLASS = {60B20 (46L53 46L54)},
  MRNUMBER = {2760897},
}

@article {HuGu2,
    AUTHOR = {Guionnet, Alice and Husson, Jonathan},
     TITLE = {Asymptotics of {$k$} dimensional spherical integrals and
              applications},
   JOURNAL = {ALEA Lat. Am. J. Probab. Math. Stat.},
  FJOURNAL = {ALEA. Latin American Journal of Probability and Mathematical
              Statistics},
    VOLUME = {19},
      YEAR = {2022},
    NUMBER = {1},
     PAGES = {769--797},
      ISSN = {1980-0436},
   MRCLASS = {60B20 (60F10)},
  MRNUMBER = {4436026},
MRREVIEWER = {Zahir\ Mouhoubi},
       DOI = {10.30757/alea.v19-30},
       URL = {https://doi.org/10.30757/alea.v19-30},
}

@article {GuMa,
    AUTHOR = {Guionnet, A. and Ma\"{\i}da, M.},
     TITLE = {A {F}ourier view on the {$R$}-transform and related
              asymptotics of spherical integrals},
   JOURNAL = {J. Funct. Anal.},
  FJOURNAL = {Journal of Functional Analysis},
    VOLUME = {222},
      YEAR = {2005},
    NUMBER = {2},
     PAGES = {435--490},
      ISSN = {0022-1236,1096-0783},
   MRCLASS = {60F10 (46L54)},
  MRNUMBER = {2132396},
MRREVIEWER = {Sompong\ Dhompongsa},
       DOI = {10.1016/j.jfa.2004.09.015},
       URL = {https://doi.org/10.1016/j.jfa.2004.09.015},
}

@book{VDN92,
	key={VDN92},
    AUTHOR = {D. V. 
	    Voiculescu and K. J. Dykema and A. Nica},
     TITLE = {Free random variables},
    SERIES = {CRM Monograph Series},
    VOLUME = {1},
 PUBLISHER = {American Mathematical Society},
   ADDRESS = {Providence, RI},
      YEAR = {1992}
}

@ARTICLE{BV93,
    KEY = {BV93},
    AUTHOR = "H. Bercovici and D. Voiculescu",
    TITLE = "Free convolution of measures with unbounded support",
    JOURNAL = "Indiana U. Math. J.",
    YEAR = 1993,
    VOLUME = 42,
    PAGES = "733--773"                                    
    }

@Article{CaDoFeFe,
 Author = {Capitaine, Mireille and Donati-Martin, Catherine and F{\'e}ral, Delphine and F{\'e}vrier, Maxime},
 Title = {Free convolution with a semicircular distribution and eigenvalues of spiked deformations of {Wigner} matrices},
 FJournal = {Electronic Journal of Probability},
 Journal = {Electron. J. Probab.},
 ISSN = {1083-6489},
 Volume = {16},
 Pages = {1750--1792},
 Note = {Id/No 64},
 Year = {2011},
 Language = {English},
 DOI = {10.1214/EJP.v16-934},
 zbMATH = {6049120},
 Zbl = {1245.15037}
}

@article{benaych2012large,
  title={Large deviations of the extreme eigenvalues of random deformations of matrices},
  author={Benaych-Georges, Florent and Guionnet, Alice and Ma{\"\i}da, Myl{\`e}ne},
  journal={Probability Theory and Related Fields},
  volume={154},
  pages={703--751},
  year={2012},
  publisher={Springer}
}

@article{HuGu1,
    TITLE = {Large deviations for the largest eigenvalue of {R}ademacher
              matrices},
              AUTHOR={Guionnet, Alice and Husson, Jonathan},
   JOURNAL = {Ann. Probab.},
  FJOURNAL = {The Annals of Probability},
    VOLUME = {48},
      YEAR = {2020},
    NUMBER = {3},
     PAGES = {1436--1465},
      ISSN = {0091-1798},
   MRCLASS = {60B20 (60F10)},
  MRNUMBER = {4112720},
       DOI = {10.1214/19-AOP1398},
       URL = {https://doi.org/10.1214/19-AOP1398},
}

@article{guionnet2020largeb,
 AUTHOR = {Guionnet, Alice and Ma\"ida, Myl\`ene},
     TITLE = {Large deviations for the largest eigenvalue of the sum of two
              random matrices},
   JOURNAL = {Electron. J. Probab.},
  FJOURNAL = {Electronic Journal of Probability},
    VOLUME = {25},
      YEAR = {2020},
     PAGES = {Paper No. 14, 24},
      ISSN = {1083-6489},
   MRCLASS = {60B20 (46L54 60F10)},
  MRNUMBER = {4073675},
       DOI = {10.1214/19-ejp405},
       URL = {https://doi.org/10.1214/19-ejp405},
}

@article{biroli2020large,
AUTHOR = {Biroli, Giulio and Guionnet, Alice},
     TITLE = {Large deviations for the largest eigenvalues and eigenvectors
              of spiked {G}aussian random matrices},
   JOURNAL = {Electron. Commun. Probab.},
  FJOURNAL = {Electronic Communications in Probability},
    VOLUME = {25},
      YEAR = {2020},
     PAGES = {Paper No. 70, 13},
      ISSN = {1083-589X},
   MRCLASS = {60B20 (60F10)},
  MRNUMBER = {4158230},
MRREVIEWER = {Deli\ Li},
       DOI = {10.3390/mca25010013},
       URL = {https://doi.org/10.3390/mca25010013},
}

@article{augeri2021large,
  title={Large deviations for the largest eigenvalue of sub-Gaussian matrices},
  author={Augeri, Fanny and Guionnet, Alice and Husson, Jonathan},
  journal={Communications in mathematical physics},
  volume={383},
  pages={997--1050},
  year={2021},
  publisher={Springer}
}

@article{augeri2016large,
AUTHOR = {Augeri, Fanny},
     TITLE = {Large deviations principle for the largest eigenvalue of
              {W}igner matrices without {G}aussian tails},
   JOURNAL = {Electron. J. Probab.},
  FJOURNAL = {Electronic Journal of Probability},
    VOLUME = {21},
      YEAR = {2016},
     PAGES = {Paper No. 32, 49},
      ISSN = {1083-6489},
   MRCLASS = {60B20 (60F10)},
  MRNUMBER = {3492936},
MRREVIEWER = {Ofer\ Zeitouni},
       DOI = {10.1214/16-EJP4146},
       URL = {https://doi.org/10.1214/16-EJP4146},
}

@article{cook2023full,
  title={Full large deviation principles for the largest eigenvalue of sub-Gaussian Wigner matrices},
  author={Cook, Nicholas A and Ducatez, Raphael and Guionnet, Alice},
  journal={arXiv preprint arXiv:2302.14823},
  year={2023}
}

@article{ducatez2024large,
  title={Large deviation principle for the largest eigenvalue of random matrices with a variance profile},
  author={Ducatez, Rapha{\"e}l and Husson, Jonathan and Guionnet, Alice},
  journal={arXiv preprint arXiv:2403.05413},
  year={2024}
}

@article{wigner,
  author  = {Wigner, Eugene P.},
  title   = {Characteristic vectors of bordered matrices with infinite dimensions},
  journal = {Ann. of Math. (2)},
  volume  = {62},
  year    = {1955},
  pages   = {548--564},
}

@article{BaErSc,
  author  = {Bao, Zhigang and Erd{\H{o}}s, L{\'a}szl{\'o} and Schnelli, Kevin},
  title   = {Local law of addition of random matrices on optimal scale},
  journal = {Comm. Math. Phys.},
  volume  = {349},
  number  = {3},
  year    = {2017},
  pages   = {947--990},
}

@article{GZ00,
  author  = {Guionnet, Alice and Zeitouni, Ofer},
  title   = {Concentration of the spectral measure for large matrices},
  journal = {Electron. Comm. Probab.},
  volume  = {5},
  year    = {2000},
  pages   = {119--136},
}

@book{DZ,
  author    = {Dembo, Amir and Zeitouni, Ofer},
  title     = {Large deviations techniques and applications},
  edition   = {Second},
  series    = {Stochastic Modelling and Applied Probability},
  volume    = {38},
  publisher = {Springer-Verlag, Berlin},
  year      = {2010},
  note      = {Corrected reprint of the 1998 edition},
}

@article{BeliusCG,
  author  = {Belius, David and Cook, Nicholas A. and Guionnet, Alice},
  title   = {Large deviations for the largest eigenvalue of the sum of two random matrices and the Hopf--Lax formula},
  journal = {Preprint},
  year    = {2023},
}

\end{document}